\documentclass[11pt]{preprint}
\usepackage[full]{textcomp}
\usepackage[osf]{newtxtext} 
\usepackage[cal=boondoxo]{mathalfa}
\usepackage{colortbl}

\usepackage{comment}

\usepackage{amssymb}
\usepackage{mathtools}
\usepackage{hyperref}
\usepackage{breakurl}
\usepackage{mhenvs}
\usepackage{mhequ} 
\usepackage{mhsymb}
\usepackage{booktabs}
\usepackage{tikz}
\usepackage{tcolorbox}
\usepackage{mathrsfs}
\usepackage[utf8]{inputenc}
\usepackage{longtable}
\usepackage{wrapfig}
\usepackage{subcaption}
\usepackage{mathrsfs}
\usepackage{epsfig}
\usepackage{microtype}
\usepackage{comment}
\usepackage{wasysym}
\usepackage{centernot}
\usepackage{enumitem}
\usepackage{bm}
\usepackage{stackrel}
\usepackage{graphicx}

\makeatletter
\newcommand{\globalcolor}[1]{%
  \color{#1}\global\let\default@color\current@color
}
\makeatother

\usetikzlibrary{calc}
\usetikzlibrary{decorations}
\usetikzlibrary{positioning}
\usetikzlibrary{shapes}
\usetikzlibrary{external}

\definecolor{blush}{rgb}{0.87, 0.36, 0.51}
	\definecolor{brightcerulean}{rgb}{0.11, 0.67, 0.84}
	\definecolor{greenryb}{rgb}{0.4, 0.69, 0.2}

\newif\ifdark
\darkfalse

\ifdark
\definecolor{darkred}{rgb}{0.9,0.2,0.2}
\definecolor{darkblue}{rgb}{0.7,0.3,1}
\definecolor{darkgreen}{rgb}{0.1,0.9,0.1}
\definecolor{franck}{rgb}{0,0.8,1}
\definecolor{pagebackground}{rgb}{.15,.21,.18}
\definecolor{pageforeground}{rgb}{.84,.84,.85}
\pagecolor{pagebackground}
\AtBeginDocument{\globalcolor{pageforeground}}
\definecolor{symbols}{rgb}{0,0.7,1}
\colorlet{connection}{red!80!black}
\colorlet{boxcolor}{blue!50}

\else

\definecolor{darkred}{rgb}{0.7,0.1,0.1}
\definecolor{darkblue}{rgb}{0.4,0.1,0.8}
\definecolor{darkgreen}{rgb}{0.1,0.7,0.1}
\definecolor{franck}{rgb}{0,0,1}
\definecolor{pagebackground}{rgb}{1,1,1}
\definecolor{pageforeground}{rgb}{0,0,0}
\colorlet{symbols}{blue!90!black}
\colorlet{connection}{red!30!black}
\colorlet{boxcolor}{blue!50!black}

\fi

\def\slash{\leavevmode\unskip\kern0.18em/\penalty\exhyphenpenalty\kern0.18em}
\def\dash{\leavevmode\unskip\kern0.18em--\penalty\exhyphenpenalty\kern0.18em}

\DeclareMathAlphabet{\mathbbm}{U}{bbm}{m}{n}

\DeclareFontFamily{U}{BOONDOX-calo}{\skewchar\font=45 }
\DeclareFontShape{U}{BOONDOX-calo}{m}{n}{
  <-> s*[1.05] BOONDOX-r-calo}{}
\DeclareFontShape{U}{BOONDOX-calo}{b}{n}{
  <-> s*[1.05] BOONDOX-b-calo}{}
\DeclareMathAlphabet{\mcb}{U}{BOONDOX-calo}{m}{n}
\SetMathAlphabet{\mcb}{bold}{U}{BOONDOX-calo}{b}{n}

\setlist{noitemsep,topsep=4pt,leftmargin=1.5em}

\DeclareMathAlphabet{\mathbbm}{U}{bbm}{m}{n}

\DeclareMathAlphabet{\mcb}{U}{BOONDOX-calo}{m}{n}
\SetMathAlphabet{\mcb}{bold}{U}{BOONDOX-calo}{b}{n}
\DeclareFontFamily{U}{mathx}{\hyphenchar\font45}
\DeclareFontShape{U}{mathx}{m}{n}{
      <5> <6> <7> <8> <9> <10>
      <10.95> <12> <14.4> <17.28> <20.74> <24.88>
      mathx10
      }{}
\DeclareSymbolFont{mathx}{U}{mathx}{m}{n}
\DeclareMathSymbol{\bigtimes}{1}{mathx}{"91}

\def\emptyset{{\centernot\ocircle}}

\providecommand{\figures}{false}
{ \ifthenelse{\equal{\figures}{false}} {#1}{\[ {\rm Figure \ missing !} \]} }{}
\def\id{\mathrm{id}}

\def\CH{\mathcal{H}}
\def\CP{\mathcal{P}}

\def\CA{\mathcal{A}}

\def\CC{\mathcal{C}}

\def\CB{\mathcal{B}}
\def\CM{\mathcal{M}}
\def\CT{\mathcal{T}}

\tikzstyle{tinydots}=[dash pattern=on \pgflinewidth off \pgflinewidth]
\tikzstyle{superdense}=[dash pattern=on 4pt off 1pt]

\newcommand{\beq}{\begin{equation}}
\newcommand{\eeq}{\end{equation}}

\usepackage{empheq}

\newcommand{\T}{\mathbf{T}}

\def\Labe{\mathfrak{e}}

\def\Labn{\mathfrak{n}}

\def\Lab{\mathfrak{L}}

\def\${|\!|\!|}

\newcounter{theorem}

\newtheorem{ex}[theorem]{Example}

\newenvironment{DIFnomarkup}{}{} 

\newtheorem{assumption}{Assumption}
 
\theorembodyfont{\rmfamily}

\newfont{\indic}{bbmss12}

\def\Nabla_#1{\nabla_{\!#1}}

\makeatletter
\pgfdeclareshape{crosscircle}
{
  \inheritsavedanchors[from=circle] 
  \inheritanchorborder[from=circle]
  \inheritanchor[from=circle]{north}
  \inheritanchor[from=circle]{north west}
  \inheritanchor[from=circle]{north east}
  \inheritanchor[from=circle]{center}
  \inheritanchor[from=circle]{west}
  \inheritanchor[from=circle]{east}
  \inheritanchor[from=circle]{mid}
  \inheritanchor[from=circle]{mid west}
  \inheritanchor[from=circle]{mid east}
  \inheritanchor[from=circle]{base}
  \inheritanchor[from=circle]{base west}
  \inheritanchor[from=circle]{base east}
  \inheritanchor[from=circle]{south}
  \inheritanchor[from=circle]{south west}
  \inheritanchor[from=circle]{south east}
  \inheritbackgroundpath[from=circle]
  \foregroundpath{
    \centerpoint%
    \pgf@xc=\pgf@x%
    \pgf@yc=\pgf@y%
    \pgfutil@tempdima=\radius%
    \pgfmathsetlength{\pgf@xb}{\pgfkeysvalueof{/pgf/outer xsep}}%
    \pgfmathsetlength{\pgf@yb}{\pgfkeysvalueof{/pgf/outer ysep}}%
    \ifdim\pgf@xb<\pgf@yb%
      \advance\pgfutil@tempdima by-\pgf@yb%
    \else%
      \advance\pgfutil@tempdima by-\pgf@xb%
    \fi%
    \pgfpathmoveto{\pgfpointadd{\pgfqpoint{\pgf@xc}{\pgf@yc}}{\pgfqpoint{-0.707107\pgfutil@tempdima}{-0.707107\pgfutil@tempdima}}}
    \pgfpathlineto{\pgfpointadd{\pgfqpoint{\pgf@xc}{\pgf@yc}}{\pgfqpoint{0.707107\pgfutil@tempdima}{0.707107\pgfutil@tempdima}}}
    \pgfpathmoveto{\pgfpointadd{\pgfqpoint{\pgf@xc}{\pgf@yc}}{\pgfqpoint{-0.707107\pgfutil@tempdima}{0.707107\pgfutil@tempdima}}}
    \pgfpathlineto{\pgfpointadd{\pgfqpoint{\pgf@xc}{\pgf@yc}}{\pgfqpoint{0.707107\pgfutil@tempdima}{-0.707107\pgfutil@tempdima}}}
  }
}
\makeatother

\def\symbol#1{\textcolor{symbols}{#1}}

\def\decorate#1#2{
        \ifnum#2>0
    		\foreach \count in {1,...,#2}{
	       	let
				\p1 = (sourcenode.center),
                \p2 = (sourcenode.east),
				\n1 = {\x2-\x1},
				\n2 = {1mm},
				\n3 = {(1.3+0.6*(\count-1))*\n1},
				\n4 = {0.7*\n1}
			in 
        		node[rectangle,fill=symbols,rotate=30,inner sep=0pt,minimum width=0.2*\n2,minimum height=\n2] at ($(sourcenode.center) + (\n3,\n4)$) {}
				}
		\fi
        \ifnum#1>0
    		\foreach \count in {1,...,#1}{
	       	let
				\p1 = (sourcenode.center),
                \p2 = (sourcenode.east),
				\n1 = {\x2-\x1},
				\n2 = {1mm},
				\n3 = {(1.3+0.6*(\count-1))*\n1},
				\n4 = {0.7*\n1}
			in 
        		node[rectangle,fill=symbols,rotate=-30,inner sep=0pt,minimum width=0.2*\n2,minimum height=\n2] at ($(sourcenode.center) + (-\n3,\n4)$) {}
				}
		\fi
}

\tikzset{
    dectriangle/.style 2 args={
        triangle,
        alias=sourcenode,
        append after command={\decorate{#1}{#2}}
    },
    dectriangle/.default={0}{0},
}

\tikzset{
	cross/.style={path picture={ 
  		\draw[symbols]
			(path picture bounding box.south east) -- (path picture bounding box.north west) (path picture bounding box.south west) -- (path picture bounding box.north east);
		}},
root/.style={circle,fill=green!50!black,inner sep=0pt, minimum size=1.2mm},
        dot/.style={circle,fill=pageforeground,inner sep=0pt, minimum size=1mm},
        dotred/.style={circle,fill=pageforeground!50!pagebackground,inner sep=0pt, minimum size=2mm},
        var/.style={circle,fill=pageforeground!10!pagebackground,draw=pageforeground,inner sep=0pt, minimum size=3mm},
        kernel/.style={semithick,shorten >=2pt,shorten <=2pt},
        kernels/.style={snake=zigzag,shorten >=2pt,shorten <=2pt,segment amplitude=1pt,segment length=4pt,line before snake=2pt,line after snake=5pt,},
        rho/.style={densely dashed,semithick,shorten >=2pt,shorten <=2pt},
           testfcn/.style={dotted,semithick,shorten >=2pt,shorten <=2pt},
        renorm/.style={shape=circle,fill=pagebackground,inner sep=1pt},
        labl/.style={shape=rectangle,fill=pagebackground,inner sep=1pt},
        xic/.style={very thin,circle,draw=symbols,fill=symbols,inner sep=0pt,minimum size=1.2mm},
        g/.style={very thin,rectangle,draw=symbols,fill=symbols!10!pagebackground,inner sep=0pt,minimum width=2.5mm,minimum height=1.2mm},
        xi/.style={very thin,circle,draw=symbols,fill=symbols!10!pagebackground,inner sep=0pt,minimum size=1.2mm},
	xies/.style={very thin,rectangle,fill=green!50!black!25,draw=symbols,inner sep=0pt,minimum size=1.1mm},
	xiesf/.style={very thin,rectangle,fill=green!50!black,draw=symbols,inner sep=0pt,minimum size=1.1mm},
        xix/.style={very thin,crosscircle,fill=symbols!10!pagebackground,draw=symbols,inner sep=0pt,minimum size=1.2mm},
        X/.style={very thin,cross,rectangle,fill=pagebackground,draw=symbols,inner sep=0pt,minimum size=1.2mm},
	xib/.style={thin,circle,fill=symbols!10!pagebackground,draw=symbols,inner sep=0pt,minimum size=1.6mm},
	xie/.style={thin,circle,fill=green!50!black,draw=symbols,inner sep=0pt,minimum size=1.6mm},
	xid/.style={thin,circle,fill=symbols,draw=symbols,inner sep=0pt,minimum size=1.6mm},
	xibx/.style={thin,crosscircle,fill=symbols!10!pagebackground,draw=symbols,inner sep=0pt,minimum size=1.6mm},
	kernels2/.style={very thick,draw=connection,segment length=12pt},
	keps/.style={thin,draw=symbols,->},
	kepspr/.style={thick,draw=connection,->},
	krho/.style={thin,draw=symbols,superdense,->},
	krhopr/.style={thick,draw=connection,superdense},
	triangle/.style = { regular polygon, regular polygon sides=3},
	not/.style={thin,circle,draw=connection,fill=connection,inner sep=0pt,minimum size=0.5mm},
	diff/.style = {very thin,draw=symbols,triangle,fill=red!50!black,inner sep=0pt,minimum size=1.6mm},
	diff1/.style = {very thin,dectriangle={1}{0},fill=red!50!black,draw=symbols,inner sep=0pt,minimum size=1.6mm},
	diff2/.style = {very thin,dectriangle={1}{1},fill=red!50!black,draw=symbols,inner sep=0pt,minimum size=1.6mm},
		diffmini/.style = {very thin,rectangle,fill=black,draw=black,inner sep=0pt,minimum size=0.75mm},
	 kernelsmod/.style={very thick,draw=connection,segment length=12pt},
	 rec/.style = {very thin,rectangle,fill=black,draw=black,inner sep=0pt,minimum size=2mm},
	cerc/.style={very thin,circle,draw=black,fill=symbols,inner sep=0pt,minimum size=2mm},
	stars/.style={very thin,star,star points=6,star point ratio=0.5, draw=black,fill=red,inner sep=0pt,minimum size=0.7mm},
	>=stealth,
        }
        \tikzset{
root/.style={circle,fill=black!50,inner sep=0pt, minimum size=3mm},
        circ/.style={circle,fill=white,draw=black,very thin,inner sep=.5pt, minimum size=1.2mm},
        round1/.style={fill=white,outer sep = 0,inner sep=2pt,rounded corners=1mm,draw,text=black,thin,minimum size=1.2mm},
          circ1/.style={circle,fill=red!10,draw=red,very thin,inner sep=.5pt, minimum size=1.2mm},
        rect/.style={fill=white,outer sep = 0,inner sep=2pt,rectangle,draw,text=black,thin,minimum size=1.2mm},
        rect1/.style={fill=white,outer sep = 0,inner sep=2pt,rectangle,draw,text=black,thin,minimum size=1.2mm},
        round2/.style={fill=red!10,outer sep = 0,inner sep=2pt,rounded corners=1mm,draw,text=black,thin,minimum size=1.2mm},
       round3/.style={fill=blue!10,outer sep = 0,inner sep=2pt,rounded corners=1mm,draw,text=black,thin,minimum size=1.2mm}, 
        rect2/.style={fill=black!10,outer sep = 0,inner sep=2pt,rectangle,draw,text=black,thin,minimum size=1.2mm},
        dot/.style={circle,fill=black,inner sep=0pt, minimum size=1.2mm},
        dotred/.style={circle,fill=black!50,inner sep=0pt, minimum size=2mm},
        var/.style={circle,fill=black!10,draw=black,inner sep=0pt, minimum size=3mm},
        kernel/.style={semithick,shorten >=2pt,shorten <=2pt},
         diag/.style={thin,shorten >=4pt,shorten <=4pt},
        kernel1/.style={thick},
        kernels/.style={snake=zigzag,shorten >=2pt,shorten <=2pt,segment amplitude=1pt,segment length=4pt,line before snake=2pt,line after snake=5pt,},
		kernels1/.style={snake=zigzag,segment amplitude=0.5pt,segment length=2pt},
		rho1/.style={densely dotted,semithick},
        rho/.style={densely dashed,semithick,shorten >=2pt,shorten <=2pt},
           testfcn/.style={dotted,semithick,shorten >=2pt,shorten <=2pt},
           visible/.style={draw, circle, fill, inner sep=0.25ex},
        renorm/.style={shape=circle,fill=white,inner sep=1pt},
        labl/.style={shape=rectangle,fill=white,inner sep=1pt},
        xic/.style={very thin,circle,fill=symbols,draw=black,inner sep=0pt,minimum size=1.2mm},
        xi/.style={very thin,circle,fill=blue!10,draw=black,inner sep=0pt,minimum size=1.2mm},
	xib/.style={very thin,circle,fill=blue!10,draw=black,inner sep=0pt,minimum size=1.6mm},
	xie/.style={very thin,circle,fill=green!50!black,draw=black,inner sep=0pt,minimum size=1mm},
	xid/.style={very thin,circle,fill=symbols,draw=black,inner sep=0pt,minimum size=1.6mm},
	edgetype/.style={very thin,circle,draw=black,inner sep=0pt,minimum size=5mm},
	nodetype/.style={very thick,circle,draw=black,inner sep=0pt,minimum size=5mm},
	kernels2/.style={very thick,draw=connection,segment length=12pt},
clean/.style={thin,circle,fill=black,inner sep=0pt,minimum size=1mm},	not/.style={thin,circle,fill=symbols,draw=connection,fill=connection,inner sep=0pt,minimum size=0.8mm},
	>=stealth,
        }

\makeatletter
\def\DeclareSymbol#1#2#3{%
	\expandafter\gdef\csname MH@symb@#1\endcsname{\tikzsetnextfilename{symbol#1}%
	\tikz[baseline=#2,scale=0.15,draw=symbols,line join=round]{#3}}%
	\expandafter\gdef\csname MH@symb@#1s\endcsname{\scalebox{0.75}{\tikzsetnextfilename{symbol#1}%
	\tikz[baseline=#2,scale=0.15,draw=symbols,line join=round]{#3}}}%
	\expandafter\gdef\csname MH@symb@#1ss\endcsname{\scalebox{0.65}{\tikzsetnextfilename{symbol#1}%
	\tikz[baseline=#2,scale=0.15,draw=symbols,line join=round]{#3}}}%
	}
\def\<#1>{\ifthenelse{\boolean{mmode}}{\mathchoice{\csname MH@symb@#1\endcsname}{\csname MH@symb@#1\endcsname}{\csname MH@symb@#1s\endcsname}{\csname MH@symb@#1ss\endcsname}}{\csname MH@symb@#1\endcsname}}
\makeatother

\DeclareSymbol{Xi22}{0.5}{\draw (0,0) node[xi] {} -- (-1,1) node[not] {} -- (0,2) node[xi] {};} 
\DeclareSymbol{Xi2}{-2}{\draw (-1,-0.25) node[xi] {} -- (0,1) node[xi] {};} 
\DeclareSymbol{Xi2b}{-2}{\draw (-1,-0.25) node[xic] {} -- (0,1) node[xic] {};} 
\DeclareSymbol{Xi2g}{-2}{\draw (-1,-0.25) node[xies] {} -- (0,1) node[xi] {};} 
\DeclareSymbol{Xi2g2}{-2}{\draw (-1,-0.25) node[xi] {} -- (0,1) node[xies] {};} 
\DeclareSymbol{cXi2}{-2}{\draw (0,-0.25) node[xi] {} -- (-1,1) node[xic] {};}
\DeclareSymbol{Xi3}{0}{\draw (0,0) node[xi] {} -- (-1,1) node[xi] {} -- (0,2) node[xi] {};}
\DeclareSymbol{XiIIXi}{0}{\draw (0,0) node[xi] {} -- (-1,1); \draw[kernels2] (-1,1) node[not] {} -- (0,2) node[xi] {};}

\DeclareSymbol{Xi4}{2}{\draw (0,0) node[xi] {} -- (-1,1) node[xi] {} -- (0,2) node[xi] {} -- (-1,3) node[xi] {};}
\DeclareSymbol{Xi4_1}{2}{\draw (0,0) node[xic] {} -- (-1,1) node[xic] {} -- (0,2) node[xi] {} -- (-1,3) node[xi] {};}
\DeclareSymbol{Xi4_2}{2}{\draw (0,0) node[xic] {} -- (-1,1) node[xi] {} -- (0,2) node[xi] {} -- (-1,3) node[xic] {};}
\DeclareSymbol{Xi2X}{-2}{\draw (0,-0.25) node[xi] {} -- (-1,1) node[xix] {};}
\DeclareSymbol{XXi2}{-2}{\draw (0,-0.25) node[xix] {} -- (-1,1) node[xi] {};}
\DeclareSymbol{IIXi}{0}{\draw (0,-0.25) node[not] {} -- (-1,1) node[xi] {} -- (0,2) node[xi] {};}
\DeclareSymbol{IXi^2}{-1}{\draw (-1,1) node[xi] {} -- (0,0) node[not] {} -- (1,1) node[xi] {};}
\DeclareSymbol{IIXi^2}{-4}{\draw (0,-1.5) node[not] {} -- (0,0);
\draw[kernels2] (-1,1) node[xi] {} -- (0,0) node[not] {} -- (1,1) node[xi] {};}
\DeclareSymbol{XiX}{-2.8}{\node[xibx] {};}
\DeclareSymbol{tauX}{-2.8}{ \node[X] {};}
\DeclareSymbol{Xi}{-2.8}{\node[xib] {};}

\DeclareSymbol{IXiX}{-1}{\draw (0,-0.25) node[not] {} -- (-1,1) node[xix] {};}
\DeclareSymbol{IXi3}{2}{\draw (0,-0.25) node[not] {} -- (-1,1) node[xi] {} -- (0,2) node[xi] {} -- (-1,3) node[xi] {};}
\DeclareSymbol{IXi}{-2}{\draw (0,-0.25) node[not] {} -- (-1,1) node[xi] {};}
\DeclareSymbol{XiI}{-2}{\draw (0,-0.25) node[xi] {} -- (-1,1) node[not] {};}

\DeclareSymbol{Xi4b}{0}{\draw(0,1.5) node[xi] {} -- (0,0); \draw (-1,1) node[xi] {} -- (0,0) node[xi] {} -- (1,1) node[xi] {};}
\DeclareSymbol{Xi4b'}{0}{\draw(0,1.5) node[xi] {} -- (0,-0.2); \draw (-1,1) node[xi] {} -- (0,-0.2) node[not] {} -- (1,1) node[xi] {};}
\DeclareSymbol{Xi4c}{0}{\draw (0,1) -- (0.8,2.2) node[xi] {};\draw (0,-0.25) node[xi] {} -- (0,1) node[xi] {} -- (-0.8,2.2) node[xi] {};}
\DeclareSymbol{Xi4d}{-4.5}{\draw (0,-1.5) node[not] {} -- (0,0); \draw (-1,1) node[xi] {} -- (0,0) node[xi] {} -- (1,1) node[xi] {};}
\DeclareSymbol{Xi4e}{0}{\draw (0,2) node[xi] {} -- (-1,1) node[xi] {} -- (0,0) node[xi] {} -- (1,1) node[xi] {};}
\DeclareSymbol{Xi4e'}{0}{\draw (0,2) node[xi] {} -- (-1,1) node[xi] {} -- (0,-0.2) node[not] {} -- (1,1) node[xi] {};}

\DeclareSymbol{Xitwo}
{0}{\draw[kernels2] (0,0) node[not] {} -- (-1,1) node[not] {}
-- (-2,2) node[not]{} -- (-3,3) node[xi]  {};
\draw[kernels2] (0,0) -- (1,1) node[xi] {};
\draw[kernels2] (-1,1) -- (0,2) node[xi] {};
\draw[kernels2] (-2,2) -- (-1,3) node[xi] {};}

\DeclareSymbol{IXitwo}
{0}{\draw (-.7,1.2) node[xi] {} -- (0,-0.2) -- (.7,1.2) node[xi] {};}
\DeclareSymbol{I1Xitwo}
{0}{\draw[kernels2] (0,0) node[not] {} -- (-1,1) node[xi] {};
\draw[kernels2] (0,0) -- (1,1) node[xi] {};}

\DeclareSymbol{I1Xitwobis}
{0}{\draw[kernels2] (0,0) node[not] {} -- (-1,1) node[xies] {};
\draw[kernels2] (0,0) -- (1,1) node[xies] {};}

\DeclareSymbol{I1Xitwog}
{0}{\draw[kernels2] (0,0) node[not] {} -- (-1,1) node[xies] {};
\draw[kernels2] (0,0) -- (1,1) node[xi] {};}

\DeclareSymbol{cI1Xitwo}
{0}{\draw[kernels2] (0,0) node[not] {} -- (-1,1) node[xic] {};
\draw[kernels2] (0,0) -- (1,1) node[xi] {};}

\DeclareSymbol{I1IXi3}{0}{\draw (0,0) node[xi] {} -- (-1,1) ; 
\draw[kernels2] (-1,1) node[not] {} -- (0,2) node[xi] {};
\draw[kernels2] (-1,1) node[not] {} -- (-2,2) node[xi] {};}

\DeclareSymbol{I1Xi3c}{-1}{\draw[kernels2](0,1.5) node[xi] {} -- (0,0) node[not] {}; \draw (-1,1) node[xi] {} -- (0,0) ; \draw[kernels2] (0,0) -- (1,1) node[xi] {};}

\DeclareSymbol{I1Xi3cbis}{-1}{\draw[kernels2](0,1.5) node[xies] {} -- (0,0) node[not] {}; \draw (-1,1) node[xies] {} -- (0,0) ; \draw[kernels2] (0,0) -- (1,1) node[xies] {};}

\DeclareSymbol{I1IXi3b}{0}{\draw[kernels2] (0,0) node[not] {} -- (-1,1) ; \draw[kernels2] (0,0)   -- (1,1) node[xi] {} ;
\draw (-1,1) node[xi] {} -- (0,2) node[xi] {};
}

\DeclareSymbol{I1IXi3c}{0}{\draw[kernels2] (0,0) node[not] {} -- (-1,1) ; \draw[kernels2] (0,0)   -- (1,1) node[xi] {} ;
\draw[kernels2] (-1,1) node[not] {} -- (0,2) node[xi] {};
\draw[kernels2] (-1,1) node[not] {} -- (-2,2) node[xi] {};}

\DeclareSymbol{I1IXi3cbis}{0}{\draw[kernels2] (0,0) node[not] {} -- (-1,1) ; \draw[kernels2] (0,0)   -- (1,1) node[xies] {} ;
\draw[kernels2] (-1,1) node[not] {} -- (0,2) node[xies] {};
\draw[kernels2] (-1,1) node[not] {} -- (-2,2) node[xies] {};}

\DeclareSymbol{I1Xi}{0}{\draw[kernels2] (0,0) node[not] {} -- (-1,1)  node[xi] {} ;}

\DeclareSymbol{I1Xi4a}{2}{\draw[kernels2] (0,0) node[not] {} -- (-1,1) ; \draw[kernels2] (0,0) node[not] {} -- (1,1) node[xi] {} ;
\draw (-1,1) node[xi] {} -- (0,2) node[xi] {} -- (-1,3) node[xi] {};}

\DeclareSymbol{cI1Xi4a}{2}{\draw[kernels2] (0,0) node[not] {} -- (-1,1) ; \draw[kernels2] (0,0) node[not] {} -- (1,1) node[xic] {} ;
\draw (-1,1) node[xic] {} -- (0,2) node[xi] {} -- (-1,3) node[xi] {};}

\DeclareSymbol{I1Xi4b}{2}{\draw (0,0) node[xi] {} -- (-1,1) node[xi] {} -- (0,2) ; \draw[kernels2] (0,2) node[not] {} -- (-1,3) node[xi] {};\draw[kernels2] (0,2)  -- (1,3) node[xi] {};
}

\DeclareSymbol{cI1Xi4b}{2}{\draw (0,0) node[xic] {} -- (-1,1) node[xic] {} -- (0,2) ; \draw[kernels2] (0,2) node[not] {} -- (-1,3) node[xi] {};\draw[kernels2] (0,2)  -- (1,3) node[xi] {};
}

\DeclareSymbol{I1Xi4c}{2}{\draw (0,0) node[xi] {} -- (-1,1) node[not] {}; \draw[kernels2] (-1,1) -- (0,2) ; 
\draw[kernels2] (-1,1) -- (-2,2) node[xi] {} ;
\draw (0,2) node[xi] {} -- (-1,3) node[xi] {};}

\DeclareSymbol{cI1Xi4c}{2}{\draw (0,0) node[xic] {} -- (-1,1) node[not] {}; \draw[kernels2] (-1,1) -- (0,2) ; 
\draw[kernels2] (-1,1) -- (-2,2) node[xic] {} ;
\draw (0,2) node[xi] {} -- (-1,3) node[xi] {};}

\DeclareSymbol{I1Xi4ab}{2}{\draw[kernels2] (0,0) node[not] {} -- (-1,1) ; \draw[kernels2] (0,0) node[not] {} -- (1,1) node[xi] {};\draw (-1,1) node[xi] {} -- (0,2) ; \draw[kernels2] (0,2) node[not] {} -- (-1,3) node[xi] {};\draw[kernels2] (0,2)  -- (1,3) node[xi] {}; }

\DeclareSymbol{cI1Xi4ab}{2}{\draw[kernels2] (0,0) node[not] {} -- (-1,1) ; \draw[kernels2] (0,0) node[not] {} -- (1,1) node[xic] {};\draw (-1,1) node[xic] {} -- (0,2) ; \draw[kernels2] (0,2) node[not] {} -- (-1,3) node[xi] {};\draw[kernels2] (0,2)  -- (1,3) node[xi] {}; }

\DeclareSymbol{I1Xi4bc}{2}{\draw (0,0) node[xi] {} -- (-1,1) node[not] {}; \draw[kernels2] (-1,1) -- (0,2) ; 
\draw[kernels2] (-1,1) -- (-2,2) node[xi] {} ; \draw[kernels2] (0,2) node[not] {} -- (-1,3) node[xi] {};\draw[kernels2] (0,2)  -- (1,3) node[xi] {};
}

\DeclareSymbol{cI1Xi4bc}{2}{\draw (0,0) node[xic] {} -- (-1,1) node[not] {}; \draw[kernels2] (-1,1) -- (0,2) ; 
\draw[kernels2] (-1,1) -- (-2,2) node[xic] {} ; \draw[kernels2] (0,2) node[not] {} -- (-1,3) node[xi] {};\draw[kernels2] (0,2)  -- (1,3) node[xi] {};
}

\DeclareSymbol{I1Xi4abcc1}{2}{\draw[kernels2] (0,0) node[not] {} -- (-1,1) node[not] {}
-- (-2,2) node[not]{} -- (-3,3) node[xic]  {};
\draw[kernels2] (0,0) -- (1,1) node[xic] {};
\draw[kernels2] (-1,1) -- (0,2) node[xi] {};
\draw[kernels2] (-2,2) -- (-1,3) node[xi] {};
}

\DeclareSymbol{I1Xi4abcc1b}{2}{\draw[kernels2] (0,0) node[not] {} -- (-1,1) node[not] {}
-- (-2,2) node[not]{} -- (-3,3) node[xi]  {};
\draw[kernels2] (0,0) -- (1,1) node[xic] {};
\draw[kernels2] (-1,1) -- (0,2) node[xic] {};
\draw[kernels2] (-2,2) -- (-1,3) node[xi] {};
}

\DeclareSymbol{I1Xi4abcc2}{2}{\draw[kernels2] (0,0) node[not] {} -- (-1,1) node[not] {}
-- (-2,2) node[not]{} -- (-3,3) node[xic]  {};
\draw[kernels2] (0,0) -- (1,1) node[xi] {};
\draw[kernels2] (-1,1) -- (0,2) node[xi] {};
\draw[kernels2] (-2,2) -- (-1,3) node[xic] {};
}

\DeclareSymbol{I1Xi4ac}{2}{\draw[kernels2] (0,0) node[not] {} -- (-1,1) ; \draw[kernels2] (0,0) node[not] {} -- (1,1) node[xi] {}; 
\draw[kernels2] (-1,1) node[not] {} -- (0,2) ; 
\draw[kernels2] (-1,1) -- (-2,2) node[xi] {} ;
\draw (0,2) node[xi] {} -- (-1,3) node[xi] {};}

\DeclareSymbol{cI1Xi4ac}{2}{\draw[kernels2] (0,0) node[not] {} -- (-1,1) ; \draw[kernels2] (0,0) node[not] {} -- (1,1) node[xic] {}; 
\draw[kernels2] (-1,1) node[not] {} -- (0,2) ; 
\draw[kernels2] (-1,1) -- (-2,2) node[xic] {} ;
\draw (0,2) node[xi] {} -- (-1,3) node[xi] {};}

\DeclareSymbol{I1Xi4acc1}{2}{\draw[kernels2] (0,0) node[not] {} -- (-1,1) ; \draw[kernels2] (0,0) node[not] {} -- (1,1) node[xic] {}; 
\draw[kernels2] (-1,1) node[not] {} -- (0,2) ; 
\draw[kernels2] (-1,1) -- (-2,2) node[xi] {} ;
\draw (0,2) node[xic] {} -- (-1,3) node[xi] {};}

\DeclareSymbol{I1Xi4acc2}{2}{\draw[kernels2] (0,0) node[not] {} -- (-1,1) ; \draw[kernels2] (0,0) node[not] {} -- (1,1) node[xic] {}; 
\draw[kernels2] (-1,1) node[not] {} -- (0,2) ; 
\draw[kernels2] (-1,1) -- (-2,2) node[xi] {} ;
\draw (0,2) node[xi] {} -- (-1,3) node[xic] {};}

\DeclareSymbol{2I1Xi4}{2}{\draw[kernels2] (0,0) node[not] {} -- (-1,1) node[not] {};
\draw[kernels2] (0,0) -- (1,1) node[not] {};
\draw[kernels2] (-1,1) -- (-1.5,2.5) node[xi] {};
\draw[kernels2] (-1,1) -- (-0.5,2.5) node[xi] {};
\draw[kernels2] (1,1) -- (0.5,2.5) node[xi] {};
\draw[kernels2] (1,1) -- (1.5,2.5) node[xi] {};
}

\DeclareSymbol{2I1Xi4dis}{2}{\draw[kernels2] (0,0) node[not] {} -- (-1,1) node[not] {};
\draw[kernels2] (0,0) -- (1,1) node[not] {};
\draw[kernels2] (-1,1) -- (-1.5,2.5) node[xies] {};
\draw[kernels2] (-1,1) -- (-0.5,2.5) node[xies] {};
\draw[kernels2] (1,1) -- (0.5,2.5) node[xies] {};
\draw[kernels2] (1,1) -- (1.5,2.5) node[xies] {};
}

\DeclareSymbol{2I1Xi4c1}{2}{\draw[kernels2] (0,0) node[not] {} -- (-1,1) node[not] {};
\draw[kernels2] (0,0) -- (1,1) node[not] {};
\draw[kernels2] (-1,1) -- (-1.5,2.5) node[xic] {};
\draw[kernels2] (-1,1) -- (-0.5,2.5) node[xi] {};
\draw[kernels2] (1,1) -- (0.5,2.5) node[xic] {};
\draw[kernels2] (1,1) -- (1.5,2.5) node[xi] {};
}

\DeclareSymbol{2I1Xi4c2}{2}{\draw[kernels2] (0,0) node[not] {} -- (-1,1) node[not] {};
\draw[kernels2] (0,0) -- (1,1) node[not] {};
\draw[kernels2] (-1,1) -- (-1.5,2.5) node[xic] {};
\draw[kernels2] (-1,1) -- (-0.5,2.5) node[xic] {};
\draw[kernels2] (1,1) -- (0.5,2.5) node[xi] {};
\draw[kernels2] (1,1) -- (1.5,2.5) node[xi] {};
}

\DeclareSymbol{2I1Xi4b}{2}{\draw[kernels2] (0,0) node[not] {} -- (-1,1) ;
\draw[kernels2] (0,0) -- (1,1);
\draw (-1,1) node[xi] {} -- (-1,2.5) node[xi] {};
\draw (1,1)  node[xi] {} -- (1,2.5) node[xi] {};
}

\DeclareSymbol{2I1Xi4bb}{2}{\draw[kernels2] (0,0) node[not] {} -- (-1,1) ;
\draw[kernels2] (0,0) -- (1,1);
\draw (-1,1) node[xi] {} -- (-1,2.5) node[xiesf] {};
\draw (1,1)  node[xi] {} -- (1,2.5) node[xic] {};
}

\DeclareSymbol{2I1Xi4c}{2}{\draw[kernels2] (0,0) node[not] {} -- (-1,1);
\draw[kernels2] (0,0) -- (1,1) node[not] {};
\draw (-1,1)  node[xi] {} -- (-1,2.5) node[xi] {};
\draw[kernels2] (1,1) -- (0.4,2.5) node[xi] {};
\draw[kernels2] (1,1) -- (1.6,2.5) node[xi] {};
}

\DeclareSymbol{2I1Xi4cc1}{2}{\draw[kernels2] (0,0) node[not] {} -- (-1,1);
\draw[kernels2] (0,0) -- (1,1) node[not] {};
\draw (-1,1)  node[xic] {} -- (-1,2.5) node[xi] {};
\draw[kernels2] (1,1) -- (0.4,2.5) node[xic] {};
\draw[kernels2] (1,1) -- (1.6,2.5) node[xi] {};
}

\DeclareSymbol{2I1Xi4cc2}{2}{\draw[kernels2] (0,0) node[not] {} -- (-1,1);
\draw[kernels2] (0,0) -- (1,1) node[not] {};
\draw (-1,1)  node[xic] {} -- (-1,2.5) node[xic] {};
\draw[kernels2] (1,1) -- (0.4,2.5) node[xi] {};
\draw[kernels2] (1,1) -- (1.6,2.5) node[xi] {};
}

\DeclareSymbol{Xi4ba}{0}{\draw(-0.5,1.5) node[xi] {} -- (0,0); \draw (-1.5,1) node[xi] {} -- (0,0) node[not] {}; \draw[kernels2] (0,0) -- (1.5,1) node[xi] {};
\draw[kernels2] (0,0) -- (0.5,1.5) node[xi] {} ;}

\DeclareSymbol{Xi4badis}{0}{\draw(-0.5,1.5) node[xies] {} -- (0,0); \draw (-1.5,1) node[xies] {} -- (0,0) node[not] {}; \draw[kernels2] (0,0) -- (1.5,1) node[xies] {};
\draw[kernels2] (0,0) -- (0.5,1.5) node[xies] {} ;}

\DeclareSymbol{Xi4ba1}{0}{\draw(-0.5,1.5) node[xi] {} -- (0,0); \draw (-1.5,1) node[xi] {} -- (0,0) node[not] {}; \draw[kernels2] (0,0) -- (1.5,1) node[xic] {};
\draw[kernels2] (0,0) -- (0.5,1.5) node[xic] {} ;}

\DeclareSymbol{Xi4ba1b}{0}{\draw(-0.5,1.5) node[xic] {} -- (0,0); \draw (-1.5,1) node[xic] {} -- (0,0) node[not] {}; \draw[kernels2] (0,0) -- (1.5,1) node[xi] {};
\draw[kernels2] (0,0) -- (0.5,1.5) node[xi] {} ;}

\DeclareSymbol{Xi4ba1bdiff}{0}{\draw(-0.5,1.5) node[xic] {} -- (0,0); \draw (-1.5,1) node[xic] {} -- (0,0) node[not] {}; \draw (0,0) -- (1.5,1) node[xi] {};
\draw (0,0) -- (0.5,1.5) node[xi] {};
\draw(0,0) node[diff] {};}

\DeclareSymbol{Xi4ba1bb}{0}{\draw(-0.5,1.5) node[xic] {} -- (0,0); \draw (-1.5,1) node[xiesf] {} -- (0,0) node[not] {}; \draw[kernels2] (0,0) -- (1.5,1) node[xi] {};
\draw[kernels2] (0,0) -- (0.5,1.5) node[xi] {} ;}

\DeclareSymbol{Xi4ba2}{0}{\draw(-0.5,1.5) node[xi] {} -- (0,0); \draw (-1.5,1) node[xic] {} -- (0,0) node[not] {}; \draw[kernels2] (0,0) -- (1.5,1) node[xi] {};
\draw[kernels2] (0,0) -- (0.5,1.5) node[xic] {} ;}

\DeclareSymbol{Xi4ba2b}{0}{\draw(-0.5,1.5) node[xi] {} -- (0,0); \draw (-1.5,1) node[xic] {} -- (0,0) node[not] {}; \draw[kernels2] (0,0) -- (1.5,1) node[xi] {};
\draw[kernels2] (0,0) -- (0.5,1.5) node[xiesf] {} ;}

\DeclareSymbol{Xi4ca}{0}{\draw (0,1) -- (-1,2.2) node[xi] {};\draw (0,-0.25) node[xi] {} -- (0,1) ; \draw[kernels2] (0,1) node[not] {} -- (1,2.2) node[xi] {};
\draw[kernels2] (0,1) {} -- (0,2.7) node[xi] {};
}

\DeclareSymbol{Xi4cb}{0}{\draw (-1,1) -- (-2,2) node[xi] {};\draw[kernels2] (0,0)  -- (-1,1) node[xi] {} ; \draw[kernels2] (0,0) node[not] {} -- (1,1) node[xi] {} ; 
\draw (-1,1) node[xi] {} -- (0,2) node[xi] {};}

\DeclareSymbol{Xi4cbb}{0}{\draw (-1,1) -- (-2,2) node[xiesf] {};\draw[kernels2] (0,0)  -- (-1,1) node[xi] {} ; \draw[kernels2] (0,0) node[not] {} -- (1,1) node[xi] {} ; 
\draw (-1,1) node[xi] {} -- (0,2) node[xic] {};}

\DeclareSymbol{Xi4cbc1}{0}{\draw (-1,1) -- (-2,2) node[xic] {};\draw[kernels2] (0,0)  -- (-1,1) node[xic] {} ; \draw[kernels2] (0,0) node[not] {} -- (1,1) node[xi] {} ; 
\draw (-1,1) node[xic] {} -- (0,2) node[xi] {};}

\DeclareSymbol{Xi4cbc2}{0}{\draw (-1,1) -- (-2,2) node[xi] {};\draw[kernels2] (0,0)  -- (-1,1) node[xi] {} ; \draw[kernels2] (0,0) node[not] {} -- (1,1) node[xic] {} ; 
\draw (-1,1) node[xic] {} -- (0,2) node[xi] {};}

\DeclareSymbol{Xi4cab}{0}{\draw (-1,1) -- (-2,2) node[xi] {};\draw[kernels2] (0,0)  -- (-1,1); \draw[kernels2] (0,0) node[not] {} -- (1,1) node[xi] {} ; 
\draw[kernels2] (-1,1)  {} -- (0,2) node[xi] {};
\draw[kernels2] (-1,1) node[not] {} -- (-1,2.5) node[xi] {};
}

\DeclareSymbol{Xi4cabdis}{0}{\draw (-1,1) -- (-2,2) node[xies] {};\draw[kernels2] (0,0)  -- (-1,1); \draw[kernels2] (0,0) node[not] {} -- (1,1) node[xies] {} ; 
\draw[kernels2] (-1,1)  {} -- (0,2) node[xies] {};
\draw[kernels2] (-1,1) node[not] {} -- (-1,2.5) node[xies] {};
}

\DeclareSymbol{Xi4cabc1}{0}{\draw (-1,1) -- (-2,2) node[xi] {};\draw[kernels2] (0,0)  -- (-1,1); \draw[kernels2] (0,0) node[not] {} -- (1,1) node[xic] {} ; 
\draw[kernels2] (-1,1)  {} -- (0,2) node[xic] {};
\draw[kernels2] (-1,1) node[not] {} -- (-1,2.5) node[xi] {};
}

\DeclareSymbol{Xi4cabc2}{0}{\draw (-1,1) -- (-2,2) node[xic] {};\draw[kernels2] (0,0)  -- (-1,1); \draw[kernels2] (0,0) node[not] {} -- (1,1) node[xic] {} ; 
\draw[kernels2] (-1,1)  {} -- (0,2) node[xi] {};
\draw[kernels2] (-1,1) node[not] {} -- (-1,2.5) node[xi] {};
}

\DeclareSymbol{Xi4ea}{1.5}{\draw (-1,2.5) node[xi] {} -- (-1,1) node[xi] {} -- (0,0); 
 \draw[kernels2] (0,0)  -- (1,1) node[xi] {};
\draw[kernels2] (0,0) node[not] {} -- (0,1.5) node[xi] {}; }

\DeclareSymbol{Xi4eac1}{1.5}{\draw (-1,2.5) node[xic] {} -- (-1,1) node[xi] {} -- (0,0); 
 \draw[kernels2] (0,0)  -- (1,1) node[xic] {};
\draw[kernels2] (0,0) node[not] {} -- (0,1.5) node[xi] {}; }

\DeclareSymbol{Xi4eac1b}{1.5}{\draw (-1,2.5) node[xic] {} -- (-1,1) node[xi] {} -- (0,0); 
 \draw[kernels2] (0,0)  -- (1,1) node[xiesf] {};
\draw[kernels2] (0,0) node[not] {} -- (0,1.5) node[xi] {}; }

\DeclareSymbol{Xi4eac2}{1.5}{\draw (-1,2.5) node[xic] {} -- (-1,1) node[xic] {} -- (0,0); 
 \draw[kernels2] (0,0)  -- (1,1) node[xi] {};
\draw[kernels2] (0,0) node[not] {} -- (0,1.5) node[xi] {}; }

\DeclareSymbol{Xi4eact1}{1.5}{\draw (-1,2.5) node[xic] {} -- (-1,1) node[xi] {} -- (0,0); 
 \draw (0,0)  -- (1,1) node[xic] {};
\draw[rho] (0,0) node[not] {} -- (0,1.5) node[xi] {}; }

\DeclareSymbol{Xi4eact2}{1.5}{\draw[rho] (-1,2.5) node[xic] {} -- (-1,1) node[xi] {} -- (0,0); 
 \draw (0,0)  -- (1,1) node[xic] {};
\draw (0,0) node[not] {} -- (0,1.5) node[xi] {}; }

\DeclareSymbol{Xi4eabis}{1.5}{\draw (-1,2.5) node[xi] {} -- (-1,1) ; \draw[kernels2] (-1,1) node[xi] {} -- (0,0); 
 \draw (0,0)  -- (1,1) node[xi] {};
\draw[kernels2] (0,0) node[not] {} -- (0,1.5) node[xi] {}; }

\DeclareSymbol{Xi4eabisc1}{1.5}{\draw (-1,2.5) node[xic] {} -- (-1,1) ; \draw[kernels2] (-1,1) node[xi] {} -- (0,0); 
 \draw (0,0)  -- (1,1) node[xi] {};
\draw[kernels2] (0,0) node[not] {} -- (0,1.5) node[xic] {}; }

\DeclareSymbol{Xi4eabisc1b}{1.5}{\draw (-1,2.5) node[xic] {} -- (-1,1) ; \draw[kernels2] (-1,1) node[xi] {} -- (0,0); 
 \draw (0,0)  -- (1,1) node[xi] {};
\draw[kernels2] (0,0) node[not] {} -- (0,1.5) node[xiesf] {}; }

\DeclareSymbol{Xi4eabisc1bis}{1.5}{\draw (-1,2.5) node[xi] {} -- (-1,1) ; \draw[kernels2] (-1,1) node[xi] {} -- (0,0); 
 \draw (0,0)  -- (1,1) node[xi] {};
\draw[kernels2] (0,0) node[not] {} -- (0,1.5) node[xi] {};
\draw (-2,1) node[] {\tiny{$i$}};
\draw (-2,2.5) node[] {\tiny{$\ell$}};
\draw (2,1) node[] {\tiny{$k$}};
\draw (0,2.5) node[] {\tiny{$j$}};
 }

\DeclareSymbol{Xi4eabisc1tris}{1.5}{\draw (-1,2.5) node[xi] {} -- (-1,1) ; \draw[kernels2] (-1,1) node[xi] {} -- (0,0); 
 \draw (0,0)  -- (1,1) node[xi] {};
\draw[kernels2] (0,0) node[not] {} -- (0,1.5) node[xi] {};
\draw (-2,1) node[] {\tiny{i}};
\draw (-2,2.5) node[] {\tiny{j}};
\draw (2,1) node[] {\tiny{j}};
\draw (0,2.5) node[] {\tiny{i}};
 }

\DeclareSymbol{Xi4eabisc1quater}{1.5}{\draw (-1,2.5) node[xic] {} -- (-1,1) ; \draw[kernels2] (-1,1) node[xi] {} -- (0,0); 
 \draw (0,0)  -- (1,1) node[xic] {};
\draw[kernels2] (0,0) node[not] {} -- (0,1.5) node[xi] {};
 }

\DeclareSymbol{Xi4eabisc2}{1.5}{\draw (-1,2.5) node[xic] {} -- (-1,1) ; \draw[kernels2] (-1,1) node[xi] {} -- (0,0); 
 \draw (0,0)  -- (1,1) node[xic] {};
\draw[kernels2] (0,0) node[not] {} -- (0,1.5) node[xi] {}; }

\DeclareSymbol{Xi4eabisc2l}{1.5}{\draw (-1,2.5) node[xiesf] {} -- (-1,1) ; \draw[kernels2] (-1,1) node[xi] {} -- (0,0); 
 \draw (0,0)  -- (1,1) node[xic] {};
\draw[kernels2] (0,0) node[not] {} -- (0,1.5) node[xi] {}; }

\DeclareSymbol{Xi4eabisc2r}{1.5}{\draw (-1,2.5) node[xic] {} -- (-1,1) ; \draw[kernels2] (-1,1) node[xi] {} -- (0,0); 
 \draw (0,0)  -- (1,1) node[xiesf] {};
\draw[kernels2] (0,0) node[not] {} -- (0,1.5) node[xi] {}; }

\DeclareSymbol{Xi4eabisc3}{1.5}{\draw (-1,2.5) node[xic] {} -- (-1,1) ; \draw[kernels2] (-1,1) node[xic] {} -- (0,0); 
 \draw (0,0)  -- (1,1) node[xi] {};
\draw[kernels2] (0,0) node[not] {} -- (0,1.5) node[xi] {}; }

\DeclareSymbol{Xi4eb}{0}{
\draw[kernels2] (0,2) node[xi] {} -- (-1,1) ; \draw[kernels2] (-2,2)  node[xi] {} -- (-1,1) ; \draw (-1,1)  node[not] {} -- (0,0); 
 \draw (0,0) node[xi] {}  -- (1,1) node[xi] {};
}

\DeclareSymbol{Xi4eab}{1.5}{\draw[kernels2] (-1,2.5) node[xi] {} -- (-1,1) ; \draw[kernels2] (-2,2)  node[xi] {} -- (-1,1) ; \draw (-1,1)  node[not] {} -- (0,0); 
 \draw[kernels2] (0,0)  -- (1,1) node[xi] {};
\draw[kernels2] (0,0) node[not] {} -- (0,1.5) node[xi] {}; 
}

\DeclareSymbol{Xi4eabdis}{1.5}{\draw[kernels2] (-1,2.5) node[xies] {} -- (-1,1) ; \draw[kernels2] (-2,2)  node[xies] {} -- (-1,1) ; \draw (-1,1)  node[not] {} -- (0,0); 
 \draw[kernels2] (0,0)  -- (1,1) node[xies] {};
\draw[kernels2] (0,0) node[not] {} -- (0,1.5) node[xies] {}; 
}

\DeclareSymbol{Xi4eabc1}{1.5}{\draw[kernels2] (-1,2.5) node[xic] {} -- (-1,1) ; \draw[kernels2] (-2,2)  node[xi] {} -- (-1,1) ; \draw (-1,1)  node[not] {} -- (0,0); 
 \draw[kernels2] (0,0)  -- (1,1) node[xic] {};
\draw[kernels2] (0,0) node[not] {} -- (0,1.5) node[xi] {}; 
}

\DeclareSymbol{Xi4eabc2}{1.5}{\draw[kernels2] (-1,2.5) node[xi] {} -- (-1,1) ; \draw[kernels2] (-2,2)  node[xi] {} -- (-1,1) ; \draw (-1,1)  node[not] {} -- (0,0); 
 \draw[kernels2] (0,0)  -- (1,1) node[xic] {};
\draw[kernels2] (0,0) node[not] {} -- (0,1.5) node[xic] {}; 
}

\DeclareSymbol{Xi4eabbis}{1.5}{\draw[kernels2] (-1,2.5) node[xi] {} -- (-1,1) ; \draw[kernels2] (-2,2)  node[xi] {} -- (-1,1) ; \draw[kernels2] (-1,1)  node[not] {} -- (0,0); 
 \draw (0,0)  -- (1,1) node[xi] {};
\draw[kernels2] (0,0) node[not] {} -- (0,1.5) node[xi] {}; 
}

\DeclareSymbol{Xi4eabbisc1}{1.5}{\draw[kernels2] (-1,2.5) node[xic] {} -- (-1,1) ; \draw[kernels2] (-2,2)  node[xi] {} -- (-1,1) ; \draw[kernels2] (-1,1)  node[not] {} -- (0,0); 
 \draw (0,0)  -- (1,1) node[xic] {};
\draw[kernels2] (0,0) node[not] {} -- (0,1.5) node[xi] {}; 
}

\DeclareSymbol{Xi4eabbisc1perm}{1.5}{\draw[kernels2] (-1,2.5) node[xi] {} -- (-1,1) ; \draw[kernels2] (-2,2)  node[xic] {} -- (-1,1) ; \draw[kernels2] (-1,1)  node[not] {} -- (0,0); 
 \draw (0,0)  -- (1,1) node[xic] {};
\draw[kernels2] (0,0) node[not] {} -- (0,1.5) node[xi] {}; 
}

\DeclareSymbol{Xi4eabbisc2}{1.5}{\draw[kernels2] (-1,2.5) node[xi] {} -- (-1,1) ; \draw[kernels2] (-2,2)  node[xi] {} -- (-1,1) ; \draw[kernels2] (-1,1)  node[not] {} -- (0,0); 
 \draw (0,0)  -- (1,1) node[xic] {};
\draw[kernels2] (0,0) node[not] {} -- (0,1.5) node[xic] {}; 
}

\DeclareSymbol{Xi2cbis}{0}{\draw[kernels2] (0,1) -- (0.8,2.2) node[xi] {};\draw[kernels2] (0,-0.25) node[not] {} -- (0,1); \draw[kernels2] (0,1) node[not] {} -- (-0.8,2.2) node[xi] {};}

\DeclareSymbol{Xi2cbis1}{0}{\draw (0,1) -- (-0.8,2.2) node[xi] {};\draw[kernels2] (0,-0.25) node[not] {} -- (0,1) node[xi] {}; }

\DeclareSymbol{Xi2Xbis}{-2}{\draw[kernels2] (0,-0.25)  -- (-1,1) ; \draw (-1,1) node[xix] {};
\draw[kernels2] (0,-0.25) node[not] {} -- (1,1) node[xi] {};}

\DeclareSymbol{XXi2bis}{-2}{\draw[kernels2] (0,-0.25) -- (-1,1) node[xi] {};
\draw[kernels2] (0,-0.25) node[X] {} -- (1,1) node[xi] {};}

\DeclareSymbol{I1XiIXi}{0}{\draw[kernels2] (0,-0.25) -- (1,1) node[xi] {};
\draw (0,-0.25) node[not] {} -- (-1,1) node[xi] {};}

\DeclareSymbol{I1XiIXib}{0}{\draw  (0,-0.25) node[xi] {} -- (0,1) node[not] {};
\draw[kernels2] (0,1) -- (0,2.25) ; \draw (0,2.25) node[xi]{}; }

\DeclareSymbol{I1XiIXic}{0}{
\draw[kernels2] (0,0) -- (1,1) node[xi] {} ; 
\draw[kernels2] (0,0) node[not] {}  -- (-1,1) node[not] {} -- (0,2) node[xi] {};
}

\DeclareSymbol{thin}{1.4}{\draw[pagebackground] (-0.3,0) -- (0.3,0); \draw  (0,0) -- (0,2);}
\DeclareSymbol{thin2}{1.4}{\draw[pagebackground] (-0.3,0) -- (0.3,0); \draw[tinydots]  (0,0) -- (0,2);}

\DeclareSymbol{thick}{1.4}{\draw[pagebackground] (-0.3,0) -- (0.3,0); \draw[kernels2]  (0,0) -- (0,2);}

\DeclareSymbol{thick2}{1.4}{\draw[pagebackground] (-0.3,0) -- (0.3,0); \draw[kernels2,tinydots]  (0,0) -- (0,2);}

\DeclareSymbol{Xi4ind}{2}{\draw (0,0) node[xi,label={[label distance=-0.2em]right: \scriptsize  $ i $}]  { } -- (-1,1) node[xi,label={[label distance=-0.2em]left: \scriptsize  $ j $}] {} -- (0,2) node[xi,label={[label distance=-0.2em]right: \scriptsize  $ k $}] {} -- (-1,3) node[xi,label={[label distance=-0.2em]left: \scriptsize  $ \ell $}] {};}

\DeclareSymbol{Xi4c1}{2}{\draw (0,0) node[xic] {} -- (-1,1) node[xi] {} -- (0,2) node[xic] {} -- (-1,3) node[xi] {};} 
\DeclareSymbol{IXi2ex}{0}{\draw (0,-0.25) node[xie] {} -- (-1,1) node[xi] {} ; \draw (0,-0.25)-- (1,1) node[xi] {};}
\DeclareSymbol{IXi2ex1}{0}{\draw (0,-0.25) node[xie] {} -- (-1,1) node[xi] {} -- (0,2) node[xi] {};}

\DeclareSymbol{Xi4b1}{0}{\draw(0,1.5) node[xic] {} -- (0,0); \draw (-1,1) node[xic] {} -- (0,0) node[xi] {} -- (1,1) node[xi] {};}

\DeclareSymbol{Xi4ec1}{0}{\draw (0,2) node[xi] {} -- (-1,1) node[xic] {} -- (0,0) node[xic] {} -- (1,1) node[xi] {};}
\DeclareSymbol{Xi4ec2}{0}{\draw (0,2) node[xic] {} -- (-1,1) node[xi] {} -- (0,0) node[xic] {} -- (1,1) node[xi] {};}
\DeclareSymbol{Xi4ec3}{0}{\draw (0,2) node[xic] {} -- (-1,1) node[xic] {} -- (0,0) node[xi] {} -- (1,1) node[xi] {};}

\DeclareSymbol{I1Xi4ac1}{2}{\draw[kernels2] (0,0) node[not] {} -- (-1,1) ; \draw[kernels2] (0,0) node[not] {} -- (1,1) node[xic] {} ;
\draw (-1,1) node[xi] {} -- (0,2) node[xic] {} -- (-1,3) node[xi] {};}

\DeclareSymbol{I1Xi4ac2}{2}{\draw[kernels2] (0,0) node[not] {} -- (-1,1) ; \draw[kernels2] (0,0) node[not] {} -- (1,1) node[xic] {} ;
\draw (-1,1) node[xi] {} -- (0,2) node[xi] {} -- (-1,3) node[xic] {};}

\DeclareSymbol{I1Xi4bp}{2}{\draw (0,0) node[not] {} -- (-1,1) node[xi] {} -- (0,2) ; \draw[kernels2] (0,2) node[not] {} -- (-1,3) node[xi] {};\draw[kernels2] (0,2)  -- (1,3) node[xi] {};
}

\DeclareSymbol{I1Xi4bc1}{2}{\draw (0,0) node[xic] {} -- (-1,1) node[xi] {} -- (0,2) ; \draw[kernels2] (0,2) node[not] {} -- (-1,3) node[xi] {};\draw[kernels2] (0,2)  -- (1,3) node[xic] {};
}

\DeclareSymbol{I1Xi4bc2}{2}{\draw (0,0) node[xic] {} -- (-1,1) node[xi] {} -- (0,2) ; \draw[kernels2] (0,2) node[not] {} -- (-1,3) node[xic] {};\draw[kernels2] (0,2)  -- (1,3) node[xi] {};
}

\DeclareSymbol{I1Xi4cp}{2}{\draw (0,0) node[not] {} -- (-1,1) node[not] {}; \draw[kernels2] (-1,1) -- (0,2) ; 
\draw[kernels2] (-1,1) -- (-2,2) node[xi] {} ;
\draw (0,2) node[xi] {} -- (-1,3) node[xi] {};}

\DeclareSymbol{I1Xi4cc1}{2}{\draw (0,0) node[xic] {} -- (-1,1) node[not] {}; \draw[kernels2] (-1,1) -- (0,2) ; 
\draw[kernels2] (-1,1) -- (-2,2) node[xi] {} ;
\draw (0,2) node[xic] {} -- (-1,3) node[xi] {};}

\DeclareSymbol{I1Xi4cc2}{2}{\draw (0,0) node[xic] {} -- (-1,1) node[not] {}; \draw[kernels2] (-1,1) -- (0,2) ; 
\draw[kernels2] (-1,1) -- (-2,2) node[xi] {} ;
\draw (0,2) node[xi] {} -- (-1,3) node[xic] {};}

\DeclareSymbol{I1Xi4abc1}{2}{\draw[kernels2] (0,0) node[not] {} -- (-1,1) ; \draw[kernels2] (0,0) node[not] {} -- (1,1) node[xic] {};\draw (-1,1) node[xi] {} -- (0,2) ; \draw[kernels2] (0,2) node[not] {} -- (-1,3) node[xic] {};\draw[kernels2] (0,2)  -- (1,3) node[xi] {}; }

\DeclareSymbol{I1Xi4abc2}{2}{\draw[kernels2] (0,0) node[not] {} -- (-1,1) ; \draw[kernels2] (0,0) node[not] {} -- (1,1) node[xic] {};\draw (-1,1) node[xi] {} -- (0,2) ; \draw[kernels2] (0,2) node[not] {} -- (-1,3) node[xi] {};\draw[kernels2] (0,2)  -- (1,3) node[xic] {}; }

\DeclareSymbol{R1}{0}{\draw (-1,1) node[xi] {} -- (0,0) node[not] {};
\draw[kernels2] (0,1.5) node[xic] {} -- (0,0) -- (1,1) node[xic] {};}
\DeclareSymbol{R2}{0}{\draw (-1,1) node[xic] {} -- (0,0) node[not] {};
\draw[kernels2] (0,1.5)  {} -- (0,0) -- (1,1)  {};
\draw (0,1.5) node[xi] {};
\draw (1,1) node[xic] {};
}
\DeclareSymbol{R3}{1}{\draw[kernels2] (-1,1.5)  {} -- (0,0) node[not] {} -- (1,1.5);
\draw (-1,1.5) node[xi] {};
\draw[kernels2] (0,3) {} -- (1,1.5) -- (2,3)  {};
\draw  (0,3) node[xic] {} ;
\draw (2,3) node[xic] {};}
\DeclareSymbol{R4}{1}{\draw[kernels2] (-1,1.5) node[xic] {} -- (0,0) node[not] {} -- (1,1.5);
\draw[kernels2] (0,3) {} -- (1,1.5) -- (2,3) node[xic] {};
\draw (0,3) node[xi] {};}

\DeclareSymbol{I1Xi4bcp}{2}{\draw (0,0) node[not] {} -- (-1,1) node[not] {}; \draw[kernels2] (-1,1) -- (0,2) ; 
\draw[kernels2] (-1,1) -- (-2,2) node[xi] {} ; \draw[kernels2] (0,2) node[not] {} -- (-1,3) node[xi] {};\draw[kernels2] (0,2)  -- (1,3) node[xi] {};
}

\DeclareSymbol{I1Xi4bcc1}{2}{\draw (0,0) node[xic] {} -- (-1,1) node[not] {}; \draw[kernels2] (-1,1) -- (0,2) ; 
\draw[kernels2] (-1,1) -- (-2,2) node[xi] {} ; \draw[kernels2] (0,2) node[not] {} -- (-1,3) node[xi] {};\draw[kernels2] (0,2)  -- (1,3) node[xic] {};
}

\DeclareSymbol{I1Xi4bcc2}{2}{\draw (0,0) node[xic] {} -- (-1,1) node[not] {}; \draw[kernels2] (-1,1) -- (0,2) ; 
\draw[kernels2] (-1,1) -- (-2,2) node[xi] {} ; \draw[kernels2] (0,2) node[not] {} -- (-1,3) node[xic] {};\draw[kernels2] (0,2)  -- (1,3) node[xi] {};
} 

\DeclareSymbol{2I1Xi4bc1}{2}{\draw[kernels2] (0,0) node[not] {} -- (-1,1) ;
\draw[kernels2] (0,0) -- (1,1);
\draw (-1,1) node[xic] {} -- (-1,2.5) node[xi] {};
\draw (1,1)  node[xic] {} -- (1,2.5) node[xi] {};
}

\DeclareSymbol{2I1Xi4bc2}{2}{\draw[kernels2] (0,0) node[not] {} -- (-1,1) ;
\draw[kernels2] (0,0) -- (1,1);
\draw (-1,1) node[xi] {} -- (-1,2.5) node[xic] {};
\draw (1,1)  node[xic] {} -- (1,2.5) node[xi] {};
}

\DeclareSymbol{diff2I1Xi4bc2}{2}{\draw (0,0) node[diff] {} -- (-1,1) ;
\draw (0,0) -- (1,1);
\draw (-1,1) node[xi] {} -- (-1,2.5) node[xic] {};
\draw (1,1)  node[xic] {} -- (1,2.5) node[xi] {};
}

\DeclareSymbol{2I1Xi4bc3}{2}{\draw[kernels2] (0,0) node[not] {} -- (-1,1) ;
\draw[kernels2] (0,0) -- (1,1);
\draw (-1,1) node[xic] {} -- (-1,2.5) node[xic] {};
\draw (1,1)  node[xi] {} -- (1,2.5) node[xi] {};
}

\DeclareSymbol{Xi41}{0}{\draw (0,1) -- (0.8,2.2) node[xic] {};\draw (0,-0.25) node[xi] {} -- (0,1) node[xi] {} -- (-0.8,2.2) node[xic] {};} 

\DeclareSymbol{Xi42}{0}{\draw (0,1) -- (0.8,2.2) node[xi] {};\draw (0,-0.25) node[xic] {} -- (0,1) node[xi] {} -- (-0.8,2.2) node[xic] {};}

\DeclareSymbol{Xi4ca1}{0}{\draw (0,1) -- (-1,2.2) node[xic] {};\draw (0,-0.25) node[xi] {} -- (0,1) ; \draw[kernels2] (0,1) node[not] {} -- (1,2.2) node[xic] {};
\draw[kernels2] (0,1) {} -- (0,2.7) node[xi] {};
}

\DeclareSymbol{Xi4ca2}{0}{\draw (0,1) -- (-1,2.2) node[xi] {};\draw (0,-0.25) node[xi] {} -- (0,1) ; \draw[kernels2] (0,1) node[not] {} -- (1,2.2) node[xic] {};
\draw[kernels2] (0,1) {} -- (0,2.7) node[xic] {};
}

\DeclareSymbol{Xi4cap}{0}{\draw (0,1) -- (-1,2.2) node[xi] {};\draw (0,-0.25) node[not] {} -- (0,1) ; \draw[kernels2] (0,1) node[not] {} -- (1,2.2) node[xi] {};
\draw[kernels2] (0,1) {} -- (0,2.7) node[xi] {};
}

\DeclareSymbol{Xi3a}{0}{
 \draw (-1,1)  node[xi] {} -- (0,0); 
 \draw (0,0) node[xi] {}  -- (1,1) node[xi] {};
}

\DeclareSymbol{Xi4ebc1}{0}{
\draw[kernels2] (0,2) node[xi] {} -- (-1,1) ; \draw[kernels2] (-2,2)  node[xic] {} -- (-1,1) ; \draw (-1,1)  node[not] {} -- (0,0); 
 \draw (0,0) node[xic] {}  -- (1,1) node[xi] {};
}

\DeclareSymbol{Xi4ebc2}{0}{
\draw[kernels2] (0,2) node[xi] {} -- (-1,1) ; \draw[kernels2] (-2,2)  node[xi] {} -- (-1,1) ; \draw (-1,1)  node[not] {} -- (0,0); 
 \draw (0,0) node[xic] {}  -- (1,1) node[xic] {};
}

\DeclareSymbol{Xi2cbispex}{0}{\draw[kernels2] (0,1) -- (0.8,2.2) node[xi] {};\draw (0,-0.25) node[xie] {} -- (0,1); \draw[kernels2] (0,1) node[not] {} -- (-0.8,2.2) node[xi] {};}

\DeclareSymbol{Xi2cbis1p}{0}{\draw (0,1) -- (-0.8,2.2) node[xi] {};\draw (0,-0.25) node[not] {} -- (0,1) node[xi] {}; }

\DeclareSymbol{Xi2Xp}{-2}{\draw (0,-0.25) node[not] {} -- (-1,1) node[xix] {};} 

\DeclareSymbol{I1XiIXib}{0}{\draw  (0,-0.25) node[xi] {} -- (0,1) node[not] {};
\draw[kernels2] (0,1) -- (0,2.25) ; \draw (0,2.25) node[xi]{}; }

\DeclareSymbol{IXi2b}{0}{\draw  (0,-0.25) node[xi] {} -- (0,1) node[not] {};
\draw (0,1) -- (0,2.25) ; \draw (0,2.25) node[xi]{}; }

\DeclareSymbol{IXi2bex}{0}{\draw  (0,-0.25) node[xi] {} -- (0,1) node[xie] {};
\draw (0,1) -- (0,2.25) ; \draw (0,2.25) node[xi]{}; }

 \def\1{\mathbf{\symbol{1}}}

\def\one{\mathbf{1}}

\DeclareSymbol{diff}{0}{
\draw (0,0.5) node[diff] {};
}

\DeclareSymbol{diff1}{0}{
\draw (0,0.5) node[diff1] {};
}

\DeclareSymbol{diff2}{0}{
\draw (0,0.5) node[diff2] {};
}

\DeclareSymbol{geo}{0}{
\draw (0,0) node[diff] {};
\draw (0.3,0) node[diff] {};
}

\DeclareSymbol{generic}{0}{
\draw (0,0.6) node[xi] {};
}

\DeclareSymbol{g}{0}{
\draw (0,0.6) node[g] {};
}

\DeclareSymbol{Ito}{0}{
\draw (0,0.6) node[xies] {};
}

\DeclareSymbol{Itob}{0}{
\draw (0,0.6) node[xiesf] {};
}

\DeclareSymbol{greycirc}{0}{
\draw (0,0.3) node[xi] {};
}

\DeclareSymbol{not}{0}{
\draw (0,0.6) node[not] {};
\draw[tinydots] (0,0.6) circle (0.8);
}

\DeclareSymbol{genericb}{0}{
\draw (0,0.6) node[xic] {};
}

\DeclareSymbol{bluecirc}{0}{
\draw (0,0.3) node[xic] {};
}

\DeclareSymbol{genericxix}{0}{
\draw (0,0.6) node[xix] {};
}

\DeclareSymbol{genericX}{0}{
\draw (0,0.6) node[X] {};
}

\DeclareSymbol{diffIto}{1}{
\draw  (0,2.5) -- (0,0) ;
\draw (0,-0.1) node[diff] {};
\draw (0,2.5) node[xies] {};
}
\DeclareSymbol{Itodiff}{2}{
\draw(0,2.9) -- (0,-0.2);
\draw (0,2.9) node[diff] {};
\draw (0,-0.1) node[xies] {};
}

\DeclareSymbol{diffgeneric}{1}{
\draw  (0,2.5) -- (0,0) ;
\draw (0,-0.1) node[diff] {};
\draw (0,2.5) node[xi] {};
}

\DeclareSymbol{genericdiff}{2}{
\draw(0,2.9) -- (0,-0.2);
\draw (0,2.9) node[diff] {};
\draw (0,-0.1) node[xi] {};
}

\DeclareSymbol{diffdot}{2}{
\draw  (0,3) -- (0,-0.1) ;
\draw (0,3) node[not] {};
\draw (0,-0.1) node[diff] {};
}

\DeclareSymbol{diffdotmini}{0}{
\draw  (0,0) -- (0,1.2) ;
\draw (0,1.2) node[not] {};
\draw (0,0) node[diffmini] {};
}

\DeclareSymbol{dotdiff}{2}{
\draw[kernelsmod]  (0,3) -- (0,-0.1) ;
\draw (0,3) node[diff] {};
\draw (0,-0.1) node[not] {};
}

\DeclareSymbol{dotdiff1}{2}{
\draw[kernelsmod]  (0,3) -- (0,-0.1) ;
\draw (0,3) node[diff1] {};
\draw (0,-0.1) node[not] {};
}

\DeclareSymbol{dotdiff1mini}{0}{
\draw[kernelsmod]  (0,1.2) -- (0,0) ;
\draw (0,1.2) node[diffmini] {};
\draw (0,0) node[not] {};
}

\DeclareSymbol{dotdiff2}{2}{
\draw (0,3) -- (0,-0.1) ;
\draw (0,3) node[diff] {};
\draw (0,-0.1) node[not] {};
}

\DeclareSymbol{dotdiff2mini}{0}{
\draw (0,1.2) -- (0,0) ;
\draw (0,1.2) node[diffmini] {};
\draw (0,0) node[not] {};
}

\DeclareSymbol{dotdiffstraight}{0}{
\draw  (0,3) -- (0,-0.1) ;
\draw (0,3) node[diff] {};
\draw (0,-0.1) node[not] {};
}

\DeclareSymbol{arbre1}{0}{
\draw  (0,0) -- (1.5,1.5) ;
\draw (1.5,1.5) node[not] {};
\draw (0,0) node[not] {};
}

\DeclareSymbol{arbre2}{0}{
\draw  (0,0) -- (1.5,1.5) ;
\draw[kernelsmod] (0,0) -- (-1.5,1.5);
\draw (1.5,1.5) node[not] {};
\draw (0,0) node[not] {};
\draw (-1.5,1.5) node[xi] {};
}

\DeclareSymbol{arbre3}{0}{
\draw  (0,0) -- (1.5,1.5) ;
\draw[kernelsmod] (1.5,1.5) -- (0,3);
\draw (0,0) node[not] {};
\draw (1.5,1.5) node[not] {};
\draw (0,3) node[xi] {};
}

\DeclareSymbol{treeeval}{0}{
\draw (0,0) -- (1,1);
\draw (0,0) node[xi] {};
\draw (1.25,1.25) node[xi] {};
\draw (-0.6,0.6) node[]{\tiny{$i$}};
\draw (0.65,1.85) node[]{\tiny{$j$}};
}

\DeclareSymbol{testeval}{0}{
\draw (0,0) -- (1,1);
\draw (0,0) -- (-1,1);
\draw (0,0) node[xi] {};
\draw (1.25,1.25) node[xi] {};
\draw (-1.25,1.25) node[xi] {};
\draw (-0.6,-0.6) node[]{\tiny{$i$}};
\draw (0.65,1.85) node[]{\tiny{$j$}};
\draw (-1.95,1.85) node[]{\tiny{$k$}};
}

\DeclareSymbol{treeeval2}{0}{
\draw[kernelsmod] (-0.25,-1) -- (1,0.5) ;
\draw[kernelsmod] (1,0.5) -- (-0.25,2);
\draw (1,0.5) node[diff2] {};
\draw (-0.25,-1) node[not] {};
\draw (-0.25,2) node[xi] {};
\draw (-0.6,1.2) node[]{\tiny{1}};
}

\DeclareSymbol{arbreact}{1}{
\draw (0,0) node[not] {};
\draw[kernelsmod] (0,0) -- (1,1);
\draw[kernelsmod] (0,0) -- (-1,1);
\draw (-1,1) node[xic] {};
\draw  (0,2) -- (1,1) ;
\draw (0,2) node[xic] {};
\draw (1,1) node[xi] {};
}

\DeclareSymbol{arbreact1}{0}{
\draw (0,-1.5) -- (0,0);
\draw[kernelsmod] (0,0) -- (1,1);
\draw[kernelsmod] (0,0) -- (-1,1);
\draw  (0,2) -- (1,1) ;
\draw (0,-1.5) node[diff] {};
\draw (0,0) node[not] {};
\draw (-1,1) node[xic] {};
\draw (0,2) node[xic] {};
\draw (1,1) node[xi] {};
}

\DeclareSymbol{arbreact2}{0}{
\draw (0,-0.75) -- (-1,0.5); 
\draw (0,-0.75) -- (1,0.5);
\draw (0,1.5) -- (1,0.5);
\draw (0,1.5) node[xic] {};
\draw (1,0.5) node[xi] {};
\draw (-1,0.5) node[xic] {};
\draw (0,-0.75) node[diff] {};
}

\DeclareSymbol{arbreact3}{0}{
\draw[kernelsmod] (0,-0.75) -- (-1,0.5); 
\draw[kernelsmod] (0,-0.75) -- (1,0.5);
\draw (0,1.75) -- (1,0.5);
\draw (2,1.75) -- (1,0.5);
\draw (0,1.75) node[xic] {};
\draw (1,0.5) node[diff] {};
\draw (-1,0.5) node[xic] {};
\draw (2,1.75) node[xi] {};
\draw (0,-0.75) node[not] {};
}

\DeclareSymbol{pre_im_I1Xitwo}{0}{
\draw[kernels2] (0,-0.3) node[not] {} -- (-0.6,0.7) ;
\draw[kernels2] (0,-0.3) -- (0.6,0.7);
\draw (0,0.9) node[g] {};
}

\DeclareSymbol{pre_im_cI1Xi4ab}{2}{
\draw[kernels2] (0,-1) node[not] {} -- (-0.6,0) ;
\draw[kernels2] (0,-1) -- (0.6,0);
\draw (0,0.2) node[g] {};
\draw (0,0.6) -- (0,1.5);
\draw[kernels2] (0,1.5) node[not] {} -- (-0.6,2.5) ;
\draw[kernels2] (0,1.5) -- (0.6,2.5);
\draw (0,2.7) node[g] {};
}

\DeclareSymbol{pre_im_I1Xi4acc2}{0}{
\draw[kernels2] (-1,-0.5) node[not] {} -- (-1.6,0.5) ;
\draw[kernels2] (-1,-0.5) -- (-0.4,0.5);
\draw[kernels2] (-1,-0.5) -- (0.2,-1.5) node[not] {} ;
\draw (-1,1.1) -- (-1,2);
\draw[kernels2] (0.2,-1.5) -- (0.2,2);
\draw (-1,0.7) node[g] {};
\draw (-0.3,2.2) node[g] {};
}

\DeclareSymbol{pre_im_I1Xi4abcc2}{2}{
\draw[kernels2] (0,-1) node[not] {} -- (-1,0) node[not] {};
\draw[kernels2] (-1,1.2) node[not] {} -- (-1,0);
\draw[kernels2] (-1,1.2) -- (-1.5,2.5);
\draw[kernels2] (-1,1.2) -- (-0.5,2.5);
\draw[kernels2] (-1,0) -- (0.7,2.5);
\draw[kernels2] (0,-1) -- (1.5,2.5);
\draw (-1,2.7) node[g] {};
\draw (1,2.7) node[g] {};
}

\DeclareSymbol{pre_im_2I1Xi4c1}{2}{
\draw[kernels2] (0,-0.5) node[not] {} -- (-1,0.5) node[not] {};
\draw[kernels2] (0,-0.5) -- (1,0.5) node[not] {};
\draw[kernels2] (-1,0.5) node[not] {}-- (-1.7,2);
\draw[kernels2]  (-1,2) -- (1,0.5);
\draw[kernels2] (-1,0.5) -- (1,2);
\draw[kernels2] (1,0.5) -- (1.7,2);
\draw (-1.2,2.2) node[g] {};
\draw (1.2,2.2) node[g] {};
}

\DeclareSymbol{pre_im_Xi4eabisc2}{2}{
\draw[kernels2] (1.2,-0.5) node[not] {} -- (-0.7,0.8) ;
\draw[kernels2] (1.2,-0.5) -- (0.4,0.8);
\draw (0,1.4)  -- (0,2.2);
\draw (1.2,2.2) -- (1.2,-0.6);
\draw (0,1) node[g] {};
\draw (0.6,2.4) node[g] {};
}

\DeclareSymbol{pre_im_Xi4eabisc22}{2}{
\draw (1.2,-0.5) node[not] {} -- (-0.7,0.8) ;
\draw[kernels2] (1.2,-0.5) -- (0.4,0.8);
\draw (0,1.4)  -- (0,2.2);
\draw[kernels2] (1.2,2.2) -- (1.2,-0.6);
\draw (0,1) node[g] {};
\draw (0.6,2.4) node[g] {};
}

\DeclareSymbol{pre_im_Xi4eabisc222}{2}{
\draw[kernels2] (0.4,-0.5) node[not] {} -- (-0.6,1) ;
\draw[kernels2] (1.2,0) -- (0.3,1);
\draw (0,1.1)  -- (0,2.5);
\draw[kernels2] (1.2,2.5) -- (1.2,0) node[not] {} -- (0.4,-0.6);
\draw (0,1.2) node[g] {};
\draw (0.6,2.5) node[g] {};
}

\DeclareSymbol{pre_im_Xi4eabc2}{2}{
\draw (0,-0.5) node[not] {} -- (-1,0.5) node[not] {};
\draw[kernels2] (-1,0.5) -- (-1.5,2);
\draw[kernels2] (0,-0.5)  -- (0.7,2);
\draw[kernels2] (-1,0.5) -- (-0.5,2);
\draw[kernels2] (0,-0.5) -- (1.5,2);
\draw (-1,2.2) node[g] {};
\draw (1,2.2) node[g] {};
}

\DeclareSymbol{pre_im_Xi4eabbisc2}{2}{
\draw[kernels2] (0,-0.5) node[not] {} -- (-1,0.5) node[not] {};
\draw[kernels2] (-1,0.5) -- (-1.5,2);
\draw[kernels2] (0,-0.5)  -- (0.7,2);
\draw[kernels2] (-1,0.5) -- (-0.5,2);
\draw (0,-0.5) -- (1.5,2);
\draw (-1,2.2) node[g] {};
\draw (1,2.2) node[g] {};
}

\DeclareSymbol{pre_im_I1Xi4abcc1}{2}{
\draw[kernels2] (0,-1) node[not] {} -- (-1,0) node[not] {};
\draw[kernels2] (0,1.1) node[not] {} -- (-1,0);
\draw[kernels2] (-1,0) -- (-1.5,2.5);
\draw[kernels2] (0,1.1) node[not] {} -- (-0.5,2.5);
\draw[kernels2] (0,1.1) -- (0.5,2.5);
\draw[kernels2] (0,-1) -- (1.5,2.5);
\draw (-1,2.7) node[g] {};
\draw (1,2.7) node[g] {};
}

\DeclareSymbol{pre_im_Xi4eabc1}{2}{
\draw (0,-0.5) node[not] {} -- (-1,0.5) node[not] {};
\draw[kernels2] (-1,0.5) -- (-1.5,2);
\draw[kernels2] (0,-0.5)  -- (-0.5,2);
\draw[kernels2] (-1,0.5) -- (0.8,2);
\draw[kernels2] (0,-0.5) -- (1.5,2);
\draw (-1,2.2) node[g] {};
\draw (1,2.2) node[g] {};
}

\DeclareSymbol{pre_im_Xi4ba1b}{2}{
\draw[kernels2] (0,0) node[not] {}  -- (1.8,1.5);
\draw[kernels2] (0,0) -- (0.8,1.5);
\draw (0,-0.1) -- (-1.8,1.5);
\draw (0,-0.1) -- (-0.8,1.5);
\draw (-1,1.7) node[g] {};
\draw (1,1.7) node[g] {};
}

\DeclareSymbol{pre_im_Xi4ba2}{2}{
\draw (0,-0.1) node[not] {}  -- (1.8,1.5);
\draw[kernels2] (0,0) -- (0.8,1.5);
\draw (0,-0.1) -- (-1.8,1.5);
\draw[kernels2] (0,0) -- (-0.8,1.5);
\draw (-1,1.7) node[g] {};
\draw (1,1.7) node[g] {};
}

\DeclareSymbol{pre_im_Xi4cabc2}{2}{
\draw[kernels2] (0,-0.5) node[not] {} -- (-1,0.5) node[not] {};
\draw[kernels2] (-1,0.5) -- (-1.5,2);
\draw (-1,0.5)  -- (0.7,2);
\draw[kernels2] (-1,0.5) -- (-0.5,2);
\draw[kernels2] (0,-0.5) -- (1.5,2);
\draw (-1,2.2) node[g] {};
\draw (1,2.2) node[g] {};
}

\DeclareSymbol{pre_im_Xi4cabc1}{2}{
\draw[kernels2] (0,-0.5) node[not] {} -- (-1,0.5) node[not] {};
\draw (-1,0.5) -- (-1.5,2);
\draw[kernels2] (-1,0.5)  -- (0.7,2);
\draw[kernels2] (-1,0.5) -- (-0.5,2);
\draw[kernels2] (0,-0.5) -- (1.5,2);
\draw (-1,2.2) node[g] {};
\draw (1,2.2) node[g] {};
}

\DeclareSymbol{pre_im_Xi4eabbisc1}{2}{
\draw[kernels2] (0,-0.5) node[not] {} -- (-1,0.5) node[not] {};
\draw[kernels2] (-1,0.5) -- (-1.5,2);
\draw[kernels2] (0,-0.5)  -- (-0.5,2);
\draw[kernels2] (-1,0.5) -- (0.8,2);
\draw (0,-0.5) -- (1.5,2);
\draw (-1,2.2) node[g] {};
\draw (1,2.2) node[g] {};
}

\DeclareSymbol{pre_im_1}{0}{
\draw[kernels2] (0,-0.5) node[not] {} -- (-0.6,0.5) ;
\draw[kernels2] (0,-0.5) -- (0.6,0.5);
\draw (0,1.1)  -- (-0.55,2);
\draw (0,1.1)  -- (0.55,2);
\draw (0,0.7) node[g] {};
\draw (0,2.2) node[g] {};
}

\DeclareSymbol{disconnect}{0}{
\draw[kernels2] (0,-0.5) node[not] {} -- (-0.6,0.5) ;
\draw[kernels2] (0,-0.5) -- (0.6,0.5);
\draw (-0.55,1.1)  -- (-0.55,2.3);
\draw (0.55,2.3) -- (0.55,1.5) -- (1.2,1.5) -- (1.2,3.5) -- (0.55,3.5) -- (0.55,2.7);
\draw (0,0.7) node[g] {};
\draw (0,2.5) node[g] {};
}

\DeclareSymbol{pre_im_2}{2}{\draw[kernels2] (0,0) node[not] {} -- (-1,1) node[not] {};
\draw[kernels2] (0,0) -- (1,1) node[not] {};
\draw[kernels2] (-1,1) -- (-1.5,2.5);
\draw[kernels2] (-1,1) -- (-0.5,2.5);
\draw[kernels2] (1,1) -- (0.5,2.5);
\draw[kernels2] (1,1) -- (1.5,2.5);
\draw (-1,2.7) node[g] {};
\draw (1,2.7) node[g] {};
}

\DeclareSymbol{CX_rec}{0}{
\draw [black] (-0.3,1) to (-0.3,-0.3);
\draw [black] (0.3,1) to (0.3,-0.3);
\draw [black] (-0.3,1) to (-0.3,2.3);
\draw [black] (0.3,1) to (0.3,2.3);
\draw (0,1) node[rec] {};
}

\DeclareSymbol{CX_cerc}{0}{
\draw [black] (0,1) to (0,-0.3);
\draw (0,1) node[cerc] {};
}

\DeclareMathAlphabet{\mathpzc}{OT1}{pzc}{m}{it}

\def\eqref#1{(\ref{#1})}

        \newcommand{\cnab}{\langle \nabla\rangle_m}

\makeatletter 
\newcommand*{\bigcdot}{}
\DeclareRobustCommand*{\bigcdot}{%
  \mathbin{\mathpalette\bigcdot@{}}%
}
\newcommand*{\bigcdot@scalefactor}{.5}
\newcommand*{\bigcdot@widthfactor}{1.15}
\newcommand*{\bigcdot@}[2]{%
  \sbox0{$#1\vcenter{}$}
  \sbox2{$#1\cdot\m@th$}%
  \hbox to \bigcdot@widthfactor\wd2{%
    \hfil
    \raise\ht0\hbox{%
      \scalebox{\bigcdot@scalefactor}{%
        \lower\ht0\hbox{$#1\bullet\m@th$}%
      }%
    }%
    \hfil
  }%
}
\makeatother

\tcbset
{colframe=boxcolor,colback=symbols!7!pagebackground,coltext=pageforeground,
fonttitle=\bfseries,nobeforeafter,center title,size=fbox,boxsep=1.5pt,
top=0mm,bottom=0mm,boxsep=0mm,tcbox raise base}

\def\two{{\<generic>\kern0.05em\<genericb>}}
\def\twoI{{\<Ito>\kern0.05em\<Itob>}}

\def\mail#1{\burlalt{#1}{mailto:#1}}

\usepackage{thmtools}

\declaretheorem[style=definition]{example}

\begin{document}

\title{Low regularity integrators via decorated trees}
\author{Yvonne Alama Bronsard$^1$, Yvain Bruned$^2$, Katharina Schratz$^1$}
\institute{LJLL (UMR 7598), Sorbonne University \and University of Edinburgh\\
Email:\ \begin{minipage}[t]{\linewidth}
\mail{yvonne.alama_bronsard@upmc.fr},\mail{Yvain.Bruned@ed.ac.uk}, \mail{katharina.schratz@sorbonne-universite.fr}.
\end{minipage}} 

\maketitle 

\begin{abstract}
We introduce a general framework of low regularity integrators which allows us to approximate the time dynamics of a large class of equations, including  parabolic and hyperbolic problems, as well as dispersive equations, up to arbitrary high order on general domains. The structure of the local error of the new schemes  is driven by nested commutators which in general require (much)  lower regularity assumptions than classical methods do. Our main idea   lies in embedding the central oscillations of the nonlinear PDE into the numerical discretisation. The latter is achieved by a novel decorated tree formalism  inspired by singular SPDEs with Regularity Structures and allows us to control the nonlinear interactions in the system up to arbitrary high order on the infinite dimensional (continuous) as well as finite dimensional (discrete) level. 
\end{abstract}
\setcounter{tocdepth}{2}
\setcounter{secnumdepth}{4}
\tableofcontents
\section{Introduction}

We consider  a general class of evolution equations under the form 
\begin{equs}\label{ev}
\partial_t u_o - \mathcal{L}_o u_o  = \sum_{\mathfrak{l} \in \Lab_-}\Psi_{o}^{\mathfrak{l}}(\mathbf{u}_{o}^{\mathfrak{l}}) V_{\mathfrak{l}}(x), \quad (t,x) \in  \R \times \Omega, \quad o \in \Lab_+, 
\end{equs}
where $ \Lab_-, \Lab_+ $ are finite sets, $ \Omega \subseteq \R^{d} $,  $ \mathbf{u}_{o}^{\mathfrak{l}} = (u_{\mathcal{o}})_{\mathcal{o} \in \Lab_+^{o,\mathfrak{l}}} $ and $ \Lab_+^{o,\mathfrak{l}} \subset \Lab_+ $.
For every $ (o,\mathfrak{l}) \in \Lab_+ \times \Lab_-$, $ \mathcal{L}_o $ denotes a linear (possibly) unbounded operator,  $\Psi_o^{\mathfrak{l}}$ represents the nonlinearity and $V_{\mathfrak{l}}$ the potential or noise.  The precise assumptions on the operator $\mathcal{L}_o$, the nonlinearity $\Psi_o^{\mathfrak{l}}$ and the potential $ V_{\mathfrak{l}} $ are stated in  Section \ref{sec:lin} below.  We add initial conditions 
$u_o(0) = v_o 
$ 
and when $ \partial \Omega \neq \emptyset $  suitable boundary conditions  which will be encoded in the domain of  the operator  $\mathcal{L}_{o}$.

Evolution equations of type \eqref{ev} are meanwhile  extensively studied in numerical analysis literature and a large variety of  discretisation techniques for their time resolution was proposed, reaching, e.g., from splitting  methods over exponential   integrators up to Runge--Kutta and Lawson   type schemes \cite{Faou12,HNW93,HLW,HochOst10,HLO20,HLRS10,LR04,McLacQ02,SanBook}.  
 While  such \emph{classical} discretisation techniques   provide  a good approximation to smooth solutions,   they often drastically break down at low regularity:   Rough  data and high oscillations in general cause  severe  loss of convergence which leads to huge computational costs hindering in many situations reliable approximations. Nonlinear partial differential equations (PDEs)  at low regularity are at large an ongoing challenge in numerical analysis.
 
In this work we introduce a   general framework of low regularity integrators which allows us to approximate the time dynamics of   \eqref{ev}  up to arbitrary high order under lower regularity assumptions than classical methods, such as Runge--Kutta, splitting, exponential integrators or Lawson schemes, require.   Our new framework  greatly enhances the low regularity framework recently introduced in~\cite{BS} which is restricted to  dispersive equations with periodic boundary conditions and polynomial nonlinearities $ u^\ell \overline u^m$.  Within our new framework we in particular overcome periodic boundary conditions and cover a much larger class of equations,  including for instance parabolic, and hyperbolic problems, as well as dispersive equations. Furthermore, our model   \eqref{ev} allows for noise or potentials and non polynomial nonlinearities.  

In order to bypass the limitations of  \cite{BS} we  introduce novel \emph{commutator structures}. Such commutators were previously used in \cite{FS} for  the derivation of  first- and second-order low regularity integrators for a simplified version of \eqref{ev}  (in particular without potential, i.e., $V_{\mathfrak{l}} \equiv 1$). 
At low order the central oscillations in   \eqref{ev} can  quite easily be computed  and their dominant parts can be extracted ``by hand" (see \cite{FS}). This is however no longer the case at higher order: The control of the underlying oscillations of \eqref{ev} up to arbitrary high order involves a scale of nested, iterated oscillatory integrals depending on the nonlinear interaction between the leading differential operator $ \mathcal{L}_o$ and the nonlinearity $\sum_{\mathfrak{l} \in \Lab_-}\Psi_{o}^{\mathfrak{l}}(\mathbf{u}_{o}^{\mathfrak{l}}) V_{\mathfrak{l}}(x)$ which is in particular challenging on the finite dimensional (discrete)  level.


In order to overcome this, our key idea lies in developing a new class of decorated trees. This formalism is crucial as it allows us to  encode the oscillatory systems of iterated integrals on the continuous as well as on the discrete level. We  control the local error at each order with suitable nested commutators. The decorated trees used in this work are inspired by \cite{BS}. We develop a new framework overcoming the limitations of the formalism developed in  \cite{BS} which heavily relied on Fourier expansion techniques, periodic boundary conditions and polynomial nonlinearities. Our  decorated trees are close in spirit to the ones used in singular SPDEs in \cite{BCCH} where decorations on the nodes encode the drivers/potentials of the equation~\eqref{ev}. 
This work is the second example after \cite{BS} where Regularity Structures decorated trees introduced in \cite{reg,BHZ} appear in context of  numerical analysis which stresses their robustness.
Another key difference with~\cite{BS} lies in the fact that our setting does not allow for a Hopf algebra to conduct the local error analysis. The latter seems  limited to Fourier space and the control of the corresponding Fourier modes. The lack of a Hopf algebra in this work is, however, not  surprising as we push forward low regularity integrators to the much more general setting \eqref{ev}. 


Our main result is the new general low regularity scheme presented in Definition \ref{scheme}  with its error structure given in Theorem \ref{thm:genloc}. The approximation relies crucially on the formulae given in Definition~\ref{def_CK}. We will illustrate our framework on the concrete examples of the Gross--Pitaevskii and the Sine--Gordon equation in full detail   in Section \ref{sec:examples}.
 In the next subsection, we illustre the main ideas of our scheme on a parabolic equation.

\subsection{Linearization and decorated trees}\label{sec:lin}
Before introducing our linearization techniques based on decorated trees, let us first specify the assumptions on the nonlinearity $\Psi$ and on the leading operator $ \mathcal{L}_o$ appearing in \eqref{ev}.

\begin{assumption}[nonlinearity $\Psi$]\label{assump:psi}
    In the following we assume that the  nonlinearity  $ \Psi^{\mathfrak{l}}_{o}(\mathbf{u}^{\mathfrak{l}}_o)$ is tensorized  under the form
\begin{equation}
\label{nonlin}
\Psi_{o}^{\mathfrak{l}}(\mathbf{u}^{\mathfrak{l}}_o) = \mathcal{B}^{\mathfrak{l}}_{o}\left(\prod_{\mathcal{o} \in \Lab_+^{\mathfrak{l},o}} f^{\mathfrak{l}}_{o,\mathcal{o}}(u_{\mathcal{o}})\right), \quad  f^{\mathfrak{l}}_{o,\mathcal{o}} : \, \C \rightarrow \C^J, 
\end{equation}
where $ \left( \mathcal{B}^{\mathfrak{l}}_o \right)_{(o,\mathfrak{l)} \in \Lab_+ \times \Lab_-} $ is a family of (linear) operators and  we use  the  notation 
 $ \prod_i X^{(i)} = \sum_{k \le J}\prod_i X^{(i)}_{k}$ where $ X^{(i)} \in \C^J$.
\end{assumption}

\begin{assumption}[leading operator $ \mathcal{L}_o$]\label{assump:L}
The linear  operator $ \mathcal{L}_o$ shall be defined on a Hilbert space
   $X$  of complex valued  functions $u \in \C$ with norm denoted by $\| \cdot \|$ and domain $D(\mathcal{L}_o)$. To make sense of the exponential functions, we in addition   assume that $ \mathcal{L}_o$ generates a strongly continuous semigroup $\{e^{t \mathcal{L}_o}\}_{t\geq 0}$ of contractions  on $X$.
\end{assumption}
We illustrate the main idea of our general scheme via a simple example that contains already some important challenges which could not be overcome in previous works \cite{BS,FS}. Let us consider the abstract parabolic equation:
\begin{equs} \label{simple_example}
\partial_t u  -  \Delta u = f(u) V, \quad u_0 = v, \quad  (t,x) \in  \R \times \Omega
\end{equs}
with $\Omega \subseteq \R^d$ sufficiently smooth. In the case  of $\partial \Omega \neq \emptyset $ we assume
 suitable boundary conditions encoded in $D(\Delta)$  and shall denote by $V$ a potential or noise.

 The starting point of our new class of schemes is based on Duhamel's formula for \eqref{simple_example} which is given by
\begin{equs}\label{Ek1}
u(t) = e^{t \Delta} v + \int_{0}^{t} 
e^{(t-\xi)\Delta} f(u(\xi)) V d\xi.
\end{equs}
Classical methods are in general based on  Taylor series expansions of  $u(\xi)$ (the solution of  \eqref{simple_example} at time $t=\xi$)  around the initial value $u_0=v$ in the sense that for small $\xi$ we have at first order that
\begin{equs}\label{Tclass}
u(\xi) = v + \mathcal{O}(\xi  u').
\end{equs}
This classical Taylor series expansion, however, requires regularity in the solution since from the PDE \eqref{simple_example} we have
$$
\mathcal{O}(\xi  u') = \mathcal{O}\left(\xi \Delta u \right).
$$
Hence, the numerical scheme will only converge to order one for sufficiently smooth solutions
$$
u \in D(\Delta).
$$
The situation is even worse for higher order methods which are in general based on higher order Taylor series expansions of the solution
\begin{equs}\label{Thigh}
u(\xi) = v + \xi u'(0) + \frac{\xi^2}{2} u''(0) + \ldots + \frac{\xi^r}{r!} u^{(r)}(0)+  \mathcal{O}(\xi^{r+1} u^{(r+1)}).
\end{equs}
By using the PDE \eqref{simple_example} we have at leading order
$$
\mathcal{O}(\xi^{r+1} u^{(r+1)}) = \mathcal{O}\left(\xi^{r+1} \Delta^{r+1} u \right).
$$
Hence, it follows that at order $r+1$ the necessary regularity assumptions 
on the solutions is increased to
\begin{equs}\label{Reghigh}
u \in D(\Delta^{r+1}).
\end{equs}
\begin{remark}\label{re:comp}
Note that \eqref{Reghigh} does not only require sufficiently smooth solutions (i.e., that the $\partial^{2(r+1)}$th  derivative is bounded), but also additional  compatibility conditions   on the boundary whenever $\partial \Omega \neq \emptyset$. For example, for a second order approximation, namely when $r=1$, the classical Taylor expansion \eqref{Thigh} requires 
$$
u\in D(\Delta^2) = \left \{f \in D(\Delta) \,;\, \Delta f \in D(\Delta)\right\}.
$$
In case of $\Omega$ a bounded smooth open set of $\R^d$  equipped with homogeneous Dirichlet boundary conditions this means that
\begin{equs}\label{D^2}
u \in D(\Delta^2) =\left \{f \in (H^2\cap H^1_0)(\Omega)\,;\, \Delta f \in (H^2\cap H^1_0)(\Omega)\right\}.
\end{equs}
Hence,  zero trace of the solution $u$ \emph{and} $\Delta u$ on the boundary is imposed.
\end{remark}
In this manuscript we want to overcome these high regularity assumptions by developing a new class of schemes with improved local error structures.
It is essential to note that if we want to develop such a class of  (high order) methods under low(er) regularity assumptions 
$$
u \in Y \supset  D(\Delta^{r+1})
$$
we cannot base our new schemes on  classical Taylor series expansions of type \eqref{Thigh}.

Our idea to overcome  the regularity assumptions \eqref{Reghigh} lies in embedding the central oscillations of the PDE into the numerical discretisation, while only approximating the lower order terms. More precisely, iterating Duhamel's formula (see also \cite{BS,FS}) we have thanks to \eqref{Ek1} at time $t=\xi$ that
\begin{equs}
u(\xi) = e^{\xi \mathcal{L}} v + \mathcal{O}(\xi)
\end{equs}
where the central oscillations are captured in the term
$$
e^{\xi \mathcal{L}}v
$$
and the remainder $ \mathcal{O}(\xi)$ does not involve any derivatives on the solution.  In detail one has for $ u(\xi) =  e^{\xi \Delta} v
+ A(\xi) $ 
that
\begin{equs} \label{main_equation}
f(u(\xi)) = \sum_{k \leq r} \frac{A(\xi)^{k}}{k!}
f^{(k)}(e^{\xi \Delta} v) + \mathcal{O}(A(\xi)^{r+1})
 \end{equs}
 where we truncate the expansion at  $ r $ which is associated to the order of the scheme one wants to achieve. The remainder $ A(t) $ will corresponds in practice (after a suitable approximation) to a finite sum of iterated integrals only involving bounded operators. 
 
 Plugging \eqref{main_equation} into \eqref{Ek1}, we obtain the following expansion of the exact solution
\begin{equs}\label{new_duhamel}
u(t) = e^{t \Delta} v + \sum_{k \leq r} \int_{0}^{t} 
e^{(t-\xi)\Delta}  \frac{A(\xi)^{k}}{k!}
f^{(k)}(e^{\xi \Delta} v)   V d\xi + \mathcal{O}(t^{r+2}),
\end{equs} 
where the central oscillations are embedded in the interaction of 
$$
e^{(t-\xi)\Delta} \quad \text{and}\quad f^{(k)}(e^{\xi \Delta} v) 
$$
and the remainder $ \mathcal{O}(t^{r+2})$ does, in contrast to classical approximation techniques \eqref{Thigh},   not involve any derivatives on the solution.

The main challenge lies in controlling the oscillatory integrals in \eqref{new_duhamel}
\begin{equs}\label{osciInt}
 \sum_{k \leq r} \int_{0}^{t} 
e^{(t-\xi)\Delta}  \frac{A(\xi)^{k}}{k!}
f^{(k)}(e^{\xi \Delta} v)   V d\xi 
\end{equs} 
on the discrete level. Note that a classical Taylor series expansion of the oscillations $e^{\xi \Delta}$ would again lead to the high regularity assumptions \eqref{Reghigh}. Our key  idea to overcome this lies in  embedding the central oscillations in  \eqref{osciInt} explicitly into our numerical discretization. In order to do so,  in a first step we have to linearise the terms $ f^{(k)}(e^{\xi \Delta} v) $. This will allow us to filter out the dominant oscillations while controlling the approximation error by commutators $ \mathcal{C}[\cdot,\cdot] $ with an improved error structure introduced as follows:
 
\begin{definition}\label{def:comm}
 For a function $ H(v_1,...,v_n) $, $ n \geq 1 $ and a linear operator $ \mathcal{L} $ we set\begin{equs}
 \mathcal{C}\left[H, \mathcal{L}\right](v_1,...,v_n) = - \mathcal{L}(H(v_1,...,v_n)) + \sum_{i=1}^{n} D_i H(v_1,...,v_n) \cdot \mathcal{L} v_i 
 \end{equs} where $ D_i H $ stands for the partial differential of $ H $ with respect to the variable $ v_i $.
We  define the iterated commutator $ \mathcal{C}^{k}\left[ H , \mathcal{L}\right] $ recursively by
 \begin{equs}
  \mathcal{C}^{k+1}\left[ H , \mathcal{L}\right] =  \mathcal{C}^{k}\left[ \mathcal{C}\left[ H,\mathcal{L} \right] , \mathcal{L}\right], \quad \mathcal{C}^{0}\left[ H , \mathcal{L}\right] = H .
 \end{equs}
\end{definition}

Let us illustrate this definition with our example \eqref{simple_example}. With the above notation (and $ \mathcal{L} = \Delta $) one has for  sufficiently smooth  $ f $:
 \begin{equs} \label{commutator} 
 \begin{aligned}
 f(e^{\xi \Delta} v) & =  \sum_{\ell = 0}^{r} \frac{\xi^{\ell}}{\ell!} e^{\xi \Delta} \mathcal{C}^{\ell}[f,\Delta](v)  + \mathcal{O}(\xi^{r+1} \mathcal{C}^{r+1}[f,\Delta](v)).
 \end{aligned}
 \end{equs}
 This linearisation produces an error of the form $$ \mathcal{C}^{r+1}[f,\Delta](v) $$ which in general requires less regularity and less compatibility conditions on the boundary than does a classical Taylor-series expansion  of the propagator $ e^{\xi \Delta} $. Indeed, up to order $r$ we have
 \begin{equs}\label{TaylorExp}
 e^{\xi \Delta} = \sum_{k \leq r} \frac{\xi^k}{k!} \Delta^k + \mathcal{O}( \xi^{r+1} \Delta^{r+1}),
 \end{equs}
whereby
$$
  D\left( \mathcal{C}^{r+1}[f,\Delta] \right  ) \supset D( \Delta^{k+1}).$$
  Hence, the commutator expansion \eqref{commutator} requires less regularity assumptions and compatibility conditions than classical Taylor expansion of the the propagator $ e^{\xi \Delta} $. We illustrate this through an example in the following two remarks. This observation  also holds true in the general setting \eqref{ev}.
 \begin{remark}\label{rem:domi}
 In general it holds that
 $$D(\mathcal{C}[f,\mathcal{L}] ) \supset D(\mathcal{L}).$$
  Let us for instance recall the example given in \cite[Sec. 2.1]{FS} where $ \mathcal{L} = \Delta $ on the torus $ \T^{d} $ and $ f(u) = u^2 $. Then one has $ D(\mathcal{L}) = H^{2} $ (the classical Sobolev space) and
 \begin{equs}
 \mathcal{C}[f,\mathcal{L}](v) = - \sum_{k=1}^{d} (\partial_k v)^2.
 \end{equs}
If we choose to measure the error in $ L^2 $, in order to bound the above expression we need $ (\partial_k v)^2 \in L^2 $. This follows for $v\in H^{1+\frac{d}{4}}$ by using the injection $H^{1+\frac{d}{4}}\hookrightarrow L^4$. Hence we have $ v \in H^{1+ \frac{d}{4}} \subset D(\mathcal{C}[f,\mathcal{L}]) $.

 We note that the above also holds true when working on a bounded domain of $\R^d$, equipped with homogeneous Dirichlet boundary conditions. Indeed, the same argument as above holds with $D(\Delta) = (H^2 \cap H^1_0)(\Omega)$, and where $v \in (H^{1+\frac{d}{4}} \cap H^1_0)(\Omega) \subset D(\CC[f, \Delta])$.
 \end{remark}
 \begin{remark}
The  commutator term $\CC^{r+1}[f,\Delta](v)$ in general asks for less regularity and compatibility conditions at the boundary than  the classical Taylor's expansion \eqref{TaylorExp}. Indeed, in case of homogeneous Dirichlet boundary conditions for example when $r=1$, the commutator term $\CC^2[f, \Delta](v)$ involves derivatives of at most second order. Namely, formally we have $\CC^2[f, \Delta](v) = O(\Delta v)$, which requires $v \in D(\Delta) = (H^2\cap H^1_0)(\Omega)$. A. classical second-order Taylor's expansion on the other hand asks for $v\in D(\Delta^2)$, which  does not only requires $v$ to have zero trace on the boundary, but also its second derivatives (see also Remark \ref{re:comp}).
 \end{remark}
 
 Let us now turn back to presenting the main idea of our new schemes on the example \eqref{simple_example}: Plugging \eqref{commutator} into \eqref{new_duhamel}, we get
\begin{equs}\label{appi0}
u(t) = e^{t \Delta} v + \sum_{k,\ell \leq r}   \int_{0}^{t} 
e^{(t-\xi)\Delta}  \frac{A(\xi)^{k}}{k!}
 \frac{\xi^{\ell}}{\ell!} \left(  e^{\xi \Delta} \mathcal{C}^{\ell}[f^{(k)},\Delta](v) \right)   V d\xi \\ + \mathcal{O}(t^{r+2} \mathcal{C}^{r+1}[f^{(k)},\Delta](v))
\end{equs}
where the leading error term is driven by the commutator 
$$
\mathcal{O}(t^{r+2} \mathcal{C}^{r+1}[f^{(k)},\Delta](v)).
$$
In general, the latter is more important (i.e., worse in terms of regularity assumptions) than the ones introduce by the approximation of $ A(t) $ (cf. \eqref{main_equation}) which is a polynomial in the initial condition $ v $ involving only bounded operators. Also, one can mention that in general $ \mathcal{L}_o $ is more singular than the $ \CB_o^{\mathfrak{l}} $.

From the approximation \eqref{appi0}, we can collect iterated integrals for building an approximate solution $ w^r(v,t) $ up to order $ r $. For example, when $ r =2 $, we get 
\begin{equs}\label{2_order}
w^{2}(v,t) & =  e^{t \Delta} v +
\int_{0}^{t} 
e^{(t-\xi)\Delta}  \left(  e^{\xi \Delta} f(v) \right)   V d\xi \\ &  + \int_{0}^{t} \xi
e^{(t-\xi)\Delta}  \left(  e^{\xi \Delta} \mathcal{C}[f,\Delta](v) \right)   V d\xi \\ &  + 
\int_{0}^{t}
e^{(t-\xi)\Delta} \left( \int_{0}^{\xi} 
e^{(\xi-\xi_1)\Delta}  \left(  e^{\xi_1 \Delta} f(v) \right)   V d\xi_1 \right) \left(  e^{\xi \Delta} f^{(1)}(v) \right)   V d\xi.
\end{equs}
\begin{remark} \label{rem:Ch}
The remaining challenge lies in embedding the central oscillations triggered by
\begin{equs}
e^{(t-\xi)\Delta}  \left(  e^{\xi \Delta} f(v) \right)V, \quad e^{(t-\xi)\Delta}  \left(  e^{\xi \Delta} \mathcal{C}[f,\Delta](v) \right)V,\\
e^{(t-\xi)\Delta} \left( \int_{0}^{\xi} 
e^{(\xi-\xi_1)\Delta}  \left(  e^{\xi_1 \Delta} f(v) \right)   V d\xi_1 \right) \left(  e^{\xi \Delta} f^{(1)}(v) \right)V   
\end{equs}
into the numerical discretisation.
\end{remark} 

For this purpose we will introduce a recursive map $ \Pi   $ defined on a set of trees $ \CT $ to encode these iterated integrals. The set $ \CT $ will be  decorated trees whose formalism has been introduced in \cite{BHZ} and has been extensively used in \cite{BCCH} for giving wellposedness results for a large class of Stochastic partial differential equations (SPDEs). After the recent work \cite{BS}, it is the second time that one uses this formalism also in context of Numerical Analysis. The iterated integrals for $ w^2 $ are given by trees of size (having at most two edges) two $ \CT^{2} $:
\begin{equs}
\CT^{2} = \Big\lbrace \begin{tikzpicture}[scale=0.2,baseline=-5]
\coordinate (root) at (0,-1);
\coordinate (tri) at (0,1);
\draw[kernels2] (root) -- (tri);
\node[not] (rootnode) at (root) {};
\node[xi] (trinode) at (tri) {};
\end{tikzpicture}, \begin{tikzpicture}[scale=0.2,baseline=-5]
\coordinate (root) at (0,-1);
\coordinate (tri) at (0,1);
\draw[kernels2] (root) -- (tri);
\node[not] (rootnode) at (root) {};
\node[xix] (trinode) at (tri) {};
\end{tikzpicture}, \begin{tikzpicture}[scale=0.2,baseline=-5]
\coordinate (root) at (0,-1);
\coordinate (tri) at (0,1);
\coordinate (tri1) at (0,3);
\draw[kernels2] (root) -- (tri);
\draw[kernels2] (tri) -- (tri1);
\node[not] (rootnode) at (root) {};
\node[xi] (trinode1) at (tri) {};
\node[xi] (trinode) at (tri1) {};
\end{tikzpicture} \Big\rbrace.
\end{equs}
Brown edges encode integrals of the form $ \int^t_0 e^{(t-\xi) \Delta} \cdots d \xi $, white nodes the nonlinearity $ e^{\xi \Delta} f(v) V $ and the nodes marked with a cross correspond to $ \xi \left( e^{\xi \Delta} \mathcal{C}[f,\Delta](v) \right) V $. Incoming edges to a grey dot induce derivatives on $ f $ which will allow us to encode the third term in \eqref{2_order}.

Our decorated trees strongly differ from the ones introduced in \cite{BS}: The scheme introduced in \cite{BS} heavily relies on Fourier series expansion techniques and is therefore restricted to periodic boundary conditions. In particular, the tree structure in \cite{BS} is  based on Fourier decorations. The latter allowed us to easily filter out the dominant frequency interactions in the system   based on the underlying structure of resonances of the PDE. In order to handle the general setting \eqref{ev}, and in particular deal with non periodic boundary conditions and non polynomial nonlinearities, we can, however, not  make use of the previously proposed resonance analysis. Instead, we have to introduce new tools  based on the analysis of the operator interactions and an extensive  use of the commutator based approximation
\eqref{commutator}. Let us illustrate this analysis on the first iterated integral that we denote by $ \left( \Pi \, \begin{tikzpicture}[scale=0.2,baseline=-5]
\coordinate (root) at (0,-1);
\coordinate (tri) at (0,1);
\draw[kernels2] (root) -- (tri);
\node[not] (rootnode) at (root) {};
\node[xi] (trinode) at (tri) {};
\end{tikzpicture} \right) (t)$. It takes the following form
\begin{equs}
\left( \Pi \, \begin{tikzpicture}[scale=0.2,baseline=-5]
\coordinate (root) at (0,-1);
\coordinate (tri) at (0,1);
\draw[kernels2] (root) -- (tri);
\node[not] (rootnode) at (root) {};
\node[xi] (trinode) at (tri) {};
\end{tikzpicture} \right) (t) =   \int_{0}^{t} 
e^{(t-\xi)\Delta}  \prod_{i=1}^2 e^{\xi \mathcal{L}_i} u_i d\xi, \\ \mathcal{L}_1 = \Delta, \quad u_1 = f(v), \quad  \mathcal{L}_2 = 0, \quad u_2 = V. 
\end{equs}
The first step is to distribute $ e^{(t-\xi)\Delta} $ on the product $ \prod_{i=1}^2 e^{\xi \mathcal{L}_i} u_i $:
\begin{equs}\label{2_orderO}
e^{(t-\xi)\Delta} &  \left( \prod_{i=1}^2 e^{\xi \mathcal{L}_i} u_i \right)   = 
\prod_{i=1}^2 e^{\xi \mathcal{L}_i + (t-\xi) \Delta} u_i  \\ & + \sum_{1 \leq k \leq r} \frac{(t-\xi)^k}{k!}  \mathcal{C}^{k}[ (e^{t \Delta} \mathcal{M}_{\lbrace 1\rbrace})(e^{(t-\xi) \Delta} \mathcal{M}_{\lbrace 2\rbrace}), \Delta](u_1,u_2) \\ & + \mathcal{O} \left( t^{r+1} 
 \mathcal{C}^{r+1}[  \mathcal{M}_{\lbrace 1,2\rbrace}, \Delta](u_1,u_2)\right)
\end{equs}
where we have used the following notations: 
Let $ (u_i)_{i \in I} $ and $ (v_j)_{j \in J} $ be two finite sets of functions and $ \mathcal{L}_1 $ and $ \mathcal{L}_2 $ two linear operators. Then we set
\begin{equs}
(\mathcal{L}_1 \mathcal{M}_I \mathcal{L}_2 \mathcal{M}_{J})(u_{i \in I}, v_{j \in J}) := \mathcal{L}_1 \left( \prod_{i\in I} u_i \right) \mathcal{L}_2 \left( \prod_{j \in J} u_j \right)
\end{equs}
and 
\begin{equs}
 \mathcal{C}[ \mathcal{M}_{I}, \mathcal{L}_1](u_{i \in I}) := - \mathcal{L}_1 \mathcal{M}_I (u_{i \in I}) + \sum_{j \in I} \mathcal{M}_I ( ...,\mathcal{L}_1 u_j,... ).
\end{equs}
\begin{remark} 
We can notice that the term $ e^{- \xi \Delta} $ does, however, not make sense (recall that $\xi \geq 0$). 
In the dispersive setting of \cite{BS}, one did not have such an issue (as for instance   $ e^{ i \xi \Delta}$ is well defined for all $\xi \in \R$)  and hence one could perform the resonance analysis only on the terms depending on $ \xi $. In our general setting \eqref{ev} we have to be more careful and take the oscillations of the full operator $e^{(t- \xi) \Delta}$ into account.
\end{remark}

The second step is to identify the dominant part in the operators $ \mathcal{L}_i  - \Delta$ in \eqref{2_orderO}.  For this purpose we set $\CA_{i}= \mathcal{L}_i  - \Delta$ and introduce the splitting   
\begin{equs}
\CA_i = a_i \mathcal{L}_{\tiny{\text{dom}}} + \CA_{\tiny{\text{low}}}^i, \qquad D(\mathcal{L}_{\tiny{\text{dom}}}) \subset D(\mathcal{A}_{\tiny{\text{low}}}^i) 
\end{equs}
where $ a_i \in \lbrace 0,1 \rbrace $.
The purpose of this splitting lies in the fact that in the construction of our schemes we will separate the dominant from the lower order oscillations and only embed the dominant parts exactly into our numerical discretization, while Taylor expanding the lower order parts. Note that in the above example one for instance has that
\begin{equs}
\CA_1 = 0, \quad \CA_2 = - \Delta,
\end{equs}
and hence, $a_1 = 0, \ a_2 = 1$, $\CA_{\tiny{\text{low}}}^i = 0$ for $i\in \{1,2 \}$ and $\CL_{\tiny{\text{dom}}} = -\Delta$. In this simple case one does  not need any further approximation and  one can  write the following  approximation at order $ r $:
\begin{equs}
 \left( \Pi^r \, \begin{tikzpicture}[scale=0.2,baseline=-5]
\coordinate (root) at (0,-1);
\coordinate (tri) at (0,1);
\draw[kernels2] (root) -- (tri);
\node[not] (rootnode) at (root) {};
\node[xi] (trinode) at (tri) {};
\end{tikzpicture} \right) (t) =  \sum_{0 \leq k \leq r} \int_{0}^{t} \frac{(t-\xi)^k}{k!}  \mathcal{C}^{k}[ (e^{t \Delta} \mathcal{M}_{\lbrace 1\rbrace})(e^{(t-\xi) \Delta} \mathcal{M}_{\lbrace 2\rbrace}),\Delta](u_1,u_2) d \xi.
\end{equs}

\begin{remark} In most applications, one does not have to go through the second step of the approximation like the example shown above.
The second step is needed for linear combination of operators like $\partial_x^2 + \partial_x$. One may have to approximate the following iterated integral (with notation $\Delta = \partial_x^2$):
\begin{equs}
\int_{0}^{t} 
e^{(t-\xi)\Delta}  \left(  e^{\xi (\Delta+ \partial_x)} v_2 \right)   V d\xi  .
\end{equs}
This integral is coming from a system of the form:
\begin{equs} \label{equation_A_i}
\partial_t u_1  -  \Delta u_1  & = u_2  V, \quad u_1(0) = v_1, 
\\ \partial_t u_2  -  \Delta u_2 - \partial_x u_2 & =0, \quad u_2(0) = v_2. 
\end{equs}
Then, from \eqref{2_orderO}, we get
\begin{equs}
\sum_{1 \leq k \leq r} \frac{(t-\xi)^k}{k!}  \mathcal{C}^{k}[ (e^{t \Delta + \xi \partial_x } \mathcal{M}_{\lbrace 1\rbrace})(e^{(t-\xi) \Delta} \mathcal{M}_{\lbrace 2\rbrace}), \Delta](v_2,V).
\end{equs}
Now, $\CA_{\tiny{\text{low}}}^1 = \partial_x
$ and one has to Taylor-expand the term $ e^{\xi \partial_x} $. This leads to the following approximation:
\begin{equs}
\int_{0}^t e^{(t-\xi)\Delta}  \left(  e^{\xi (\Delta+ \partial_x)} v_2 \right)   V d\xi &  = \int_{0}^t \sum_{0 \leq k + \ell \leq r} \frac{(t-\xi)^k}{k!} \frac{\xi^{\ell}}{\ell !} \\ & \mathcal{C}^{k}[ (e^{t \Delta } \mathcal{M}_{\lbrace 1\rbrace})(e^{(t-\xi) \Delta} \mathcal{M}_{\lbrace 2\rbrace}), \Delta](\partial_x^{\ell} v_2,V) d \xi
\\ & + \sum_{k+ \ell = r +1}\mathcal{O} \left( t^{r+2} 
 \mathcal{C}^{k}[  \mathcal{M}_{\lbrace 1,2\rbrace}, \Delta](\partial_x^{\ell} v_2,V)\right).
\end{equs}
 \end{remark}
In the next section, we will present our main strategy in the general setting \eqref{ev}.
\subsection{Main strategy of the general numerical scheme}

Duhamel's formulation of \eqref{ev} is given by 
\begin{equs} \label{duhamel_main}
u_o(\tau) = e^{\tau \mathcal{L}_o} v_o + \int_{0}^{\tau} 
e^{(\tau-\xi)\mathcal{L}_o} 
\sum_{\mathfrak{l} \in \Lab_-}\Psi_{o}^{\mathfrak{l}}(\mathbf{u}_{o}^{\mathfrak{l}}) V_{\mathfrak{l}}
d\xi.
\end{equs}
By iterating this formulation together with \eqref{commutator} and \eqref{main_equation}, we get iterated integrals that can be represented via decorated trees. 
   Let $ r + 1$ be the order of the scheme, we first truncate this expansion and get  the following approximation of the solution: 
\begin{equs} \label{decoratedV1}
w^r_o(\hat{\mathbf{v}},\tau) = e^{\tau \mathcal{L}_o} v_o + \sum_{T \in \CV^r}  \left( \Pi T \right)(\hat{\mathbf{v}},\tau)
\end{equs}
where $ \hat{\mathbf{v}}=(\mathbf{v},\mathbf{V}) $, $ \mathbf{v} = (v_o)_{o \in \Lab_+} $, $ \mathbf{V}=  (V_{\mathfrak{l}})_{\mathfrak{l}\in \Lab_-} $, $ \CV^r$ is a finite set of decorated trees, and $ (\Pi T)(\hat{\mathbf{v}},\tau) $ represents the iterated integral associated to $ T $. The exponent $r$ in $ \CV^r $ means that we consider only 
trees of size $ r $ which are the trees producing an iterated integral with $ r $ integrals. Proposition~\ref{exact_solution_order_r} shows that $ w^r_o $ solves \eqref{duhamel_main} with an error involving $ \mathcal{C}^{r+1}[f,\mathcal{L}_o] $ 
with $ f $ a function coming from the coefficients $ \Psi_o $. The main improvement lies in the error structure:  The error $ \mathcal{C}^{r+1}[f,\mathcal{L}_o] $  is in general  better than $ \mathcal{L}_o^r $ in the sense of requiring lower regularity assumptions on the solution.

 The main difficulty boils down to developing  for every $T \in \CV^r$ a suitable approximation to the iterated integrals $ (\Pi T)(\hat{\mathbf{v}},\tau) $ with the aim of minimising the local error structure (in the sense of regularity), see also Remark \ref{rem:Ch}. In order to achieve this, the key idea lies in commutator based approximations
\eqref{commutator} which replace the resonance analysis given in \cite{BS}.
 The approximation of $ (\Pi T) $ is given by a new map on decorated trees denoted by $\Pi^{r}_{A}$ where $r+1$ is the order of the scheme and $A$ is a domain corresponding to the a priori  assumed regularity of $\hat{\mathbf{v}}$. 
Our general scheme  takes the form
\begin{equs} \label{decoratedV2}
w^r_{o,A}(\hat{\mathbf{v}},\tau) & = e^{\tau \mathcal{L}_o} v_o + \sum_{T \in \CV^r}  \left( \Pi_{A}^{r} T \right)(\hat{\mathbf{v}},\tau)
\end{equs}
where the map $ \Pi^{r}_{A} T $ is  a low regularity  approximation of order $ r $ of the map $ \Pi T $ in \eqref{decoratedV1} in the sense that
\begin{equs}\label{eq:loci}
\left(\Pi T - \Pi^{r}_{A} T \right)(\hat{\mathbf{v}},\tau)  = \mathcal{O}\left( \tau^{r+2} \mathcal{L}^{r}_{\text{\tiny{low}}}(T,\hat{\mathbf{v}},A) \right).
\end{equs}
Here $\mathcal{L}^{r}_{\text{\tiny{low}}}(T,\hat{\mathbf{v}},A) $
 involves nested commutators that require in general less regularity than  powers of the full operator $ \mathcal{L}_o^{r} $ (see also Remark \ref{rem:domi}). The scheme \eqref{decoratedV2} and the local error approximations~\eqref{eq:loci} are the main results of this work (see Theorem~\ref{thm:genloc}).
The approximation $ \Pi^{r}_A $ is constructed from a character $ \Pi_A $ defined on the vector space $ \CH $ spanned by  decorated forests taking values in  a space $ \CC $ that will depend on the operators $ \mathcal{L}_o $ and $ \hat{\mathbf{v}} $. We add a decoration $ r $ at the root with the meaning that we will consider an  approximation of order $r$. This is performed by the symbol $ \mathcal{D}^r $ that gives $ \Pi_A \CD^r(T)  = \Pi^{r}_A T$.

As in \cite{BS}, we define the map $ \Pi_A $ recursively from an operator $ \CK $. This operator  computes a suitable approximation (matching the regularity of the solution) of the integrals introduced by the iteration of Duhamel's formula.
Here the main difference is that we perform the analysis directly on the operator without going into  Fourier space. We single out the dominant parts of the operators we are dealing with. These dominant parts are integrated exactly and only  the lower order terms are approximated which allows for an improved local error structure compared to classical approaches. 

For the local error analysis, a structure of Hopf algebra and comodule  was used in \cite{BS} for singling out oscillations. It was based on a variant of the Butcher-Connes-Kreimer coproduct \cite{Butcher72,CK,CKI} inspired by \cite{BHZ}. This structure can be seen  as a deformation of the Butcher-Connes-Kreimer coproduct (see \cite{BM}).
The central object  used as in Quantum Field Theory for the local error analysis in \cite{BS} was a Birkhoff factorisation similar to the one presenting in \cite{BHZ} for recentering iterated integrals. This is put into perspective in the work \cite{BE}. The context of commutators is much more involved than the one in Fourier space.
We therefore do not use such a construction but rather the recursive formulae  for computing the local error. It is given in Definition~\ref{def:Llow}. Via Proposition~\ref{decomp_Pi}, we illustrate what are the dominant operators involved in the computation of the local error.

\subsection{Outline of the paper}

Let us give a short review of the content of this paper.
In Section~\ref{Sec::2}, we introduce the general algebraic framework by first defining a suitable vector space of decorated trees $ \hat{\CT} $ and decorated forests $ \hat \CH $. Next, we set how to compute the dominant part of a set of operators (see Definition~\ref{def:Dom}). We define then the dominant operators associated to a decorated forest (see Definition~\ref{dom_freq}). 
 Then, we introduced new spaces of decorated trees $ \CT $  that we call approximated decorated trees. They carry an extra decoration $ r $ at the root and they represent approximation at order $ r $ of the corresponding iterated integrals. We also introduced new decorated forests based on the same construction.

In Section~\ref{sec::Iterated integrals}, we construct the approximation of the iterated integrals given by the character $ \Pi : \hat \CH \rightarrow  \CC  $ (see \eqref{recursive_Pi}) through the character $ \Pi_A : \CH \rightarrow \CC $ (see \eqref{Pi_A}) where $ A $ is a domain that is the regularity of $ \hat{\mathbf{v}} $. This is the regularity assumed a priori before writing the scheme. The approximation $ \Pi_A $ relies on a recursive construction where the operator  $ \CK $ given in Definition~\ref{def_CK} is heavily used.  The local error analysis which is the error estimate on the difference  between $ \Pi  $ and its approximation $ \Pi_A $ is given in  Theorem~\ref{approxima_tree}. It is built upon a recursive definition (see Definition~\ref{def:Llow}) involving Taylor remainders of $ \Pi_A $ with commutators. They are given in Lemma~\ref{Taylor_bound}. Proposition~\ref{decomp_Pi} gives the structures of the dominant operators appearing in the local error analysis.
In Section~\ref{sec:genScheme}, we introduce truncated series of decorated trees that solves up to order $ r $ equation~\ref{ev} (see Proposition~\ref{exact_solution_order_r}).   Then, from this series built upon the character $ \Pi $, one can write the general scheme (see Definition~\ref{genscheme}) which boils down to replace $  \Pi $ by its approximation $ \Pi_A $. In the end, we compute its local error structure (see Theorem~\ref{thm:genloc}) based on the local error between $ \Pi $ and $ \Pi_A $ for each decorated tree that appears in the expansion of the scheme.

In Section~\ref{sec:examples}, we illustrate the general framework on  various applications. The main example is the Gross-Pitaevskii (GP) equation that includes both a potential and a rough initial data. With our general framework we derive a first and second order scheme for GP and carry out precisely its local error analysis.
We also discuss adapted filter operators for the stability of the low regularity scheme in Section~\ref{stabilisater}.
The last example of this section is the Klein and Sine-Gordon equations which illustrated our framework for non-polynomial nonlinearities.

\subsection*{Acknowledgements}

{\small
This project has received funding from the European Research Council (ERC) under the European Union's Horizon 2020 research and innovation programme (grant agreement No. 850941). YB thanks the Max Planck Institute for Mathematics in the Sciences (MiS) in Leipzig for supporting his research via a long stay in Leipzig from January to June 2022 where part of this work was written.
}

\section{Decorated trees}
\label{Sec::2}

In this section, we introduce the formalism of decorated trees that is used for describing the iterated integrals stemming  from the iteration of the Duhamel's formula~\eqref{duhamel_main} of our main equation~\eqref{ev}. Decorations on the nodes will encode monomials $ \xi^k $ and potentials $ V_{\mathfrak{l}} $ that appear in the equation whereas decorations on the edges encode the various integrals in time. The two spaces of importance are $ \hat \CT $ (resp. $ \hat{\CH} $) space of decorated trees (resp. forests). They correspond to iterated integrals without approximations. We define on these combinatorial objects the dominant operators that are used for the discretisation (see Definition~\ref{dom_freq}). In the end, we consider the spaces $ \CT $ and $ \CH $ obtained by adding decorations at the root. This decoration gives the order at which we want to approximate these iterated integrals.

\subsection{Definitions and dominant operators}

We assume   a finite set $  \Lab$ which has the following splitting $ \Lab = \Lab_+ \sqcup \Lab_-$. We consider a family of differential operators $ (\mathcal{L}_{o})_{o \in \Lab_{+} } $ indexed by $ \Lab_+ $ and a family of potentials $ (V_{\mathfrak{l}})_{\mathfrak{l} \in \Lab_-} $ indexed by $ \Lab_- $. We suppose given $ (\Lab_+^{o,\mathfrak{l}})_{(o,\mathfrak{l}) \in \Lab_{+} \times \Lab_-} $ a collection of subsets of $ \Lab_+ $.
We define the set of decorated trees $ \hat \CT  $ 
as elements of the form  $ 
T_{\Labe}^{\Labn, \mathfrak{f}} =  (T,\Labn,\mathfrak{f},\Labe) $ where 
\begin{itemize}
\item $ T $ is a non-planar rooted tree with root $ \varrho_T $, node set $N_T$ and edge set $E_T$. We denote the leaves of $ T $ by $ L_T $.
\item the map $ \Labe : E_T \rightarrow \Lab_{+} $ are edge decorations. An edge decorated by $ o \in \Lab_+ $ will encode an integral of the form $ \int^{t}_0 e^{(t-s)\mathcal{L}_o} ... ds $. 
\item the map $ \Labn : N_T  \rightarrow \N $ are node decorations. They correspond to monomials of the form $ \xi^k $, $ k \in \N $ that appear in iterated integrals. 
\item the map $ \mathfrak{f} : N_T  \rightarrow \Lab_-  $ are node decorations encoding the various potentials $ V_{\mathfrak{l}} $.
\end{itemize}

When the node decoration $ \Labn $ is omitted, we will denote the decorated trees as
 $ T_{\Labe}^{\mathfrak{f}} $. This set of decorated trees is denoted by $ \hat \CT_0 $.  We say that  $ \bar T_{\bar \Labe}^{\bar{\mathfrak{f}}} $ is a decorated subtree of $ T_{\Labe}^{\mathfrak{f}} \in \hat \CT_0 $ if  $ \bar T $ is a subtree of $ T $ and the restriction of the decorations $  \Labe $ and $ \mathfrak{f} $ of $ T $ to $ \bar T $ are given by $ \bar \Labe $ and $ \bar{\mathfrak{f}} $. We denote by $ \hat{\mathcal{P}} $ the set of planted trees $ T^{\Labn, \mathfrak{f}}_{\Labe} $ that are decorated trees with only one edge connected to the root and with no decorations at the root. Similarly when the node decoration $\Labn$ is omitted we denote this set $ \hat \CP_0 $. We denote by $\hat H $ (and resp. $ \hat H_0 $) the (unordered) forests composed of trees in $ \hat \CP $ (and resp. $ \hat \CP_0 $) (including the empty forest denoted by $ \one $).  Their linear spans are denoted by $\hat \CH $ (and resp. $ \hat \CH_0 $).

In order to represent these decorated trees, we introduce a symbolic notation. An edge decorated by  $  o \in \Lab_{+}$ is denoted by $ \CI_{o} $. The symbol $  \CI_{o} (\cdot) : \hat \CT \rightarrow  \hat{\CH} $ is viewed as  the operation that connects the root of a  decorated tree to a new root with no decorations via an edge decorated by $ o $. Any decorate tree $ T $ admits the following decomposition:
\begin{equs} \label{decomposition_tree}
T = \lambda^{\ell}_{\mathfrak{l}} \prod_{i=1}^m\mathcal{I}_{o_i}(T_i)
\end{equs}
where the $ T_i $ are decorated trees, $ \lambda^{\ell}_{\mathfrak{l}} $ corresponds to the decorations at the root with $ \ell \in \N $ and $ \mathfrak{l} \in \Lab_- $. The product in \eqref{decomposition_tree} is viewed as the forest product. This decomposition means that every decorated tree can be identified with a forest $\prod_{i=1}^m\mathcal{I}_{o_i}(T_i)$ and a decoration  $\lambda^{\ell}_{\mathfrak{l}}$.  When $ \ell = 0 $, we will use the shorthand notation: $ \lambda^{0}_{\mathfrak{l}} = \lambda_{\mathfrak{l}}$. Below, we provide an example of such decorated trees:
\begin{equs} \label{example2}
\lambda^{\ell_0}_{\mathfrak{l}_0} \mathcal{I}_{o_1}(\lambda^{\ell_1}_{\mathfrak{l}_1}) \mathcal{I}_{o_2}(\lambda^{\ell_2}_{\mathfrak{l}_2}) \mathcal{I}_{o_3}(\lambda^{\ell_3}_{\mathfrak{l}_3}) = \begin{tikzpicture}[scale=0.22,baseline=0cm]
         \node at (-6,7)  (f) {}; 
           \node at (0,7)  (g) {}; 
            \node at (6,7)  (e) {}; 
        \node at (0,1) (c) {}; 
     \draw[kernel1] (c) -- node [round1] {\tiny  $o_1$}  (f) ; 
     \draw[kernel1] (c) -- node [round1] {\tiny  $ o_2$}  (g) ; 
          \draw[kernel1] (c) -- node [round1] {\tiny  $  o_3$}  (e) ; 
       \draw (f) node [rect2] {\tiny $(\ell_1,\mathfrak{l}_1)$}  ;
    \draw (g) node [rect2] {\tiny $(\ell_2,\mathfrak{l}_2)$}  ;
      \draw (e) node [rect2] {\tiny $(\ell_3,\mathfrak{l}_3)$}  ;
    \draw (c) node [rect2] {\tiny $(\ell_0,\mathfrak{l}_0)$}  ;
\end{tikzpicture} 
\end{equs}
Iterated integrals will be associated to these decorated trees. In order to approximate them numerically we will have to resolve the underlying oscillations by extracting the dominant parts of the leading operators. In the next definition, we introduce   these dominant parts.

\begin{definition}\label{def:Dom}
Let $\CA_i$, $ i \in \lbrace 1,...,m \rbrace $ be a finite set of operators with domains $ D(\CA_i) $ in the sense that $ \CA_i : D(\CA_i) \subseteq X \rightarrow X $. We define $ \mathcal{P}_{\tiny{\text{dom}}}(\lbrace\CA_1,...,\CA_m \rbrace) $ as follows: We first consider $ I \subset \lbrace 1,...,m \rbrace  $ such that $ \CA_i $,  $ i \in I $ are of smaller domain in the sense that:
\begin{equs}
D(\CA_i) = D(\CA_j), \, i,j \in I, \quad D(\CA_i) \subsetneq D(\CA_\ell)  , \, \, \ell \in \lbrace 1,...,m \rbrace \setminus I, i \in I.
\end{equs}
If  there exists an operator $  \mathcal{L}_{\tiny{\text{dom}}} $ such that for every $ i \in I $
\begin{equs} \label{decomp_1}
\CA_i = \mathcal{L}_{\tiny{\text{dom}}} + \CA_{\tiny{\text{low}}}^i. 
\end{equs}
where $\mathcal{L}_{\tiny{\text{dom}}} $ satisfies $D(\mathcal{L}_{\tiny{\text{dom}}}) \subsetneq D(\CA_{\tiny{\text{low}}}^i)   $, then we set
\begin{equs}
\mathcal{P}_{\tiny{\text{dom}}}(\lbrace\CA_1,...,\CA_m\rbrace) =  \mathcal{L}_{\tiny{\text{dom}}}.
\end{equs}
Otherwise, it is equal to $0 $.
\end{definition}

\begin{remark}
The identity \eqref{decomp_1} has a to be understood as a decomposition between lower and upper part of the operators $ \CA_i $. What is important in order to have a non-zero dominant part is to have a factorisation with some operator $ \mathcal{L}_{\tiny{\text{dom}}} $. It is similar to the approach in Fourier space presented in \cite[Def. 2.2]{BS} when one was looking at the form of the higher monomials of polynomials in the frequencies.
\end{remark}

\begin{example}
We illustrate the previous definition with  $\lbrace 0, \mathcal{L} \rbrace$, where $ \mathcal{L} $ has the domain $ D(\mathcal{L}) \neq X $. Then
\begin{equs}
\mathcal{P}_{\tiny{\text{dom}}}(\lbrace 0, \mathcal{L} \rbrace) = \mathcal{L}
\end{equs}
and for $ \lbrace 0,- \mathcal{L}, \mathcal{L} \rbrace $, one has
\begin{equs}
\mathcal{P}_{\tiny{\text{dom}}}(\lbrace 0,-\mathcal{L}, \mathcal{L} \rbrace) = 0.
\end{equs}
The latter follows as  $ D(-\mathcal{L}) = D(\mathcal{L}) $ and the coefficients in front of $ \mathcal{L} $ are different ($ 1 $ and $ -1 $).
\end{example}

Given a decorated tree, we want to compute its dominant part and lower part using the map $ \mathcal{P}_{\tiny{\text{dom}}} $. We suppose given a space of operators $ \mathcal{Op} $ in $X$ and we denote by $ \mathcal{OS} $ the collection of finite sets of operators belonging to $ \mathcal{Op} $. Given two sets $\CA_1$ and $ \CA_2$, we denote their  union by $  \CA_1 \,\cup \, \CA_2 $.
Given a finite set of operators $ \CA $ and an operator $ \mathcal{L} $, we define $ \mathcal{L} \oplus \CA $ as the set where $\mathcal{L}$ is added to each component of $\CA$:
\begin{equs}
 \mathcal{L} \oplus \CA =  \cup_{\mathcal{L}_o \in \CA} \lbrace\mathcal{L} + \mathcal{L}_o \rbrace.
\end{equs}
 The term $   \CA \ominus \mathcal{L} $ is computed as follows for each component $ \mathcal{L}_o $ of $ \CA $. If the dominant part of $ \mathcal{L}_o $ is equal to $ \mathcal{L}$, in  the sense that one has the following decomposition:
 \begin{equs} \label{decomp_operator}
 \mathcal{L}_{o} = \mathcal{L} + 
 \bar{\mathcal{L}}
 \end{equs}
 with $ D(\mathcal{L}) \subsetneq D(\bar{\mathcal{L}}) $, then
 we replace $ \mathcal{L}_o $ by $ \bar{\mathcal{L}} $. Otherwise, we keep $ \mathcal{L}_o $. 
We define $ \max(\mathcal{L},\CA) $ as the operation that replaces each element $ \mathcal{L}_o $ satisfying \eqref{decomp_operator} by $ \mathcal{L} $ otherwise we replace it by $ 0 $.
Equipped with these notations, we are able to state the main definition for computing the dominant part associated to a decorated tree:

\begin{definition} \label{dom_freq} We recursively define $  \CR_{\text{\tiny{dom}}}^{o} : \hat \CT_{0} \rightarrow  \mathcal{OS}$, $ \CR_{\text{\tiny{dom}}} :  \hat \CP_{0} \rightarrow  \mathcal{OS}  $,  $ \CR_{\text{\tiny{low}}} :  \hat \CP_{0} \rightarrow  \mathcal{OS}  $ and $ \mathcal{L}_{\text{\tiny{dom}}} :  \hat \CP_{0} \rightarrow  \mathcal{Op}  $  as:
\begin{equs}
\CR^o_{\text{\tiny{dom}}}( \lambda_{\mathfrak{l}} \prod_i \CI_{o_i}(T_i)) & = \lbrace  \mathcal{L}_{\bar{o}}, \, \bar{o} \in \mathfrak{L}_+^{o,\mathfrak{l}} \rbrace  \cup\bigcup_i \CR_{\text{\tiny{dom}}}(\mathcal{I}_{o_i}(T_i)),    \\
\mathcal{L}_{\text{\tiny{dom}}}\left( \CI_{o}(T) \right)  & =   \CP_{\text{\tiny{dom}}}\left( -\mathcal{L}_o \oplus \CR^{o}_{\text{\tiny{dom}}}(T) \right), \,
 \\
 \CR_{\text{\tiny{dom}}}\left( \CI_{o}(T) \right)  & =   
\mathcal{L}_o \oplus \max( \mathcal{L}_{\text{\tiny{dom}}},  -\mathcal{L}_o \oplus \CR^o_{\text{\tiny{dom}}}(T) ),
 \\
\CR_{\text{\tiny{low}}}\left( \CI_{o}(T) \right)  & =   
\left( \id \ominus \CP_{\text{\tiny{dom}}} \right) \left( -\mathcal{L}_o \oplus \CR^o_{\text{\tiny{dom}}}(T) \right).
\end{equs}
 We extend these maps to $ \hat{\CT} $ and $ \hat{\CP} $ by ignoring the node decorations $ \Labn $.
\end{definition}

\begin{example}\label{NLS_1}
 We illustrate the previous abstract definitions on some trees stemming from the cubic nonlinear Schr\"odinger equation
\begin{equation}\label{nls}
i \partial_t u + \Delta u  = \vert u\vert^2 u, \quad u_0 = v
\end{equation}
 set on $\R^d$, $1 \leq d \leq 3$. 
For the equation \eqref{nls}, we have two variables which are $ u $ and $ \bar{u} $. Thus, we rewrite \eqref{nls} in the following form:
\begin{equation} \label{nls2}
\begin{aligned}
\partial_t u_o - \mathcal{L}_o u_o  = -i (u_o)^2 u_{\bar o}, \quad u_o(0) = v_o \\
\partial_t u_{\bar o} - \mathcal{L}_{\bar o} u_{\bar o}  =  i ( u_{\bar o} )^2 u_o, \quad u_{\bar o}(0) = v_{\bar o}
\end{aligned}
\end{equation}
where one has $ \mathcal{L}_{o} = i \Delta$, $ \Lab_+ = \lbrace o , \bar o\rbrace $, $ \Lab_- = \lbrace 0 \rbrace $, $ \Lab_+^{o,0} = \Lab_+^{\bar o,0} = \lbrace o , \bar o\rbrace  $ with
$ V_0 = 1 $, $\mathcal{B}^{0}_{o} = \mathcal{B}^{0}_{\bar{o}} = \id$, $ \mathcal{L}_{o} = \mathcal{L} $, $ \mathcal{L}_{\bar o} = -\mathcal{L} $ and $X= L^2(\R^d)$, $D( \mathcal{L}) = H^2$, $ u_o = u  $, $ u_{\bar o} = \overline{u} $, $ v_o = v  $, $ v_{\bar o} = \overline{v} $. The nonlinearities are given by:  $ f_{o,o}^0(u) = -i u^2$, $ f^0_{o,\bar{o}}(\bar u ) = \bar u $,
$ f_{\bar{o},\bar{o}}^0(\bar u) = +i \bar{u}^2$ and  $ f_{\bar{o},o}^0( u ) =  u $. Let us recall the meaning of the subscripts in $ f_{o,\bar{o}}^0 $: The $ 0 $ corresponds to the driver $ V_0 $, $ \bar{o} $ that it appears in the equation for $ \overline{u} $ and the last subscript says that $ f_{o,\bar{o}}^0 $ depends only on the variable $ u $. 
Next we consider the following decorated tree 
 \begin{equs} 
 \label{NLStree1}
 T = \CI_{o}(\lambda_{0}) = \begin{tikzpicture}[scale=0.2,baseline=-5]
\coordinate (root) at (0,-1);
\coordinate (tri) at (0,1);
\draw[kernels2] (root) -- (tri);
\node[not] (rootnode) at (root) {};
\node[xi] (trinode) at (tri) {};
\end{tikzpicture}
\end{equs}
 which encodes the iterated integral:
\begin{equs} \label{int_NLS1}
 -i \int_0^t e^{(t-s)\mathcal{L}} \left( \left( e^{s \CL}v^2 \right) \left(  e^{ - s \CL} \bar{v} \right) \right) ds.
\end{equs}
The brown edge stands for the integral in time $  \int_0^t e^{(t-s)\mathcal{L}} ... ds $. The white dot associated to the potential $ V_0 $ encodes the product $-i\left( e^{s \CL} v^2 \right) \left( e^{ - s \CL} \bar{v}
\right)$. Indeed, it is connected via an edge associated to $ \mathcal{L} $, so we know that we have to consider the nonlinearity associated to $ V_0 $ in the equation for $ u $. One gets from the definition of $ \Lab_+^{o,0} $ 
\begin{equs}
 \CR^o_{\text{\tiny{dom}}} \left(  \begin{tikzpicture}[scale=0.2,baseline=-5]
\coordinate (root) at (0,-0.5);
\node[xi] (rootnode) at (root) {};
\end{tikzpicture} \right) = \lbrace -\mathcal{L},\mathcal{L}\rbrace.
\end{equs}
Then applying the previous definitions (see Definition \ref{dom_freq}), one gets
\begin{equs}
-\mathcal{L} \oplus \CR^o_{\text{\tiny{dom}}}\left( \begin{tikzpicture}[scale=0.2,baseline=-5]
\coordinate (root) at (0,-0.5);
\node[xi] (rootnode) at (root) {};
\end{tikzpicture}  \right) & = \lbrace - 2 \mathcal{L}, 0 \rbrace, \quad
\mathcal{L}_{\text{\tiny{dom}}}\left(  T  \right)   = \CP_{\text{\tiny{dom}}}\left(   \lbrace - 2\mathcal{L},0\rbrace \right) = - 2 \mathcal{L}, \\
\CR_{\text{\tiny{dom}}}\left( T \right)  & =   
\mathcal{L} \oplus \max( \mathcal{L}_{\text{\tiny{dom}}},  -\mathcal{L} \oplus  \lbrace -\mathcal{L},\mathcal{L}\rbrace ) = \lbrace -\mathcal{L},\mathcal{L}\rbrace
\\
\CR_{\text{\tiny{low}}}\left(  T  \right) & = 
(\id \ominus \CP_{\text{\tiny{dom}}}) \left( \lbrace - 2 \mathcal{L}, 0  \rbrace    \right) = \lbrace -2 \mathcal{L},0 \rbrace \ominus - 2 \mathcal{L} = \lbrace 0 \rbrace.
\end{equs}
\end{example}

\begin{remark} One can notice that some of the notations that we have introduced are similar to the one used in \cite{BS} which took a direct inspiration from the algebraic structures developed for singular SPDEs in \cite{BHZ,BCCH}. Our decorations  are however different to \cite{BS}, because we focus on a more general set up no longer restricted to periodic dispersive equations:
\begin{itemize}
 \item The conjugate operator was encoded in \cite{BS} with edge decorations in $ \lbrace -1,1 \rbrace $ which allows the computation of the frequency interactions. Now, we have a larger set of operators by having a bigger set $ \Lab_+ $ with elements such as  $ \overline{o} $ like in the NLS example. One has $ \mathcal{L}_{\overline{o}} = \overline{\mathcal{L}}  = \mathcal{L}_{(o,-1)}$ where the last identity corresponds to the old notation coming from \cite{BS}.
 \item Another change on the edge decorations is that in \cite{BS} certain edges correspond to some integrals in time others not. This excludes the parabolic case. If we rewrite an integral of the form $ \int_0^t e^{(t-s)\mathcal{L}} \cdots ds $ into $ e^{t \mathcal{L}} \int_0^t e^{-s\mathcal{L}} \cdots ds  $ then it does not make sense when $ \mathcal{L} = \Delta $. We recall one decorated tree coming from \cite{BS}:
 \begin{equs} \label{coding_1}
\begin{tikzpicture}[scale=0.2,baseline=-5]
\coordinate (root) at (0,2);
\coordinate (root1) at (0,0);
\coordinate (tri) at (0,-2);
\coordinate (t1) at (-2,4);
\coordinate (t2) at (2,4);
\coordinate (t3) at (0,5);
\draw[kernels2,tinydots] (t1) -- (root);
\draw[kernels2] (t2) -- (root);
\draw[kernels2] (t3) -- (root);
\draw[kernels2] (tri) -- (root1);
\draw[symbols] (root1) -- (root);
\node[not] (rootnode) at (root1) {};
\node[not] (rootnode) at (root) {};
\node[not] (trinode) at (tri) {};
\node[var] (rootnode) at (t1) {\tiny{$ k_{\tiny{1}} $}};
\node[var] (rootnode) at (t3) {\tiny{$ k_{\tiny{2}} $}};
\node[var] (trinode) at (t2) {\tiny{$ k_3 $}};
\end{tikzpicture} & \equiv -i e^{-i t k^2}\int_0^t  e^{i s k^2} \left( e^{i s k_1^2}e^{-i s k_2^2} e^{-i s k_3^2}\right) ds, 
\end{equs} where $ k =  - k_1 + k_2 + k_3 $, the leaves are decorated by the frequencies $ k_1, k_2, k_3 $ and the inner nodes are decorated by $ k $. The blue edge encodes an integral in time $  -i\int_0^t  e^{i s k^2} \cdots ds $, the brown edges are used for a factor $ e^{-is k^2} $ and the dashed brown edges are for $ e^{i s k^2} $. If we consider the same integral not in Fourier mode, one can rewrite it as 
\begin{equs} \label{coding_2}
\begin{tikzpicture}[scale=0.2,baseline=-5]
\coordinate (root) at (0,0);
\coordinate (tri) at (0,-2);
\coordinate (t1) at (-2,2);
\coordinate (t2) at (2,2);
\coordinate (t3) at (0,3);
\draw[kernels2,tinydots] (t1) -- (root);
\draw[kernels2] (t2) -- (root);
\draw[kernels2] (t3) -- (root);
\draw[symbols] (tri) -- (root);
\node[not] (rootnode) at (root) {};
\node[not] (rootnode) at (tri) {};
\node[xi,blue] (rootnode) at (t1) {};
\node[not] (rootnode) at (t3) {};
\node[not] (trinode) at (t2) {};
\end{tikzpicture} & \equiv -i \int_0^t  e^{(t-s)\mathcal{L} } \left(  \left( e^{s \mathcal{L}} v \right)^2 \left( e^{-s \mathcal{L}} \overline{v} \right) \right) ds
\end{equs}
where $\mathcal{L} = i \Delta$, the blue dot corresponds to $ \overline{v} $ and the blue edge now encodes an integral of the form $ \int_0^t e^{(t-s)\mathcal{L}} \cdots ds $. In \eqref{coding_2}, we have to incorporate the initial data $ v $ and $ \bar{v} $ in the definition while they are implicit when the integral is written in Fourier space in \eqref{coding_1} (one has just to multiply the integral with $ \bar{\hat{v}}_{k_1} \hat{v}_{k_2} \hat{v}_{k_3}  $). In fact, the coding given by \eqref{coding_2} is sufficient for cubic NLS, but not for the general case we have in mind that contains potentials and non-polynomial nonlinearities.
\item We add new decorations at the nodes stemming from the drivers $V_{\mathfrak{l}}$ of \eqref{ev}. The initial conditions are associated to the drivers and the equations we are using via the sets $ \Lab_+^{o,\mathfrak{l}} $.  Let us take for  example  a nonlinearity of the form $$f(u)g(\overline{u}) V_{\mathfrak{l}} .$$ Then, after inserting the approximation $ u(s) = e^{s \mathcal{L}} v + A(s) $, one gets thanks to \eqref{commutator}
\begin{equs}\label{nodeDef}
\left( e^{s \mathcal{L}}f(v) \right) \, \left( e^{s \overline{\mathcal{L}}} g(\overline{v}) \right)V_0  \equiv 
 \begin{tikzpicture}[scale=0.2,baseline=-5]
\coordinate (root) at (0,-0.5);
\node[xi] (rootnode) at (root) {};
\end{tikzpicture} = \lambda_0.
\end{equs}
Next if we integrate by $ \int_0^t e^{(t-s)\mathcal{L}} \cdots ds $ encoded by a brown edge, we obtain the following integral when $ f(u) = u^2 $ and $ g(\bar{u}) = \bar{u} $:
\begin{equs} \label{coding_3}
  \CI_{o}(\lambda_{0})= \begin{tikzpicture}[scale=0.2,baseline=-5]
\coordinate (root) at (0,-1);
\coordinate (tri) at (0,1);
\draw[kernels2] (root) -- (tri);
\node[not] (rootnode) at (root) {};
\node[xi] (trinode) at (tri) {};
\end{tikzpicture} \equiv -i \int_0^\tau e^{i(\tau-s)\Delta} \left( \left( e^{is \Delta}v^2 \right) \left(  e^{ - is \Delta} \bar{v} \right) \right) ds
\end{equs}
The main difference between $ \eqref{coding_2} $ and $ \eqref{coding_3} $ is $( e^{is \Delta}v)^2 $ which is replaced by $ e^{is \Delta} v^2 $. We are dealing in \eqref{coding_3} with integrals that contain less factors of the form $ e^{i s \Delta} v $.
\end{itemize}
\end{remark}

\begin{example} \label{ex_Gross} As a second example, let us  consider the Gross--Pitaevskii  equation. The main difference with the cubic NLS equation \eqref {nls} is the adjunction of a   potential $ V $. The equation takes the form:
\begin{equation}\label{nls_bis}
i \partial_t u + \Delta u  = \vert u\vert^2 u + u V, \quad u_0 = v
\end{equation}
 set on $\R^d$, $d \le 3$. Now, one has $ \Lab_-= \lbrace  0,1 \rbrace $,  $ \Lab_+^{o,1} = \lbrace o\rbrace  $,  $ \Lab_+^{\overline{o},1} = \lbrace \overline{o} \rbrace  $ with
$ V_1 = V $. The other sets remain the same, and are defined in Example \ref{NLS_1}.
The new nonlinearities coming from the added potential term are given by:  $ f_{o,o}^1(u) =  -iu$ and $ f_{\bar{o},\bar{o}}^1(\bar u ) = i\bar u $. Hence, together with the nonlinearities defined in  \eqref{nls}, we indeed recover the nonlinear terms of \eqref{nls_bis} from the general form \eqref{ev} :
\begin{equs}
\sum_{\mathfrak{l} \in \Lab_-}\Psi_{o}^{\mathfrak{l}}(\mathbf{u}_{o}^{\mathfrak{l}}) V_{\mathfrak{l}}(x)
&= \Psi_{o}^{0}(\mathbf{u}_{o}^{0}) V_{0}(x) + \Psi_{o}^{1}(\mathbf{u}_{o}^{1}) V_{1}(x) \\
&= f_{o,o}^0(u_{o})f_{o,\bar{o}}^0(u_{\bar{o}}) + f_{o,o}^1(u_{o})V \\
&= -i(u_{o})^2u_{\bar{o}} - iu_{o}V,
\end{equs}
where we recall that we set $\mathcal{B}_{o}^{1} = \mathcal{B}_{\bar{o}}^{1}= \id$.
We rewrite equation \eqref{nls_bis} into
\begin{equation} \label{nls3}
\begin{aligned}
\partial_t u_o - \mathcal{L}_o u_o  = - i (u_o)^2 u_{\bar o} - i u_o V, \quad u_o(0) = v_o \\
\partial_t u_{\bar o} - \mathcal{L}_{\bar o} u_{\bar o}  =  i ( u_{\bar o} )^2 u_o +i u_{\bar{o}} V, \quad u_{\bar o}(0) = v_{\bar o}
\end{aligned}
\end{equation}
where for simplicity we have assumed a real potential $ \overline{V} = V $.
The central iterated integrals up to second order then take the form
\begin{equs} 
\CI_{o}(\lambda_{1})  & = \begin{tikzpicture}[scale=0.2,baseline=-5]
\coordinate (root) at (0,-1);
\coordinate (tri) at (0,1);
\draw[kernels2] (root) -- (tri);
\node[not] (rootnode) at (root) {};
\node[xi,blue] (trinode) at (tri) {};
\end{tikzpicture} \equiv - i \int^{t}_0 e^{(t-s) \mathcal{L}} \left(  \left( e^{s \mathcal{L}} v \right) V \right) ds  \\ 
\CI_{\overline{o}}(\lambda_{1})  & = \begin{tikzpicture}[scale=0.2,baseline=-5]
\coordinate (root) at (0,-1);
\coordinate (tri) at (0,1);
\draw[kernels2,tinydots] (root) -- (tri);
\node[not] (rootnode) at (root) {};
\node[xi,blue] (trinode) at (tri) {};
\end{tikzpicture} \equiv   i \int^{t}_0 e^{(s-t) \mathcal{L}} \left(  \left( e^{-s \mathcal{L}} \overline{v} \right) V \right) ds  \\ &
\CI_{o}(\lambda_{0} \mathcal{I}_{o}(\lambda_1))   = \begin{tikzpicture}[scale=0.2,baseline=-5]
\coordinate (root) at (0,-1);
\coordinate (tri) at (0,1);
\coordinate (tri1) at (0,3);
\draw[kernels2] (root) -- (tri);
\draw[kernels2] (tri) -- (tri1);
\node[not] (rootnode) at (root) {};
\node[xi] (trinode1) at (tri) {};
\node[xi,blue] (trinode) at (tri1) {};
\end{tikzpicture} \\ &  \equiv  - \int_0^t e^{(t-s)\mathcal{L}} \left(  \left( \int^{s}_0 e^{(s-r) \mathcal{L}} \left( \left( e^{r \mathcal{L}} v \right) V \right) dr\right) \left( e^{  s \CL} v \right) \left( e^{ - s \CL} \bar{v} \right) \right) ds.
\\ &
\CI_{o}(\lambda_{0} \mathcal{I}_{\bar{o}}(\lambda_1))  = \begin{tikzpicture}[scale=0.2,baseline=-5]
\coordinate (root) at (0,-1);
\coordinate (tri) at (0,1);
\coordinate (tri1) at (0,3);
\draw[kernels2] (root) -- (tri);
\draw[kernels2,tinydots] (tri) -- (tri1);
\node[not] (rootnode) at (root) {};
\node[xi] (trinode1) at (tri) {};
\node[xi,blue] (trinode) at (tri1) {};
\end{tikzpicture} 
\\ & 
\equiv   \int_0^t e^{(t-s)\mathcal{L}} \left(  \left( \int^{s}_0 e^{(r-s) \mathcal{L}} \left( \left( e^{-r \mathcal{L}} \overline{v} \right) V \right) dr\right) \left( e^{  s \CL} v^2 \right)  \right) ds,
\end{equs}
where we have used a blue dot for encoding the potential $ V $. A full list of these trees is given in the Section~\ref{Gross_Pitaevskii}.
\end{example}

\subsection{Approximated decorated trees}

 We denote by $ \CT $ the set of decorated trees $ T_{\Labe,r}^{\Labn, \mathfrak{f}} = (T,\Labn,\mathfrak{f},\Labe,r)  $ where
\begin{itemize}
\item $ T_{\Labe}^{\Labn, \mathfrak{f}} \in \hat \CT $.
\item The decoration of the root is given by $ r \in \Z $, $ r \geq -1 $ such that
\begin{equs} \label{condition_trees}
 r +1 \geq  \deg(T_{\Labe}^{\Labn, \mathfrak{f}})
\end{equs}
 where $ \deg $ is defined recursively by 
 \begin{equs}
\deg(\one) & = 0, \quad \deg( \lambda^{\ell}_{\mathfrak{l}} \prod_{i=1}^m\mathcal{I}_{o_i}(T_i) )  = \ell + 1 + \max(\deg(T_1),...,\deg(T_m)).
\end{equs}
\end{itemize}
We call $ \CT $ approximated decorated trees following the terminology introduced in \cite{BS}. Indeed, the decoration $ r $ at the root means that we consider an approximation at order $ r $ of the iterated integrals associated to the same trees without this decoration. The quantity $\deg(T_{\Labe}^{\Labn})$ is the maximum number of edges and node decorations $ \Labn $ lying on the same path from one leaf to the root.

We denote by $ \CP $ the planted trees satisfying the same condition as $ \CT $. The forests formed of these trees are given by $ H $ and their linear span by $ \CH $. They are of the form $ \CI_{o}^r(T) $.
  The map $    \CI^{r}_{o}( \cdot) :  \hat \CT \rightarrow    \CH   $ is defined as the same as for $ \CI_{o}(\cdot) $ except now that the root is decorated by $ r $ and it could be zero if the inequality \eqref{condition_trees} is not satisfied.
  It can be naturally extended into a map from $ \CT $ into $ \CH $ by removing the $ r $ decoration at the root of the tree that is grafted onto a new root.
We define the map $ \CD_{r} : \hat \CP \rightarrow \CH $ which adds the decoration $ r $ and performs the projection along the identity \eqref{condition_trees}. It is given by
\begin{equs}\label{DR}
\CD_{r}(\one)= \one_{\lbrace 0 \leq  r+1\rbrace} , \quad \CD_r\left( \CI_{o}(T) \right) =  \CI^{r}_{o}( T)
\end{equs}
We extend this map to $ \CT $ by:
\begin{equs}
\CD_{r}(\lambda^{\ell}_{\mathfrak{l}} \prod_{i=1}^m\mathcal{I}_{o_i}(T_i)) = \lambda^{\ell}_{\mathfrak{l}} \prod_{i=1}^m\mathcal{I}^{r-\ell}_{o_i}(T_i)
\end{equs}
where the term on the right hand-side is identified with $ \lambda^{\ell}_{\mathfrak{l}} \prod_{i=1}^m\mathcal{I}_{o_i}(T_i) $ where the decoration $ r $ is added to its root with the projection given by \eqref{condition_trees}.

\section{Approximations of iterated integrals}

In this section, we introduce the iterated integrals associated to decorated trees via a character $ \Pi : \hat \CH \rightarrow \CC $, a multiplicative map for the forest product. We approximate this map via a new character $ \Pi_A : \CH \rightarrow \CC $ defined on approximated decorated forests with $ A $ a domain to  which $ \hat{\mathbf{v}} $ belongs to. The crucial part of the recursive definition of $ \Pi_A $ is given by an operator $ \mathcal{K} $ that filter out the dominant part of the operators. It expands only the lower order parts  into a Taylor series expansion while integrating exactly the dominant part (see Definition~\ref{def_CK}). The definition of these operators is similar to the one in Fourier given in \cite{BS}. Then, in the second part of this section, we conduct the local error analysis by comparing $ \Pi $ with $ \Pi_A $ with the main result given in Theorem~\ref{approxima_tree}.
The main argument for establishing this bound is a recursive definition (see Definition~\ref{def:Llow}) which involves several Taylor remainders (see Lemma~\ref{Taylor_bound}). These remainders are stemming from the various approximations performed by the map $ \mathcal{K} $. Such a description is analogue to the Fourier case that can provide a further expansion of the local error analysis via a Birkhoff factorisation involving a Butcher-Connes-Kreimer coproduct. We do not have such a characterisation for our scheme. However in Proposition~\ref{decomp_Pi} we are able  to show how dominant operators are involved in the local error analysis. 
This was previously observed in the Fourier case.

\label{sec::Iterated integrals}

\subsection{Characters on decorated forests}

We first introduce the linear space $ \mathcal{C} $ in which the iterated integrals associated to Duhamel's  formula \eqref{duhamel_main} live. Any element $ f \in \CC $ is such that for every $ \tau \in \R_+ $, one has that $ f(\cdot,\tau) $ is an operator on $ \hat{\mathbf{v}} $. We 
define a character  $ \Pi $  on $ \hat \CH $ taking values in $ \CC $. It is given for $ T \in \hat \CT $ and $ F_1, F_2 \in \hat{\CH} $ by
\begin{equs} \label{recursive_Pi}
 \left( \Pi \CI_{o}(T) \right) (\hat{\mathbf{v}},\tau)  & = \int_{0}^{\tau} e^{(\tau-\xi)\mathcal{L}_{o}} ( \tilde \Pi_{o} T)(\hat{\mathbf{v}},\xi)   d \xi, \\ (\Pi F_1 F_2)(\hat{\mathbf{v}},\xi) & = (\Pi F_1 )(\hat{\mathbf{v}},\xi) \, (\Pi F_2 )(\hat{\mathbf{v}},\xi)
\end{equs}
where $ F_1 F_2 $ is the forest product and $ \tilde{\Pi}_o $ is given on a decorated tree
 $ T = \lambda^k_{\mathfrak{l}} \prod_{i=1}^n \CI_{o_i}(T_{i})   $ by
\begin{equs}
\label{recursive_Pi_1}
\begin{aligned}
\left(  \tilde \Pi_o \, \lambda^k_{\mathfrak{l}} \prod_{i=1}^n \CI_{o_i}(T_i) \right)(\hat{\mathbf{v}},\xi) &  = \CB_{o}^{\mathfrak{l}} \left( \frac{\Upsilon_{\text{\scriptsize{root}}}^{\Psi_{o}}[T]}{S_{\text{\scriptsize{root}}}(T)} V_{\mathfrak{l}} \, \xi^{k} \Pi\left( \prod_{i=1}^n \ \CI_{o_i}(T_i) \right) \right) (\hat{\mathbf{v}},\xi) 
\end{aligned}
\end{equs}
where $ S_{\text{\scriptsize{root}}}(T) $ is the symmetry factor associated to the root of the decorated tree $ T $. The symmetry factor is defined by setting $S_{\text{\scriptsize{root}}}(\one) = 1$ and
\begin{equs}\label{S_root}
 S_{\text{\scriptsize{root}}} \left( \lambda^k_{\mathfrak{l}} \prod_{i,j} (\CI_{o_i}(T_{i,j}))^{\beta_{i,j}} \right) = k!\prod_{i,j}
\beta_{i,j}!
\end{equs}
with $T_{i,j} \neq T_{i,\ell}$ for $j \neq \ell$. The coefficient $ \Upsilon_{\text{\scriptsize{root}}}^{\Psi_o}[T] $ in \eqref{recursive_Pi} is given by
\begin{equs}\label{Upsilon_root}
\Upsilon_{\text{\scriptsize{root}}}^{\Psi_o}[T](\mathbf{v},\xi) 
= (\partial_k D_{o_1} \cdots D_{o_n} \hat \Psi_{o}^{\mathfrak{l}})(\mathbf{v},\xi)
\end{equs}
where for $\Psi^{\mathfrak{l}}_o(\mathbf{u}^{\mathfrak{l}}_{o}) = \CB^{\mathfrak{l}}_{o} \left( \prod_{\mathcal{o} \in \Lab_+^{\mathfrak{l},o}} f^{\mathfrak{l}}_{o,\mathcal{o}}(u_{\mathcal{o}})\right)$ we have
\begin{equs}
\hat{\Psi}^{\mathfrak{l}}_o(\mathbf{v},\xi)= \prod_{\mathcal{o} \in \Lab_+^{\mathfrak{l},o}} e^{\xi \mathcal{L}_{\mathcal{o}}} f^{\mathfrak{l}}_{o,\mathcal{o}}(v_{\mathcal{o}}).
\end{equs}
The various derivatives $ \partial, D_{o} $ follow the Leibniz rule and are given by:\begin{equs}
D_{o_i} \mathcal{C}^{\ell} \left[ D_{o_i}^k f^{\mathfrak{l}}_{o,o_i},\mathcal{L}_{o_i} \right] & =   \mathcal{C}^{\ell} \left[ D_{o_i}^{k+1} f^{\mathfrak{l}}_{o,o_i},\mathcal{L}_{o_i} \right],\\ 
\partial \mathcal{C}^{\ell} \left[  D_{o_i}^k f^{\mathfrak{l}}_{o,o_i},\mathcal{L}_{o_i} \right] & =  \mathcal{C}^{\ell+1} \left[  D_{o_i}^k f^{\mathfrak{l}}_{o,o_i},\mathcal{L}_{o_i}  \right],
\end{equs}
where $D_{o_i}$ stands for the derivative with respect to $v_{o_{i}}$. One has $D_{o_{i}} f(v_{o_{j}}) = 0$ for $o_i \neq o_j$ and by convention we set $D_{o_{i}} (e^{\xi \mathcal{L}_{{o_{i}}}} v_{o_{i}}) = 1$. The derivative $ \partial_o $ adds a commutator depending on $ \mathcal{L}_{o_i} $ to $ f_{o,o_i}^{\mathfrak{l}} $.

\begin{remark} The subscript $ o $ in $ \tilde{\Pi}_o $ is needed in the definition for identifying which nonlinearity has to be used in the coefficient $ \Upsilon_{\text{\scriptsize{root}}}^{\Psi_{o}}[T] $. This is also a crucial difference with \cite{BS} where the coefficients $ \Upsilon_{\text{\scriptsize{root}}}^{\Psi_{o}}[T]  $ are detached from the definition of the iterated integral. This is possible due to the fact that one only works in Fourier space in \cite{BS}. We have used the subscript $ \text{\scriptsize{root}} $ to stress that    $\Upsilon_{\text{\scriptsize{root}}}^{\Psi_{o}}[T]  $ computes a coefficient  of $ T $ taking into account only its edges connected to its root. This is in contrast with the classical elementary differential for B-series where all the nodes of the tree are considered for computing it.
In fact, the rest of the coefficient is computed recursively via an iteration of $ \Pi $ in \eqref{recursive_Pi_1}. 

\end{remark}

\begin{example}\label{ex_Gross_cont}
We continue  Example~\ref{ex_Gross} on the Gross-Pitaevskii equation and illustrate the definition of $ \Pi, \tilde{\Pi}_o $. One has 
\begin{equs}
\left( \tilde{\Pi}_{o}  \lambda_0 \right) (\hat{\mathbf{v}},\xi) & =   \left(- ie^{\xi \mathcal{L}} v^2 \right) \left( e^{- \xi \mathcal{L}} \bar{v} \right), \quad \left( \tilde{\Pi}_{\bar{o}}  \lambda_0 \right) (\hat{\mathbf{v}},\xi) =   \left(ie^{-\xi \mathcal{L}} \bar{v}^2 \right) \left( e^{ \xi \mathcal{L}} v \right), \\
\left( \tilde{\Pi}_{o}  \lambda_1 \right) (\hat{\mathbf{v}},\xi) & =   \left(- ie^{\xi \mathcal{L}} v \right) V, \quad \left( \tilde{\Pi}_{\bar{o}}  \lambda_1 \right) (\hat{\mathbf{v}},\xi) =   \left( ie^{-\xi \mathcal{L}} \overline{v} \right) V,
\end{equs}
since
\begin{equs}
S_{\text{\scriptsize{root}}}(\lambda_0) & =  S_{\text{\scriptsize{root}}}(\lambda_1) = 1, \ V_0 =1, \ V_1 = V, \ \mathcal{B}_{o}^{0} = \mathcal{B}_{o}^{1} = \id, \\
\Upsilon_{\text{\scriptsize{root}}}^{\Psi_{o}}[\lambda_0](\mathbf{v},\xi) & =  \left(-ie^{\xi \mathcal{L}} v^2 \right) \left( e^{- \xi \mathcal{L}} \bar{v} \right), \quad
\Upsilon_{\text{\scriptsize{root}}}^{\Psi_{\bar{o}}}[\lambda_0](\mathbf{v},\xi)  =  \left( ie^{-\xi \mathcal{L}} \bar{v}^2 \right) \left( e^{ \xi \mathcal{L}} v \right), \\
\Upsilon_{\text{\scriptsize{root}}}^{\Psi_{o}}[\lambda_1](\mathbf{v},\xi)  & = - ie^{\xi \mathcal{L}} v, \quad \Upsilon_{\text{\scriptsize{root}}}^{\Psi_{\overline{o}}}[\lambda_1](\mathbf{v},\xi)  =  ie^{-\xi \mathcal{L}} \overline{v}.
\end{equs}
Hence,
\begin{equs}
\left( \Pi  \mathcal{I}_o(\lambda_0 )\right) (\hat{\mathbf{v}},\tau) & = \int_{0}^{\tau} e^{(\tau-\xi)\mathcal{L}_{o}} ( \tilde \Pi_{o} \lambda_0)(\hat{\mathbf{v}},\xi)   d \xi \\ & = -i\int_{0}^{\tau} e^{(\tau-\xi)\mathcal{L}_{o}}   \left( e^{\xi \mathcal{L}} v^2 \right) \left( e^{- \xi \mathcal{L}} \bar{v} \right) d\xi.
\end{equs}
Then again, using the definition of the characters $ \Pi$, and $ \tilde{\Pi}_o $ in \eqref{recursive_Pi} , we obtain
\begin{equs}
\left( \tilde{\Pi}_{o}  \lambda^{1}_0 \right) (\hat{\mathbf{v}},\xi) & = \frac{\Upsilon_{\text{\scriptsize{root}}}^{\Psi_{o}}[\lambda_0^1]}{S_{\text{\scriptsize{root}}}(\lambda_0^1)}(\mathbf{v},\xi)V_0 \xi \\
& = - i\xi \left( e^{ \xi \mathcal{L}}\mathcal{C}[  u^2, \mathcal{L}](v) \right)\left( e^{- \xi \mathcal{L}} \bar{v} \right),
\end{equs}
since,
\begin{equs}
\Upsilon_{\text{\scriptsize{root}}}^{\Psi_{o}}[\lambda_0^1](\mathbf{v},\xi) & = (\partial_1 \hat{\Psi}_{o}^0)(\mathbf{v},\xi) \\
& =\partial_1[(e^{\xi\CL} f^{0}_{o,o}(v_{o}))(e^{-\xi\CL} f^{0}_{o,\bar{o}}(v_{\bar{o}}))] \\
& = - i\left( e^{ \xi \mathcal{L}}\mathcal{C}[  u^2, \mathcal{L}](v) \right)\left( e^{- \xi \mathcal{L}} \bar{v} \right)
\end{equs}
and $S_{\text{\scriptsize{root}}}(\lambda_0^1) = 1$.


\end{example}

Let  $ A $ be a given domain of regularity of $ \hat{\mathbf{v}} $ that corresponds to the regularity we assume a priori on the initial data and the potentials.
We introduce an approximation of $ \Pi $ via a new character
 $  \Pi_{A} $ defined on $ \CH $, by:
\begin{equs}
\label{Pi_A}
\begin{aligned}
\left( \Pi_A \CI^r_{o}(T) \right) (\hat{\mathbf{v}},\xi)  & = \mathcal{K}_{o,A}^r ( ( \tilde{\Pi}_{o,A} \CD_{r-1}(T))(\hat{\mathbf{v}},\cdot))(\xi)   
\\
\left(  \tilde{\Pi}_{o,A} \, \lambda^k_{\mathfrak{l}} \prod_{i=1}^n \CI_{o_i}^{r}(T_i) \right)(\hat{\mathbf{v}},\xi) &  = \CB_{o}^{\mathfrak{l}}\left( \frac{\Upsilon_{\text{\scriptsize{root}}}^{\Psi_{o}}[T]}{S_{\text{\scriptsize{root}}}(T)} V_{\mathfrak{l}} \, \xi^{k} \prod_{i=1}^n \left( \Pi_{A} \CI^r_{o_i}(T_i) \right) \right) (\hat{\mathbf{v}},\xi)
\end{aligned} 
\end{equs}
where $ T = \lambda^k_{\mathfrak{l}} \prod_{i=1}^n \CI_{o_i}(T_i) $ and the operator $ \mathcal{K}_{o,A}^r $ is given in Defintion~\ref{def_CK}. As one can notice the approximation follows the recursive definition of $ \Pi $ and $ \tilde{\Pi}_o $. The main difference occurs when one replaces the integral in time given in \eqref{recursive_Pi} by the operator $ \mathcal{K}_{o,A}^r $. This operator is used to give an approximation of order $r$  to  the following integral:
\begin{equs} \label{integral_1}
\int_{0}^{t} e^{(t-s) \mathcal{L}_{o}} \mathcal{B}\left( s^{\ell} \prod_{i=1}^m e^{s \mathcal{L}_i} u_i \right) ds
\end{equs}
where $ \mathcal{B} $, $ \mathcal{L}_i $ are some operators and the $ u_i $ depend on $ \hat{\mathbf{v}} $.
The approximation should take into account the various interactions between the operators $ \mathcal{L}_{o} $ and $ \mathcal{L}_i $. In the sequel, we will use the following shorthand notations:
\begin{equs}
\Pi^r_A = \Pi_A \mathcal{D}_r, \quad \tilde{\Pi}_{o,A}^r = \tilde{\Pi}_{o,A} \mathcal{D}_r
\end{equs}
where now $ \Pi^r_A $ and $ \tilde{\Pi}_{o,A} $ are defined on $ \hat \CH $ and
$ \hat{\mathcal{P}} $ .


\begin{definition} \label{def_CK}
Let us consider a function $F $ of the form 
\begin{equs} 
F(\hat{\mathbf{v}},\xi) = \xi^{\ell} \CB\left( \prod_{i \in J} e^{\xi \mathcal{L}_i} u_i(\hat{\mathbf{v}}) \right)
\end{equs}
 where $\CB$ is a linear operator, the operators $ \mathcal{L}_i $  satisfy Assumption \eqref{assump:L}, and $ J $ is a finite set. We suppose that the $ u_i $ are smooth functions of $\hat{\mathbf{v}}$, and we assume a given domain $A$   for which we have $\hat{\mathbf{v}} \in A$. Let $ r > 0 $, $ o \in \Lab_{+} $ and the operator $\mathcal{L}_{o} $  associated to the decoration $ o $. We assume that $ \CB $ commutes with $\mathcal{L}_o$. Given $(k_i)_{i \in J} $ with $ k_i \in \N $, we define 
\begin{equs}
G_{n,(k_i)_{i\in J}} : \hat{\mathbf{v}} \rightarrow \mathcal{L}_o^n \CB \left( \prod_{i \in J} \mathcal{L}_i^{k_i} u_i(\hat{\mathbf{v}}) \right).
\end{equs}
Note that the domain of $G$ depends on the space $X$ and it's associated norm.
Now we distinguish two cases: 
\begin{itemize}
\item If $A \subseteq D(G_{n,(k_i)_{i\in J}}) $ for all $(k_i)_{i\in J}$ and $n$ be such that, $\sum_i k_i + n \leq r-\ell+1$ we can carry out a Taylor-series expansion of all the operators and set:
 \begin{equs}\label{def_fullTalor}
 &\CK_{o,A}^r(F(\hat{\mathbf{v}},\cdot))(\tau) \\&=  \sum_{\sum_i k_i + n \leq r-\ell }  \int_0^\tau \frac{(\tau-\xi)^n \xi^{\sum_{i \in J} k_i + \ell}}{  n !\prod_{i \in J} k_i !} \mathcal{L}_o^n \CB \left( \prod_{i \in J} \mathcal{L}_i^{k_i} u_i(\hat{\mathbf{v}}) \right) d\xi  .
 \end{equs} 
\item Otherwise, we set
\begin{equs}
\CA_i = \mathcal{L}_i - \mathcal{L}_{o}, \quad
\mathcal{L}_{\tiny{\text{dom}}} = \mathcal{P}_{\tiny{\text{dom}}}( \lbrace\mathcal{A}_1,\cdots, \mathcal{A}_m\rbrace ).
\end{equs}
Then, if $ \mathcal{L}_{\tiny{\text{dom}}} \neq 0 $, let $ I \subset J $  such that for every $ i \in I$, one has
\begin{equs}
\CA_{i} = \mathcal{L}_{\tiny{\text{dom}}} + \CA_{\tiny{\text{low}}}^i, \quad D(\mathcal{L}_{\tiny{\text{dom}}})\subset D(\CA_{\tiny{\text{low}}}^i),
\end{equs}
and for  $i \in J \setminus I$ 
\begin{equs}
\CA_i = \CA_{\tiny{\text{low}}}^i, \quad
 D( \mathcal{L}_{\tiny{\text{dom}}}) \subset D(\CA_{i}) .
\end{equs}
One can rewrite $ F  $  as
\begin{equs}\label{F}
F(\cdot,\xi) = \xi^{\ell} \CB\left( \left(  \prod_{i \in I} e^{\xi (\mathcal{L}_o + \mathcal{L}_{\tiny{\text{dom}}} + \mathcal{A}_{\tiny{\text{low}}}^i ) } u_i \right) \left(  \prod_{i \in J \setminus I} e^{\xi (\mathcal{L}_o + \mathcal{A}_{\tiny{\text{low}}}^i  ) } u_i \right) \right).
\end{equs}
We assume that the operators that appear in \eqref{F} commute and generate a continuous semigroup. If this is not the case, one needs to apply the approximation \eqref{def_fullTalor}.
We have to distinguish  two cases:
\begin{itemize}
\item[(i)]  If $ \mathcal{L}_{\tiny{\text{dom}}} \neq 0 $ then
\begin{equs}\label{K_reson}
& \CK_{o,A}^r(F)(\tau)  =    \int_{0}^{\tau} \sum_{q  \leq r -\ell}\sum_{q =   n + m + p+  \sum_{i \in J} k_i  }  \frac{(\tau-\xi)^{m} \xi^{\ell + \sum_{i \in J} k_i} }{p! m! n! \prod_{i \in J} k_i !}  \\ 
& \CB \left( \mathcal{C}^{m}[ \left( e^{\tau \mathcal{L}_o} \mathcal{C}^{n}[\mathcal{M}_{J \setminus I}, \tau \mathcal{L}_o] \right) \left( e^{\xi \mathcal{L}_{\text{\tiny{dom}}} + \tau \mathcal{L}_o} \mathcal{C}^{p}[ \mathcal{M}_{ I}   ,\xi \mathcal{L}_{\text{\tiny{dom}}} + \tau \mathcal{L}_o ] \right) , \mathcal{L}_{o}] \right. \\ 
 & \left.  (((\CA^{i}_{\text{\tiny{low}}})^{k_i}  u_i )_{i \in J}) \, \right) d \xi .
\end{equs}
\item[(ii)] If $ \mathcal{L}_{\tiny{\text{dom}}} = 0 $ then 
\begin{equs}\label{K_no_reson}
& \CK_{o,A}^r(F)(\tau)  =    \int_{0}^{\tau} \sum_{q  \leq r -\ell}\sum_{q =   n + m + \sum_{i \in J} k_i  }  \frac{(\tau-\xi)^{m} \xi^{\sum_{i \in J} k_i}  }{ m! n! \prod_{i \in J} k_i !}  \\
 & \CB \left( \mathcal{C}^{m}[ \left( e^{\tau \mathcal{L}_o} \mathcal{C}^{n}[\mathcal{M}_{J }, \tau \mathcal{L}_o] \right) , \mathcal{L}_{o}](((\CA_i)^{k_i}  u_i )_{i \in J}) \right) \,  d \xi . 
\end{equs}
\end{itemize}
\end{itemize}
\end{definition}


\begin{remark} \label{general_form} In general, one has to face more complicated products to approximate. They are of the form:
\begin{equs}
F(\hat{\mathbf{v}},\xi) = \xi^{\ell} \CB \left(\prod_{i \in K} \CB_i\left( \prod_{j \in J_i} e^{\xi \mathcal{L}_j} u_{i,j}(\hat{\mathbf{v}}) \right) \right)
\end{equs}
Such a product will not appear in the Gross-Pitaevski equation neither for the first order of 
Sine-Gordon. In that case, one can perform the same analysis computing domimant and lower parts. To be in the cases $ (i) $ and $ (ii) $, there must exist some $ I \subset K $ such that one has to get the following factorisation:
\begin{equs}
F(\hat{\mathbf{v}},\xi) = \xi^{\ell} \CB \left( \prod_{i \in I} \CB_i\left(  \prod_{j \in J_i} e^{\xi (\mathcal{L}_o + \mathcal{L}_{\tiny{\text{dom}}} + \mathcal{A}_{\tiny{\text{low}}}^j ) } u_{i,j} \right) \right. \\ \left( \prod_{i \in K \setminus I}\CB_i\left(  \prod_{j \in J_i } e^{\xi (\mathcal{L}_o + \mathcal{A}_{\tiny{\text{low}}}^j  ) } u_{i,j}  \right) \right)
\end{equs}
 Then, one has just to replace the products $ \mathcal{M} $ by new ones taking into account the operators $ \CB_i $:
\begin{equs}
\prod_{i \in I} \CB_i(\CM_{J_i}), \quad \prod_{ i \in K\setminus I} \CB_i(\CM_{J_i})
\end{equs} 
instead of $ \CM_I $ and $ \CM_{K \setminus I} $.
  If $ F $  happens not to have this form one performs the full Taylor expansion as in \eqref{def_fullTalor}.
\end{remark}

\begin{remark} The 
 $\{ k_i\}_{i\in J}$ appear after having Taylor expanded the lower order parts of the operators. Let us mention that this step is necessary in a coupled system with two differential operators  $ \mathcal{L}_1 $ and $ \mathcal{L}_2 $ that differ from a lower differential operator see equation~\eqref{equation_A_i}.
\end{remark}

\begin{remark}\label{rmq:domainA}  The previous approximation is optimized according to an a priori given domain $A$. 
For solutions and potentials regular enough, for example when the following inclusions hold: $ A \subset D(\mathcal{L_o})$ or/and $  A \subset D(\mathcal{L}_{\text{dom}}) $, then in these cases, one can perform Taylor expansions (see \eqref{def_fullTalor}) in order to simplify the scheme. This approach is deeply used in Section \ref{sec::Gross_2} for the construction of the second order low-regularity scheme for Gross-Pitaevskii, and in Sections \ref{sec::Gross_1} and \ref{sec::Sine} to construct a simplified first-order scheme for the Gross-Pitaevskii and Sine-Gordon equation when we assume more regularity on the solution and/or potential.  Moreover, given that the boundary conditions are encoded in the domain of the operator $\CL_{o}$ and $\CL_{\text{dom}}$, we also define the domain $A$ with the prescribed boundary conditions. 
\end{remark}

\begin{example} \label{ex_5} We illustrate Definition~\ref{def_CK} on a tree coming from the NLS equation with $ A = (H^2(\Omega))^2 $, $\hat{\mathbf{v}} = (v, \bar{v}) $, where for simplicity we assume that $\Omega = \R^d$.
\begin{equs}
\left( \Pi_A \CI^r_{o}(\lambda_{0}) \right) (\hat{\mathbf{v}},\xi)  & = \mathcal{K}_{o,A}^r ( (\tilde{\Pi}_{o,A} \CD_{r-1}(\lambda_0))(\hat{\mathbf{v}},\cdot))(\xi)  \\
\left( \tilde{\Pi}_{o,A} \CD_{r-1}(\lambda_0) \right)(\hat{\mathbf{v}},\xi)  &  =   \left(-i e^{\xi \mathcal{L}} v^2 \right) \left( e^{- \xi \mathcal{L}} \bar{v} \right).
\end{equs}
Then using the notations in Definition~\ref{def_CK}, one gets
\begin{equs} 
F(\hat{\mathbf{v}},\xi) =  \prod_{i \in \lbrace 1,2 \rbrace} e^{\xi \mathcal{L}_i} u_i(\hat{\mathbf{v}}),\quad u_1(\hat{\mathbf{v}}) = -iv^2, \quad u_2(\hat{\mathbf{v}}) = \bar{v}, \\ \mathcal{L}_1 = \mathcal{L}, \quad \mathcal{L}_2 = - \mathcal{L}.
\end{equs} We can now compute the operator interactions as follows
\begin{equs}
\mathcal{A}_1 & = \mathcal{L}_1 - \mathcal{L} = 0, \quad \mathcal{A}_2 = \mathcal{L}_2 - \mathcal{L} = - 2 \mathcal{L} \\
\mathcal{L}_{\tiny{\text{dom}}} & = \mathcal{P}_{\tiny{\text{dom}}}( \lbrace\mathcal{A}_1, \mathcal{A}_2\rbrace )  =\mathcal{P}_{\tiny{\text{dom}}}( \lbrace 0, - 2 \mathcal{L} \rbrace ) = - 2 \mathcal{L}, \quad \CA_{\tiny{\text{low}}}^1 =  \CA_{\tiny{\text{low}}}^2 = 0.
 \end{equs}
We have $\mathcal{L}_{\tiny{\text{dom}}} = - 2 \mathcal{L} \neq 0$, therefore we apply the first point $ (i) $ of Definition~\ref{def_CK}:
 \begin{equs}
 & \CK_{o,A}^r(F)(\tau)  = \int_{0}^{\tau} \sum_{q  \leq r -\ell}\sum_{q =   n + m + p  }  \frac{(\tau-\xi)^{m}  }{p! m! n! \prod_{i \in J} k_i !}  \\ &  \mathcal{C}^{m}[ \left( e^{\tau \mathcal{L}_o} \mathcal{C}^{n}[\mathcal{M}_{J \setminus I}, \tau \mathcal{L}_o] \right) \left( e^{\xi \mathcal{L}_{\text{\tiny{dom}}} + \tau \mathcal{L}_o} \mathcal{C}^{p}[ \mathcal{M}_{ I}   ,\xi \mathcal{L}_{\text{\tiny{dom}}} + \tau \mathcal{L}_o ] \right) , \mathcal{L}_{o}]\\ &((  u_i )_{i \in J}) \,  d \xi. 
 \end{equs}
 We have $ I = \lbrace 2 \rbrace $ and $ J = \lbrace 1,2\rbrace $, therefore for every $ n,p \neq 0 $
 \begin{equs}
 \mathcal{C}^{n}[\mathcal{M}_{J \setminus I}, \tau \mathcal{L}_o] = \mathcal{C}^{p}[ \mathcal{M}_{ I}   ,\xi \mathcal{L}_{\text{\tiny{dom}}} + \tau \mathcal{L}_o ] = 0.
 \end{equs}
 
Hence, from the above computations and by considering the case where $ r=1 $ (i.e., a second order scheme) we obtain the following,
 \begin{equs} 
 & \CK_{o,A}^1(F)(\tau)  =  B + C \\ & B =   \int_{0}^{\tau} (\tau-\xi) \mathcal{C}[ \left( e^{\tau \mathcal{L}_o} \mathcal{M}_{\lbrace 1 \rbrace} \right) \left( e^{\xi \mathcal{L}_{\text{\tiny{dom}}} + \tau \mathcal{L}_o}  \mathcal{M}_{ \lbrace 2 \rbrace} \right), \mathcal{L}_{o}](u_1,u_2) \,  d \xi 
 \\  & C =  \int_{0}^{\tau} \left( e^{\tau \mathcal{L}_o} u_1\right) \left( e^{\xi \mathcal{L}_{\text{\tiny{dom}}} + \tau \mathcal{L}_o}  u_2 \right) \,  d \xi.
 \end{equs}
 In the end, one has
 \begin{equs}
\left( \Pi_A \CI^1_{o}(\lambda_{0}) \right) (\hat{\mathbf{v}},\tau) & = - i  \int_{0}^{\tau} \left( e^{\tau \mathcal{L}} v^2 \right) \left( e^{ (\tau- 2 \xi) \mathcal{L} }  \overline{v} \right) \,  d \xi \\ & +  \int_{0}^{\tau} (\tau-\xi) \mathcal{C}[ \left( e^{\tau \mathcal{L}} \mathcal{M}_{\lbrace 1 \rbrace} \right) \left( e^{(\tau-2\xi) \mathcal{L}}  \mathcal{M}_{ \lbrace 2 \rbrace} \right), \mathcal{L}](-i v^2,\overline{v}) \,  d \xi. 
\end{equs}
\end{example}

\subsection{Local error analysis for approximated iterated integrals}

Before comparing the character $ \Pi $ with $ \Pi_{A} $, one has to understand the error introduced  by the operator $ \CK $. It actually depends on the various cases in Definition~\ref{def_CK}.

\begin{lemma}\label{Taylor_bound} We keep the notations of Definition~\ref{def_CK}. We suppose that $ r \geq \ell $ then one has
\begin{equation}
 \int_{0}^{\tau} e^{(\tau- \xi) \mathcal{L}_o} F(\hat{\mathbf{v}},\xi) d\xi -\CK^{r}_{o,A} (  F(\hat{\mathbf{v}}),\cdot)(\tau) = \CO(\tau^{r+2} R^r_{o,A}(F)(\hat{\mathbf{v}}))
\end{equation}
where $ R^r_{o,A}(F) $ takes the following values:
\begin{itemize}
\item[(i)] If we are using the approximation \eqref{def_fullTalor}, we get the following error:
\begin{equs}
R_{o,A}^r(F) =\sum_{\sum_i k_i + n =  r-\ell + 1 }  \mathcal{L}_o^n \CB \left( \prod_{i \in J} \mathcal{L}_i^{k_i} u_i \right) .
\end{equs}
\item[(ii)]  If $ \mathcal{L}_{\tiny{\text{dom}}} \neq 0 $ then
\begin{equs}
R_{o,A}^r(F)  & =   \sum_{ \sum_{i \in J} k_i + m + n + p = r - \ell +1  }  \CB \left( \mathcal{C}^{m}[ \left(  \mathcal{C}^{n}[\mathcal{M}_{J \setminus I},  \mathcal{L}_o] \right) \right. \\ & \left. \left(  \mathcal{C}^{p}[ \mathcal{M}_{ I}   , \mathcal{L}_{\text{\tiny{dom}}} +  \mathcal{L}_o ] \right) , \mathcal{L}_{o}](((\CA^{i}_{\text{\tiny{low}}})^{k_i}  u_i )_{i \in J}) \right) .
\end{equs}
\item[(iii)] If $ \mathcal{L}_{\tiny{\text{dom}}} = 0 $ then 
\begin{equs}
 R_{o,A}^r(F)  & =   \sum_{ \sum_{i \in J} k_i + m + n + p = r - \ell +1  } \CB \left( \mathcal{C}^{m}[  \mathcal{C}^{n}[\mathcal{M}_{J }, \mathcal{L}_o], \mathcal{L}_{o}](((\CA_i)^{k_i}  u_i )_{i \in J}) \right). 
\end{equs}
\end{itemize}
\end{lemma}
\begin{proof} The proof reduces to applying several time Taylor expansions and using    identity \eqref{commutator}.
\end{proof}

\begin{example} \label{ex_6} We want to compute the error term for the tree $ \mathcal{I}_o(\lambda_0) $ illustrating Lemma~\ref{Taylor_bound}. From Example~\ref{ex_5}, we are in the case where $ D(\mathcal{L}_{\tiny{\text{dom}}}) = D(\mathcal{L}_o) $ such that
\begin{equs}
R^1_{o,A}(\tilde{\Pi}_{o,A}\lambda_0)(\hat{\mathbf{v}}) & = \sum_{  m + n + p = 2  } \mathcal{C}^{m}[ \left(  \mathcal{C}^{n}[\mathcal{M}_{J \setminus I},  \mathcal{L}_o] \right) \\ & \left(  \mathcal{C}^{p}[ \mathcal{M}_{ I}   , \mathcal{L}_{\text{\tiny{dom}}} +  \mathcal{L}_o ] \right) , \mathcal{L}_{o}] (-i v^2,\bar{v} )
\end{equs}
and we got only one term which is non-zero $(m,n,p)= (2,0,0) $:
\begin{equs}
R^1_{o,A}(\tilde{\Pi}_{o,A}\lambda_0)(\hat{\mathbf{v}}) & = \mathcal{C}^2[ \mathcal{M}_{\lbrace  1,2 \rbrace}, \mathcal{L}_{o}] (-i v^2,\bar{v} ) .
\end{equs}
\end{example}

The next proposition follows the steps of \cite[Prop. 3.9]{BS}. It singles out the dominant parts of the oscillations and shows that approximated iterated integrals are connected with Definition~\ref{dom_freq}. This decomposition was also connected to the Birkoff factorisaton given in \cite{BS}. It provides information about the dominant operators that will be involved in the local error analysis.  In the next proposition, we assume that $ \CB_o^{\mathfrak{l}} = \id $. In the general case, one can write a similar statement using Remark~\ref{general_form}.

\begin{proposition} \label{decomp_Pi}
For every tree $ \mathcal{I}^r_{o}( T) \in \CH $, one has the following decomposition:
\begin{equs} \label{simple_formula} \begin{aligned}
  \left(\Pi_A \mathcal{I}^r_{o}( T) \right)(\hat{\mathbf{v}},\xi)    & =  \sum_{\bar{T}  \subset \mathcal{I}_{o}(T)}      \prod_{\mathcal{L} \in \CR_{\text{\tiny{dom}}}( \bar T)} \left(  e^{i \xi \mathcal{L} } B^r \left(T, \bar T, \mathcal{L} \right)(\hat{\mathbf{v}},\xi) \right) \\ + & 
  \left( e^{i \xi \mathcal{L}_o} B^r \left(T, \one, \mathcal{L}_o \right)(\hat{\mathbf{v}},\xi) \right) \left( e^{i \xi \mathcal{L}_o} C^r \left(T, \one, \mathcal{L}_o \right)(\hat{\mathbf{v}},\xi) \right)
 \end{aligned}
\end{equs}
where the sum is over all subtrees $ \bar T \neq \one$ that have at least one edge and that contain the root of $  \CI_{o}(T) $. The $ B^r \left(T, \bar T, \mathcal{L} \right)(\hat{\mathbf{v}},\xi) $, $ B^r \left(T, \one, \mathcal{L}_o \right)(\hat{\mathbf{v}},\xi)$ and $ C^r \left(T, \one, \mathcal{L}_o \right)(\hat{\mathbf{v}},\xi) $ are polynomials in $\xi$ and the  $\CR_{\text{\tiny{dom}}}(\bar T)  $ are given in Definition~\ref{dom_freq}. 
\end{proposition}
\begin{proof} We proceed by induction on the size of $ T $ and we suppose that $ T = \lambda^{\ell}_{\mathfrak{l}}\prod_i \mathcal{I}_{o_i}(T_i) $ where one can apply the induction hypothesis on each  $\mathcal{I}_{o_i}(T_i)$. One has from \eqref{Pi_A}
\begin{equs}
\left(\Pi_A \CI^r_{o}(T)  \right)(\hat{\mathbf{v}},\xi) = \CK_{o,A}^{r}\left(  \tilde{\Pi}_{o,A} \left( \CD_{(r-1)}( T) \right)(\hat{\mathbf{v}},\cdot) \right)(\xi)
\end{equs}
and
\begin{equs}
\tilde{\Pi}_{o,A} \left(  \CD_{(r-1)}( T) \right)(\hat{\mathbf{v}},\xi) = \frac{\Upsilon_{\text{\scriptsize{root}}}^{\Psi_{o}}[T]}{S_{\text{\scriptsize{root}}}(T)}(\mathbf{v},\xi) V_{\mathfrak{l}} \, \xi^{\ell} \prod_{i} \left(\Pi_A \CI^{(r-1-\ell)}_{o_i}(  T_i) \right) (\hat{\mathbf{v}},\xi).
\end{equs}
By applying the induction hypothesis on the $ \mathcal{I}_{o_i}(T_i)$, one has the following decomposition
\begin{equs}
\tilde{\Pi}_{o,A} & \left(  \CD_{(r-1)}( T) \right)(\hat{\mathbf{v}},\xi)  =  \frac{\Upsilon_{\text{\scriptsize{root}}}^{\Psi_{o}}[T]}{S_{\text{\scriptsize{root}}}(T)}(\mathbf{v},\xi) V_{\mathfrak{l}} \prod_i \sum_{\bar{T}_i  \subset \mathcal{I}_{o_i}(T_i)}       \\ &  \prod_{\mathcal{L} \in \CR_{\text{\tiny{dom}}}( \bar T_i)} \left(  e^{i \xi \mathcal{L} } B^{r-1-\ell} \left(T_i, \bar T_i, \mathcal{L} \right)(\hat{\mathbf{v}},\xi) \right) \\ & +  \left( e^{i \xi \mathcal{L}_{o_i}} B^r \left(T_i, \one, \mathcal{L}_{o_i} \right)(\hat{\mathbf{v}},\xi) \right) \left( e^{i \xi \mathcal{L}_{o_i}} C^r \left(T_i, \one, \mathcal{L}_{o_i} \right)(\hat{\mathbf{v}},\xi) \right)
\end{equs}
where the $ B^r \left(T_i, \bar T_i, \mathcal{L} \right)(\hat{\mathbf{v}},\xi) $, $   B^r \left(T_i, \one, \mathcal{L}_{o_i} \right)(\hat{\mathbf{v}},\xi) $ and $ C^r \left(T_i, \one, \mathcal{L}_{o_i} \right)(\hat{\mathbf{v}},\xi)  $ are polynomials in $ \xi $. As a consequence of Definition~\ref{def_CK}, the map  $  \CK_{o,A}^{r} $ applied   to the term coming from a subforest  $ \bar T = \lambda^{\ell}_{\mathfrak{l}} \prod_i \bar T_i $, integrates exactly in the end a term of the form
\begin{equs}
 \int_0^{\xi}
\left( e^{\xi \CL_{o}}P(\hat{\mathbf{v}},s)  \right) \left(       e^{s \CL_{\text{\tiny{dom}}}(\CI_o(\bar T)) + \xi \mathcal{L}_o} Q(\hat{\mathbf{v}},s) \right)  
 d s 
\end{equs}
where $ P(\hat{\mathbf{v}},s) $ and $ Q(\hat{\mathbf{v}},s) $ are polynomials in $ s $. 
Then, by computing this integral, we obtain two terms of the form
\begin{equs}
\left( e^{s \mathcal{L}_o}P_1(\hat{\mathbf{v}},s) \right) \left(        e^{i s \mathcal{L}_{\text{\tiny{dom}}}(\CI_o(\bar T))} Q_1(\hat{\mathbf{v}},s) \right), \quad \left( e^{s \mathcal{L}_o}P_2(\hat{\mathbf{v}},s) \right) \left(  e^{s \mathcal{L}_o}Q_2(\hat{\mathbf{v}},s) \right)
\end{equs}
where $ P_1, Q_1  $ and $ P_2 $ are also polynomials in $ s $. The first term corresponds to the subtree $ \CI_o(\bar T)$ and the second term corresponds to $\one $. Therefore, the structure of \eqref{simple_formula} is preserved.
\end{proof}
The next definition computes the local error using the recursive construction of the decorated trees. It follows the same structure as in \cite{BS}, the main difference being that now $ \Upsilon_{\text{\scriptsize{root}}}^{\Psi_o}$ has to be added to this definition.
\begin{definition}\label{def:Llow}
Let $ r \in \Z $ and a domain $ A $ of $ \hat{\mathbf{v}} $. We recursively define $ \mathcal{L}^{r}_{\text{\tiny{low}}}(\cdot,\hat{\mathbf{v}},A)$ and $ \mathcal{L}^{r,o}_{\text{\tiny{low}}}(\cdot,\hat{\mathbf{v}},A)$, $ o \in \Lab_+ $  as
\begin{equs}
\mathcal{L}^{r}_{\text{\tiny{low}}}(T,\hat{\mathbf{v}},A) = 1, \quad r < 0, \quad \mathcal{L}^{r}_{\text{\tiny{low}}}(\one,\hat{\mathbf{v}},A) = 1 .
\end{equs}
Else for $ T =  \lambda^{\ell}_{\mathfrak{l}}\prod_i \CI_{a_i}(T_i) $, one has
\begin{equs}\label{local_error_1}
\mathcal{L}^{r,o}_{\text{\tiny{low}}}&(\lambda^{\ell}_{\mathfrak{l}}\prod_i \CI_{a_i}(T_i),\hat{\mathbf{v}},A)   = \\ & \CB_o^{\mathfrak{l}} \left(  \frac{\Upsilon_{\text{\scriptsize{root}}}^{\Psi_o}[T]}{S_{\text{\scriptsize{root}}}(T)}(\mathbf{v},0) V_{\mathfrak{l}}   \left(  \sum_i 
  \mathcal{L}^{r-\ell}_{\text{\tiny{low}}}( \CI_{a_i}(T_i),\hat{\mathbf{v}},A) \right) \right).
\end{equs}
And 
\begin{equs} \label{local_error_2} \begin{aligned}
\mathcal{L}^{r}_{\text{\tiny{low}}}(\CI_{o}( T ),\hat{\mathbf{v}},A)  =    \mathcal{L}^{r-1,o}_{\text{\tiny{low}}}(  T, \hat{\mathbf{v}},A) +     R^r_{o,A}\left(\left(\tilde{\Pi}_{o,A}^{r-1} T \right)(\hat{\mathbf{v}},\cdot)  \right) .
\end{aligned}
\end{equs}
\end{definition}
\begin{example} We continue Example~\ref{ex_6}.
For the decorated tree $ T = \mathcal{I}_o(\lambda_0) $, one gets:
\begin{equs}
\mathcal{L}^{1}_{\text{\tiny{low}}}( T,\hat{\mathbf{v}},A) & =     \mathcal{L}^{0}_{\text{\tiny{low}}}(  \lambda_0, \hat{\mathbf{v}},A) +     R^1_{o,A}\left(\left(\tilde{\Pi}_{o,A}^0  T \right)(\hat{\mathbf{v}},\cdot)  \right) 
\\ & = - i v^2 \bar{v} +  \mathcal{C}^2[ \mathcal{M}_{\lbrace  1,2 \rbrace}, \mathcal{L}_{o}] ( -i v^2,\bar{v} ).
\end{equs}
One can notice that the first term asks less regularity in comparison to the second.
\end{example}
The next theorem shows that the characters $ \Pi_A $ and $ \tilde{\Pi}_{o,A} $ are  good approximations of 
$ \Pi $ and $ \tilde{\Pi}_o $
 with a local error given by Definition~\ref{def:Llow}. We follow the same steps as in the proof of \cite[Thm 3.17]{BS}. 
\begin{theorem} \label{approxima_tree}
For $ T =  \lambda^{\ell}_{\mathfrak{l}}\prod_i \CI_{a_i}(T_i) $ one has,
\begin{equs} \label{bound_1}
\left(\tilde{\Pi}_o T - \tilde{\Pi}_{o,A} \mathcal{D}_r(T) \right)(\hat{\mathbf{v}},\tau)  = \mathcal{O}\left( \tau^{r+2} \mathcal{L}^{r,o}_{\text{\tiny{low}}}(T,\hat{\mathbf{v}},A) \right)
\end{equs}
and for $ T =  \CI_{o}(\bar T )$, one gets
\begin{equs} \label{bound_2}
\left(\Pi T - \Pi_A \mathcal{D}_r(T) \right)(\hat{\mathbf{v}},\tau)  = \mathcal{O}\left( \tau^{r+2} \mathcal{L}^{r}_{\text{\tiny{low}}}(T,\hat{\mathbf{v}},A) \right).
\end{equs}
\end{theorem}
\begin{proof} We proceed by induction by using the recursive definition \eqref{Pi_A} of $ \Pi_A $ and $ \tilde{\Pi}_{o,A} $. First, one gets
\begin{equs}
\left(\Pi - \Pi_A \right)(\one)(\hat{\mathbf{v}},\tau)  = 0 = \mathcal{O}\left( \tau^{r+2} \mathcal{L}^{r}_{\text{\tiny{low}}}(\one,\hat{\mathbf{v}},A) \right).
\end{equs}
Then for $T =  \lambda^{\ell}_{\mathfrak{l}}\prod_i \CI_{a_i}(T_i)$ and every $ o \in \Lab_+ $, one has
\begin{equs} 
 \left(\tilde{\Pi}_o  - \tilde{\Pi}_{o,A}^r \right)\left(T \right)(\hat{\mathbf{v}},\tau) & = \tau^{\ell} \CB_o^{\mathfrak{l}}\left( \frac{\Upsilon_{\text{\scriptsize{root}}}^{\Psi_{o}}[T]}{S_{\text{\scriptsize{root}}}(T)}(\mathbf{v},\tau) V_{\mathfrak{l}} \sum_{i}  \left(\Pi - \Pi_A^{r-\ell}  \right)(\CI_{a_i}(T_i))(\hat{\mathbf{v}},\tau)  \right. \\ & \left. \prod_{j \neq i} \left( \Pi_A^{r-\ell} + \Pi \right) (\CI_{a_j}(T_j))(\hat{\mathbf{v}},\tau) \right)  .
 \end{equs}
  Then by applying the induction hypothesis \eqref{bound_2} to each $ \CI_{a_i}(T_i) $, one gets
  \begin{equs}
  \left(\tilde{\Pi}_o  - \tilde{\Pi}_{o,A}^r \right)\left( T \right)(\hat{\mathbf{v}},\tau) &   = \mathcal{O} \left( \tau^{r+2} \CB_{o}^{\mathfrak{l}}\left(  \frac{\Upsilon_{\text{\scriptsize{root}}}^{\Psi_{o}}[T]}{S_{\text{\scriptsize{root}}}(T)}(\mathbf{v},0) V_{\mathfrak{l}}  \sum_i 
\mathcal{L}^{r-\ell}_{\text{\tiny{low}}}( \CI_{a_i}(T_i),\hat{\mathbf{v}},A) \right) \right)
\\ & = \mathcal{O} \left( \tau^{r+2}\mathcal{L}^{r}_{\text{\tiny{low}}}(\lambda^{\ell}_{\mathfrak{l}}\prod_i \CI_{a_i}(T_i),\hat{\mathbf{v}},A) \right).
  \end{equs}
For $ T =  \CI_{o}(  \bar T ) $, one has
\begin{equs}
&\left( \Pi-\Pi^{r}_A \right) \left( T \right) (\hat{\mathbf{v}},\tau)  = \int_{0}^{\tau}  e^{ \xi \mathcal{L}_{o}} (\tilde{\Pi}_o- \tilde{\Pi}^{r- 1}_{o,A})( \bar T)(\hat{\mathbf{v}},\xi) d \xi 
\\ &  + \int_{0}^{\tau} e^{ \xi \mathcal{L}_{o}}  (\tilde{\Pi}_{o,A}^{r-1}  \bar T)(\hat{\mathbf{v}},\xi) d \xi -\CK^{r}_{o} (  (\tilde{\Pi}_{o,A}^{r-1} T )(\hat{\mathbf{v}},\cdot) )(\tau)
 \\ & = \int_{0}^{\tau} \CO \left( \xi^{r+1} \mathcal{L}^{r-1,o}_{\text{\tiny{low}}}(\bar T,\hat{\mathbf{v}},A) \right)  d\xi +  \CO(\tau^{r+2} R^r_{o,A}\left(\left(\tilde{\Pi}_{o,A}^{r-1} \bar T \right)(\hat{\mathbf{v}},\cdot)  \right) ) \\
& =   \CO \left( \tau^{r+2} \mathcal{L}^{r}_{\text{\tiny{low}}}(T,\hat{\mathbf{v}},A) \right) .
\end{equs}
where the term $  R^r_{o,A}\left(\left(\tilde{\Pi}^{r-1}_{o,A}  \bar T \right)(\hat{\mathbf{v}},\cdot) \right)  $  is obtained by applying Lemma~\ref{Taylor_bound}.
\end{proof}
\section{Low regularity numerical scheme}\label{sec:genScheme}

The writing of our low regularity numerical scheme follows two steps. The first one is to write a truncated decorated trees series that will solve the Duhamel's formula up to order $ r $ (see Proposition~\ref{exact_solution_order_r}). This series is formed of iterated integrals produced by the character $ \Pi $. Then, one replaces $ \Pi $ by $ \Pi_A $ to get the scheme. Such steps were already used in \cite{BS}. Let us mention that the first step is more involved here due to the fact that we work with a more general setting.

\subsection{Exact solution up to order $ r $}

Recall the mild solution of  \eqref{ev} given by Duhamel's formula
\begin{equation}\label{duhLin_it}
u_o(t) = e^{ t  \mathcal{L}_o} v_o  + \sum_{\mathfrak{l} \in \Lab_-} \int_0^t e^{ (t-\xi)  \mathcal{L}_o} \Psi_o^{\mathfrak{l}}(\mathbf{u}_o^{\mathfrak{l}}) V_{\mathfrak{l}} d\xi ,
\end{equation}
where the nonlinearity $ \Psi_o^{\mathfrak{l}}(\mathbf{u}_o^{\mathfrak{l}}) $ is given by
\begin{equs}
\Psi_{o}^{\mathfrak{l}}(\mathbf{u}^{\mathfrak{l}}_o) = \mathcal{B}^{\mathfrak{l}}_{o}\left(\prod_{\mathcal{o} \in \Lab_+^{\mathfrak{l},o}} f^{\mathfrak{l}}_{o,\mathcal{o}}(u_{\mathcal{o}})\right).
\end{equs}
 In the following we want to construct a scheme with a local error of order 
 $$\CO(\tau^{r+2}).$$
 Therefore, 
before describing our numerical scheme, we need to remove the trees which are already of size $\CO(\tau^{r+2})$. Indeed, a simple recursion shows that one has for every tree $ \mathcal{I}_o(T) $
\begin{equs} \label{bounds_integral}
(\Pi \CI_o(T) )( \hat{\mathbf{v}},\tau) = \CO( p(T,\hat{\mathbf{v}})\tau^{n_+(T)})
\end{equs} 
where $ p(T,\hat{\mathbf{v}}) $ is a polynomial in the variables $ \hat{\mathbf{v}} $ that does not contain any operators $( \mathcal{L}_o)_{o \in \Lab_+} $ and at most $ |N_T| $ operators $ \CB $.  The map $n_+ $ is defined on $ T^{\Labn, \mathfrak{f}}_{\Labe} $ as 
\begin{equs}
n_+( T^{\Labn, \mathfrak{f}}_{\Labe} ) = 
\sum_{v \in N_T} \Labn(v) + |E_T|
\end{equs}
which corresponds to the number of integrations in time and polynomial decorations. 
We define the space of decorated trees $  \CT_{o}^r \subset \CT_{o} $ as 
\begin{equs} 
\CT_{o}^r = \lbrace  \CI_o(T), \, T  \in \CT^r  \rbrace, \quad \CT^r = \lbrace  T \in  \CT, \, n_+(T)  \leq r, \, \left( \Pi \CI_o(T) \right) \neq 0 \rbrace .
\end{equs}
The condition $ \left( \Pi \CI_o(T) \right) \neq 0 $ guarantees that we consider only trees which have been generated by iterations of Duhamel's formula~\ref{duhLin_it}. 
In the following we consider
\begin{equs} \label{expansion_order_r}
w^r_o(\hat{\mathbf{v}},\tau) = e^{\tau \mathcal{L}_o} v_o + \sum_{T \in \CT^r}  \left( \Pi \CI_{o}(T) \right)(\hat{\mathbf{v}},\tau)
\end{equs}
which solves \eqref{duhLin_it} up to order $ r +1$ in the following sense:

\begin{proposition} \label{exact_solution_order_r} One has that
\begin{equs}
w^r_o(\hat{\mathbf{v}},t)  =
e^{ \tau  \mathcal{L}_o} v_o  + \sum_{\mathfrak{l} \in \Lab_-} \int_0^\tau e^{ (\tau-\xi)  \mathcal{L}_o} \Psi_o^{\mathfrak{l}}(\mathbf{w}^{\mathfrak{l},r}_{o})(\hat{\mathbf{v}},\xi) V_{\mathfrak{l}} d\xi + \mathcal{O}(\tau^{r+2})
\end{equs}
where $ \mathbf{w}^{\mathfrak{l},r}_{o} $ denotes the term $ (w_{\bar o}^{r})_{\bar{o} \in \Lab_+^{\mathfrak{l},o}} $ and the remainder $ \mathcal{O}(\tau^{r+2})$ involves  commutators under the form:
\begin{equs}
\mathcal{C}^{r}[f,\mathcal{L}_o]
\end{equs}
(however, not full powers of $\mathcal{L}_o^r$) with $ f $ a function coming from the coefficients $ \Psi_o $.
\end{proposition}
\begin{proof}
Let $ B $ be given by
\begin{equs}
B & = e^{ \tau  \mathcal{L}_o} v_o  + \sum_{\mathfrak{l} \in \Lab_-} \int_0^\tau e^{ (\tau-\xi)  \mathcal{L}_o} \Psi_o^{\mathfrak{l}}(\mathbf{w}^{\mathfrak{l},r}_{o})(\hat{\mathbf{v}},\xi) V_{\mathfrak{l}} d\xi \\ & =  e^{ \tau  \mathcal{L}_o} v_o  + \sum_{\mathfrak{l} \in \Lab_-} \int_0^\tau e^{ (\tau-\xi)  \mathcal{L}_o} \mathcal{B}_{o}^{\mathfrak{l}}\left( \prod_{\bar{o} \in \Lab^{ \mathfrak{l},o}_+} f^{\mathfrak{l}}_{o, \bar{o}}(w_{\bar{o}}^r(\hat{\mathbf{v}},\xi)) \right) V_{\mathfrak{l}} d\xi .
\end{equs}
Then, we have
\begin{equs}
w^r_o(\hat{\mathbf{v}},\xi) = e^{ \xi \mathcal{L}_o} v_o + R_{r,o}(\hat{\mathbf{v}},\xi), \quad R_{r,o}(\hat{\mathbf{v}},\xi) =  \sum_{T \in \CT^r}  \left( \Pi \CI_{o}(T) \right)(\hat{\mathbf{v}},\xi).
\end{equs}
By performing Taylor expansions around the point $ e^{ \xi \mathcal{L}_o} v_o $, one gets:
\begin{equs}
B  = e^{ \tau  \mathcal{L}_o} v_o  + \sum_{\mathfrak{l} \in \Lab_-} \int_0^\tau & e^{ (\tau-\xi)  \mathcal{L}_o} \CB_o^{\mathfrak{l}} \sum_{\sum_{\bar{o}} k_{\bar{o}} \leq r} \frac{1}{\prod_{\bar{o}} k_{\bar{o}} !} \\ & \prod_{\bar{o} \in \Lab_+^{\mathfrak{l},o}}  R^{k_{\bar{o}}}_{r,\bar{o}}(\hat{\mathbf{v}},\xi) f_{o,\bar{o}}^{\mathfrak{l},(k_{\bar{o}})}(e^{ \xi \mathcal{L}_{\bar{o}}} v_{\bar{o}}) V_{\mathfrak{l}} d\xi  + \mathcal{O}(\tau^{r+2})
\end{equs}
where the sum $ \sum_{\bar{o}}  $ and the product $ \prod_{\bar{o}} $ run over $ \bar{o} \in \Lab_+^{\mathfrak{l},o} $. The error 
$\mathcal{O}(\tau^{r+2})$ depends on the derivatives of $ f_{o,\bar{o}}^{\mathfrak{l}} $ and on a polynomial in $ \hat{\mathbf{v}} $ using the bound \eqref{bounds_integral}. The next step of the approximation is to pull out the term $e^{ \xi \mathcal{L}_{\bar{o}}} v_{\bar{o}}  $:
\begin{equs}
f_{o,\bar{o}}^{\mathfrak{l},(k_{\bar{o}})}(e^{ \xi \mathcal{L}_{\bar{o}}} v_{\bar{o}})  & = e^{ \xi \mathcal{L}_{\bar{o}}} f(v_{\bar{o}}) + \sum_{\ell_{\bar{o}} = 1}^{r} \frac{\xi^{\ell_{\bar{o}}}}{\ell_{\bar{o}}!}  e^{ \xi \mathcal{L}_{\bar{o}}} \mathcal{C}^{\ell_{\bar{o}}}[f_{o,\bar{o}}^{\mathfrak{l},(k_{\bar{o}})},\mathcal{L}_{\bar{o}}](v_{\bar{o}})  \\ & + \mathcal{O}(\xi^{r+1} \mathcal{C}^{r+1}[f_{o,\bar{o}}^{\mathfrak{l},(k_{\bar{o}})},\mathcal{L}_{\bar{o}}](v_{\bar{o}})).
\end{equs}
The remainder coming from this approximation is the leading error. It involves $ \mathcal{L}_o $ with an iterated commutator. For the next computations, we will omit this error and write all the identities up to some terms which are neglectable in comparison to this error.
Inserting the previous expansion into $ A $ and neglecting the terms which are of order bigger than $  \mathcal{O}(\tau^{r+2}) $, we get:
\begin{equs}
B & =  e^{ \tau \mathcal{L}_o} v_o  + \sum_{\mathfrak{l} \in \Lab_-} \int_0^\tau e^{ (\tau-\xi)  \mathcal{L}_o}  \sum_{\sum_{\bar{o}} \ell_{\bar{o} }+ \sum_{\bar{o} ,i}  n_+(T_{\bar{o},i}) \leq r} \frac{\xi^{\sum_{\bar{o}} \ell_{\bar{o}}}}{\prod_{\bar{o}} k_{\bar{o}} !\ell_{\bar{o}} !} \\ & \CB_o^{\mathfrak{l}} \left(  \prod_{\bar{o} \in \Lab_+^{\mathfrak{l},o}} 
\prod_{i= 1}^{k_{\bar{o}}} \left( \Pi \CI_{\bar{o}}(T_{\bar{o},i}) \right)(\hat{\mathbf{v}},\xi) \left( e^{ \xi \mathcal{L}_{\bar{o}}} \mathcal{C}^{\ell_{\bar{o}}} [ f_{o,\bar{o}}^{\mathfrak{l},(k_{\bar{o}})}, \mathcal{L}_{\bar{o}} ] ( v_{\bar{o}}) \right) V_{\mathfrak{l}} \right) d\xi  
\end{equs}
where the $ T_{\bar{o},i} $ are decorated trees and   $  \sum_{\bar{o} ,i}  n_+(T_{\bar{o},i}) $ is a shorthand notation for
\begin{equs}
\sum_{\bar{o} \in \Lab_+^{\mathfrak{l},o}} \sum_{k_{\bar{o}}} \sum_{i=1}^{k_{\bar{o}}}  n_+(T_{\bar{o},i}).
\end{equs}
Then, if we fix the product $ \prod_{i=1}^{k_{\bar{o}}} \CI_{\bar{o}}(T_{\bar{o},i}) $, one has:
\begin{equs}
\prod_{i=1}^{k_{\bar{o}}} \CI_{\bar{o}}(T_{\bar{o},i}) = 
\prod_{i} \CI_{\bar{o}}(\tilde{T}_{\bar{o},i})^{k_{\bar{o},i}}
\end{equs}
where the $ \tilde{T}_{\bar{o},i} $ are distincts and $ \sum_{i} k_{\bar{o},i} = k_{\bar{o}}$. This term appears in $ B $ with the following combinatorial coefficient:
\begin{equs}
\frac{k_{\bar{o}}!}{\prod_{i} k_{\bar{o},i} !}.
\end{equs}
Let us notice
\begin{equs}
S_{\text{\scriptsize{root}}} \left( \lambda^{k}_{\mathfrak{l}}\prod_{\bar{o},i} \CI_{\bar{o}}(\tilde{T}_{\bar{o},i})^{k_{\bar{o},i}}  \right) =  k! \prod_{\bar{o},i} k_{\bar{o},i}!
\end{equs}
and 
\begin{equs}
 \sum_{\sum_{\bar{o}} \ell_{\bar{o}} = \ell}\prod_{\bar{o} \in \Lab_+^{\mathfrak{l},o}} \frac{1}{\ell_{\bar{o}}!}
\left( e^{ \xi \mathcal{L}_{\bar{o}}} \mathcal{C}^{\ell_{\bar{o}}} [ f_{o,\bar{o}}^{\mathfrak{l},(k_{\bar{o}})}, \mathcal{L}_{\bar{o}} ] ( v_{\bar{o}}) \right) = \frac{\partial_{ \ell}}{\ell !}\prod_{\bar{o} \in \Lab_+^{\mathfrak{l},o}} 
\left( e^{ \xi \mathcal{L}_{\bar{o}}}  f_{o,\bar{o}}^{\mathfrak{l},(k_{\bar{o}})}( v_{\bar{o}}) \right)  .
\end{equs}
In this identity, we exploit the Leibniz rule for the derivative $ \partial $. It can be understood as a variant of Faa di Bruno formula see \cite[Lem. A.1]{BCCH}.
We can rewrite $ B $ into
\begin{equs}
B & =  e^{ \tau  \mathcal{L}_o} v_o  + \sum_{\mathfrak{l} \in \Lab_-} \int_0^\tau e^{ (\tau-\xi)  \mathcal{L}_o} \CB_{o}^{\mathfrak{l}} \left( \sum_{\ell+ \sum_{\bar{o},i}  n_+(\tilde{T}_{\bar{o},i}) \leq r} \frac{\xi^{\ell}}{\ell!}\partial_{\ell} \right. \\  & \left. \prod_{\bar{o} \in \Lab_+^{o,\mathfrak{l}}} 
\prod_{i} \frac{1}{ k_{\bar{o},i} !} \left( \Pi \CI_{\bar{o}}(\tilde{T}_{\bar{o},i})^{k_{\bar{o},i}} \right)(\hat{\mathbf{v}},\xi) \left( e^{ \xi \mathcal{L}_{\bar{o}}}  f_{o,\bar{o}}^{\mathfrak{l},(k_{\bar{o}})}( v_{\bar{o}}) \right) V_{\mathfrak{l}} \right) d\xi  
\\ & =  e^{ \tau  \mathcal{L}_o} v_o  + \sum_{\mathfrak{l} \in \Lab_-} \int_0^\tau e^{ (\tau-\xi)  \mathcal{L}_o} \CB_o^{\mathfrak{l}}\left( \sum_{\ell+ \sum_{\bar{o},i}  n_+(\tilde{T}_{\bar{o},i}) \leq r} \frac{\xi^{\ell}}{S_{\text{\scriptsize{root}}} \left( \lambda^{\ell}_{\mathfrak{l}}\prod_{\bar{o},i} \CI_{\bar{o}}(\tilde{T}_{\bar{o},i})^{k_{\bar{o},i}}  \right)} \right. \\ &\left. \prod_{\bar{o} \in \Lab_+^{\mathfrak{l},o}} 
\prod_{i}  \left( \Pi \CI_{\bar{o}}(\tilde{T}_{\bar{o},i})^{k_{\bar{o},i}} \right)(\hat{\mathbf{v}},\xi) \Upsilon^{\Psi_o} _{\text{\scriptsize{root}}} \left( \lambda^{\ell}_{\mathfrak{l}}\prod_{\bar{o},i} \CI_{\bar{o}}(\tilde{T}_{\bar{o},i})^{k_{\bar{o},i}}  \right)(\mathbf{v},\xi) V_{\mathfrak{l}} \right) d\xi  .
\end{equs}
Then, 
\begin{equs}
B  & =  e^{ \tau \mathcal{L}_o} v_o  + \sum_{\mathfrak{l} \in \Lab_-} \int_0^\tau e^{ (\tau-\xi)  \mathcal{L}_o}  \sum_{\ell+ \sum_{\bar{o},i}  n_+(\tilde{T}_{\bar{o},i}) \leq r} \left(\tilde{\Pi}_{o}  \lambda^{\ell}_{\mathfrak{l}}\prod_{\bar{o},i} \CI_{\bar{o}}(\tilde{T}_{\bar{o},i})^{k_{\bar{o},i}}  \right)(\hat{\mathbf{v}},\xi) d\xi \\
& = e^{ \tau  \mathcal{L}_o} v_o  + \sum_{T \in \mathcal{T}^r} \int_0^\tau e^{ (\tau-\xi)  \mathcal{L}_o} \left(\tilde{\Pi}_{o} T \right)(\hat{\mathbf{v}},\xi)  d\xi
\\ & = e^{ \tau  \mathcal{L}_o} v_o  + \sum_{T \in \mathcal{T}^r} \left(\Pi \mathcal{I}_o(T) \right)(\hat{\mathbf{v}},\tau) = w^r_o(\hat{\mathbf{v}},\tau)
\end{equs}
which concludes the proof.
\end{proof}
\subsection{Numerical scheme and local error analysis}\label{subsec:num-erranal}

Now, we are able to describe the general numerical scheme:
\begin{definition}[The general numerical scheme] \label{scheme} For fixed $  r \in \N $ and a domain $ A $, we define the general numerical scheme as:
\begin{equs}\label{genscheme}
u_{o,A}^{r}( \hat{\mathbf{v}},\tau) =e^{ \tau  \mathcal{L}_o} v_o + \sum_{T \in \CT^{r+1}} \Pi_A^r \left( \CI_{o}(T) \right)(\hat{\mathbf{v}},\tau).
\end{equs}
\end{definition}
The numerical scheme \eqref{genscheme}  approximates  the exact solution locally up to order $r+2$. More precisely, the following theorem holds.
\begin{theorem}[Local error]\label{thm:genloc} 
The numerical scheme \eqref{genscheme}  with initial value $v_{o} = u_o(0)$ approximates the exact solution $u_o $ up to a  local error of type
\begin{equs}
u^{r}_{o,A}(\hat{\mathbf{v}},\tau) - u_{o}(\hat{\mathbf{v}},\tau) = \sum_{T \in \CT^{r+1}} \CO\left(\tau^{r+2} \CL^{r}_{\text{\tiny{low}}}(\CI_{o}(T),\hat{\mathbf{v}},A)  \right)
\end{equs}
where the operator $ \CL^{r}_{\text{\tiny{low}}}(\CI_{o}(T),\hat{\mathbf{v}},A)  $, given in Definition \ref{def:Llow}, embeds the necessary regularity of the solution.
\end{theorem}
\begin{proof}
We recall that the exact solution $w^r$ up to order $ r $ is given by 
\begin{equs}
w^r_o(\hat{\mathbf{v}},\tau) = e^{\tau \mathcal{L}_o} v_o + \sum_{T \in \CT^r}  \left( \Pi \CI_{o}(T) \right)(\hat{\mathbf{v}},\tau)
\end{equs}
which satisfies from Proposition~\ref{exact_solution_order_r}
\begin{equs}\label{app1}
 u_o(\tau) - w^r_o(\tau)  = 
 \CO\left( \tau^{r+2} p (\hat{\mathbf{v}},(\mathcal{L}_o)_{o \in \Lab_+})\right)
\end{equs}
 for some polynomial $ p$ such that every $ \mathcal{L}_o $ appear under the form $ \mathcal{C}^{r+1}[f,\mathcal{L}_o] $. Thanks to  Proposition~\ref{approxima_tree} we furthermore obtain that
\begin{equs}\label{app2}
u_{o,A}^{r} & (\hat{\mathbf{v}},\tau) - w_{o}^{r}(\hat{\mathbf{v}},\tau) 
= \sum_{T \in \CT^{r}}  \left( \Pi -\Pi_A^{r} \right) \left( \CI_{o}(T)\right)(\hat{\mathbf{v}},\tau) \\
& = \sum_{T \in \CT^{r}} \CO\left(\tau^{r+2} \CL^{r}_{\text{\tiny{low}}}(T,\hat{\mathbf{v}},A)  \right).
\end{equs}
Next we write
\begin{equs}
u_{o,A}^{r}(\hat{\mathbf{v}},\tau) -u_{o}(\hat{\mathbf{v}},\tau)  & = u_{o,A}^{r}(\hat{\mathbf{v}},\tau) - w^r_{o}(\hat{\mathbf{v}},\tau)  + w^r_o(\hat{\mathbf{v}},\tau) - u_{o}(\hat{\mathbf{v}},\tau)
\end{equs}
where by the definition of $\CL^{r}_{\text{\tiny{low}}}(T,\hat{\mathbf{v}},A)$ we easily see that the approximation error~\eqref{app2} is dominant compared to \eqref{app1}. Indeed, we will have also commutators of the form $\mathcal{C}^{r+1}[f,\mathcal{L}_o]$ and $ \mathcal{C}^{r+1}[f,\mathcal{L}_{\scriptsize{dom}}] $ due to Proposition~\ref{decomp_Pi}. The extra regularity needed is coming from the Taylor expansion of the lower part in Definition~\ref{def_CK}.
\end{proof}

\section{Examples}\label{sec:examples}

In this section we illustrate our general framework (see Definition \ref{genscheme}) and its local error 
analysis (see Theorem \ref{thm:genloc}) on two examples: The Gross--Pitaevskii equation (see Section~\ref{Gross_Pitaevskii}) and the Sine--Gordon equation (see Section \ref{sec:sine}).
\subsection{The
 Gross--Pitaevskii equation} \label{Gross_Pitaevskii}
As a first example let us consider the  Gross--Pitaevskii  (GP) equation
\begin{equation}\label{evGP}
i  \partial_t u(t,x) +  \Delta u(t,x) =  V(x)u(t,x)  + \vert u(t,x)\vert^2 u(t,x)\quad (t,x) \in \R \times  \Omega 
\end{equation}
on a sufficiently smooth domain 
$\Omega \subset \R^d$ in dimension $d \le 3$, and an initial condition
\begin{equation}
\label{init}
u_{|t=0}= u_{0}.
\end{equation}
We prescribe homogeneous Dirichlet boundary conditions,
$$
u(t,\cdot)|_{\partial \Omega} = 0, \quad V(\cdot)|_{\partial \Omega} = 0,
$$
where $\Omega$ is a smooth open set with compact boundary. 
We define the operator $\mathcal{L} = i\Delta$ on the Hilbert space $L^2(\Omega)$. Its domain  is given by 
\begin{equation}\label{dom:Ddi}
 D(\Delta) = D(\mathcal{L}) = (H^2 \cap H^1_0)(\Omega)
 \end{equation}
where $H^2(\Omega), H^1_0(\Omega)$ denote the classical Sobolev spaces. 

One setting of the Gross-Pitaevskii equation is to describe the dynamics of Bose-Einstein condensates in a potential trap. In many physically relevant situations the potential is  rough or disordered (\cite{GPHenPeter}, \cite{GPdisorder}, \cite{GPphysics}) which motivates the study of the low-regularity framework.

\subsubsection{First order low regularity integrator for Gross--Pitaevskii}\label{sec::Gross_1}

\begin{corollary}\label{cor_schemeGP1}
At first order our general low regularity scheme  \eqref{genscheme}  for the Gross--Pitaevskii equation  \eqref{evGP}   takes the form
\begin{equs}\label{schemeGP}
u^{n+1}  & = \Phi_{\text{\tiny GP}}^\tau(u^n) = e^{i \tau    \Delta} u^n -i \tau  \left[ (e^{i \tau \Delta}u^n) (e^{i \tau \Delta}
 \varphi_1(- i \tau \Delta) V) \right. \\ & \left.  + (e^{i \tau \Delta}(u^n)^2) (e^{i \tau \Delta} \varphi_1( -2 i\tau \Delta) \overline u^n)  \right]
\end{equs}
 where the filter function $\varphi_1$ is defined as $\varphi_1(\sigma) = \frac{e^\sigma - 1}{\sigma}$. 
The scheme  \eqref{schemeGP} is locally of order $\CO(\tau^2|\nabla|(u+V))$. 
 
In case of more regular solutions and potential 
the above low regularity scheme can be simplified to
\begin{equs}\label{schemeGPclass}
u^{n+1} = e^{i \tau    \Delta} u^n -i \tau  \left( u^n V+ (u^n)^2 \overline u^n  \right),
\end{equs}
which is locally of order $\CO(\tau^2 \Delta (u + V))$.

\end{corollary}
\begin{proof}
We choose $r=0$ in Definition \ref{scheme} in order to obtain a local error of order one. 
We recall from Example \ref{ex_Gross} that for the Gross-Pitaevskii equation we have that $\mathcal{L}_{o} = i\Delta$, $V_0 = 1$, $V_1 = V$, $ \Lab_+ = \lbrace o , \bar o\rbrace $, $u_{o} = u, \ u_{\bar{o}} = \bar{u}, \ v_{o} = v$ and $ v_{\bar{o}} = \bar{v}$. From equation \eqref{genscheme} it then follows that the first-order scheme is of the form,
\begin{equs}\label{GP1}
u^{0}_{o,A}( \hat{\mathbf{v}},\tau) = e^{ i\tau\Delta }v + \sum_{T \in \CT^{1}} \Pi_A^0 \left( \CI_{o}(T) \right)(\hat{\mathbf{v}},\tau),
\end{equs}
where $ \hat{\mathbf{v}} = (v,\bar{v},V)$ and where on has
\begin{equs}
\CT^1 & =  \lbrace  T_0, T_1  \, \rbrace, \quad T_0 = \lambda_0 =  \begin{tikzpicture}[scale=0.2,baseline=-5]
\coordinate (root) at (0,-0.5);
\node[xi] (rootnode) at (root) {}; 
\end{tikzpicture}\, , \quad T_1 = \lambda_1 = \begin{tikzpicture}[scale=0.2,baseline=-5]
\coordinate (root) at (0,-0.5);
\node[xi,blue] (rootnode) at (root) {};
\end{tikzpicture}.
\end{equs}
We recall from \eqref{NLStree1} that
 \begin{equs}
 \CI_{o}(\lambda_{0}) = \begin{tikzpicture}[scale=0.2,baseline=-5]
\coordinate (root) at (0,-1);
\coordinate (tri) at (0,1);
\draw[kernels2] (root) -- (tri);
\node[not] (rootnode) at (root) {};
\node[xi] (trinode) at (tri) {};
\end{tikzpicture}
\end{equs}
 encodes the iterated integral
\begin{equs} \label{int_NLS1}
\left( \Pi \CI_o(\lambda_0) \right) (\hat{\mathbf{v}},\tau) =   -i \int_0^\tau e^{i(\tau-s)\Delta} \left( \left( e^{is \Delta}v^2 \right) \left(  e^{ - is \Delta} \bar{v} \right) \right) ds
\end{equs}
and that 
\begin{equs} 
\CI_{o}(\lambda_{1}) & = \begin{tikzpicture}[scale=0.2,baseline=-5]
\coordinate (root) at (0,-1);
\coordinate (tri) at (0,1);
\draw[kernels2] (root) -- (tri);
\node[not] (rootnode) at (root) {};
\node[xi,blue] (trinode) at (tri) {};
\end{tikzpicture}
\end{equs}
 encodes the iterated integral
\begin{equs}\label{int_GP1}
 \left( \Pi \CI_o(\lambda_0) \right) (\hat{\mathbf{v}},\tau) = -i \int_0^\tau e^{i(\tau -s)\Delta} \left( \left( e^{is \Delta}v \right) V \right) ds.
\end{equs}

In order to compute the approximation \eqref{GP1}, we refer to equation \eqref{Pi_A} where the approximated character $\Pi_A^r$ is defined.
We have, for $i=0,1$, that
\begin{equs}
(\Pi^0_{A} \CI_{o}(\lambda_{i}))(\hat{\mathbf{v}},\tau) &= \mathcal{K}_{o,A}^0 ( \tilde{\Pi}_{o,A} \CD_{-1}(\lambda_i)(\hat{\mathbf{v}},\cdot))(\tau) \\
&= \mathcal{K}_{o,A}^0\left(\frac{\Upsilon_{\text{\scriptsize{root}}}^{\Psi_{o}}[\lambda_i]}{S_{\text{\scriptsize{root}}}(\lambda_i)}(\hat{\mathbf{v}},\cdot) V_{i} \right)(\tau)  
\end{equs}
%
where to obtain the second line we used the definition \ref{DR} of $\CD_{-1}$, namely that $\CD_{-1}(\lambda_i) = \lambda_i$. By recalling the computations made in Example \ref{ex_Gross_cont} one has,
\begin{equs}
&S_{\text{\scriptsize{root}}}(\lambda_0) =  S_{\text{\scriptsize{root}}}(\lambda_1) = 1, \\ 
&\Upsilon_{\text{\scriptsize{root}}}^{\Psi_{o}}[\lambda_0](\mathbf{v},\xi) = \left( -ie^{i\xi \Delta} v^2 \right) \left( e^{- i\xi \Delta} \bar{v} \right), \quad 
 \Upsilon_{\text{\scriptsize{root}}}^{\Psi_{o}}[\lambda_1](\mathbf{v},\xi) =  -ie^{i\xi \Delta} v,
\end{equs}
where $ \mathbf{v} = (v, \bar{v}) $ and we recall that we set $\mathcal{B}^{\mathfrak{l}}_{o} = \id$ for $ (o,\mathfrak{l}) \in \Lab_+ \times \Lab_- $ when studying the equation \eqref{evGP}. 
Hence, by collecting the above computations one has,
\begin{equs}\label{int-1}
\Pi_A^0 \left( \CI_{o}(\lambda_0) \right)(\hat{\mathbf{v}},\tau) &= \mathcal{K}_{o,A}^0\left( F^1(\hat{\mathbf{v}},\cdot) )\right)(\tau), \\
\Pi_A^0 \left( \CI_{o}(\lambda_1) \right)(\hat{\mathbf{v}},\tau) &= \mathcal{K}_{o,A}^0\left(F^2(\hat{\mathbf{v}},\cdot) \right)(\tau)
\end{equs}
where
\begin{equation}\label{F^1_GP}
F^1(\hat{\mathbf{v}},\xi):= -i(e^{i\xi\Delta} v^2)( e^{- i\xi \Delta \bar{v}} )
\end{equation}
and
\begin{equation}\label{F^2_GP}
F^2(\hat{\mathbf{v}},\xi):= -i(e^{i\xi \Delta} v)V.
\end{equation}
We are left to apply Definition \ref{def_CK} of the operator $\mathcal{K}_{o,A}^0$ on \eqref{F^1_GP} and \eqref{F^2_GP} which will yield a first-order approximation of the integrals \eqref{int_NLS1} and \eqref{int_GP1}. The approximation of these integrals, and hence the  structure of the scheme, will depend on the regularity assumptions on the initial data $v$ and the potential $V$. We first show the construction of the first order scheme \eqref{schemeGP}, which requires $H^1$-regularity on $v$ and $V$. Hence, by taking into account the boundary conditions, we fix $A = H^1_0(\Omega)^3$ and construct a first order scheme for $(v, \bar{v}, V) \in A $  (see also Remark \ref{rmq:domainA} for further details on the choice of $A$).

\underline{The case of $A = (H^1_0(\Omega))^3$}:\\
We first note that given our regularity assumptions we have that $v, V \not\in D(\CL_{o}) = H^2(\Omega) \cap H^1_0(\Omega)$.   Hence, our approximation cannot consist of merely applying Taylor-expansions \eqref{def_fullTalor} of all the operators, since it would require $H^2$ regularity on the initial data and on the potential. Indeed, given that $r, \ell = 0$ we have that $$\cup_{n + \sum_{i\in J}k_i \le 1} D(G_{n,(k_i)_{i\in J}}) \subseteq (H^2(\Omega))^3$$ which implies that we do not have that $A \subseteq D(G_{n,(k_i)_{i\in J}})$, for all $n + \sum_{i\in J}k_i \le 1$.  In order to make an approximation of order one of the integrals \eqref{int-1} while only requiring $H^1(\Omega)$ regularity on the initial data and potential we will apply the first point $ (i) $ of Definition~\ref{def_CK}.%

We start by dealing with $F^1$, namely the first order approximation of the integral~\eqref{int_NLS1}.
\\
1. Computation of $\Pi_A^0 \left( \CI_{o}(\lambda_0) \right)(\hat{\mathbf{v}},\tau)$.
Using the notations in Definition \ref{def_CK} we compute the operator interactions, as is done in example \ref{ex_5} to obtain the following,
\begin{equs}
\mathcal{L}_1 = i\Delta, \ \mathcal{L}_2 = -\mathcal{L}_1, \ J = \{1,2 \}, \ \text{and hence}, \ \mathcal{A}_1 = 0, \ \mathcal{A_2} = -2i\Delta.
\end{equs} 
This implies that 
\begin{equs}
\mathcal{L}_{\tiny{\text{dom}}} = \mathcal{A_2} = -2i\Delta, \quad I = \{ 2\},
\end{equs}
and that
\begin{equs}
 \mathcal{A}^1_{\tiny{\text{low}}} =\mathcal{A}_1 = 0, \quad
\mathcal{A}^2_{\tiny{\text{low}}} = \mathcal{A}_2 - \mathcal{L}_{\tiny{\text{dom}}} = 0.
\end{equs}
We note that following \eqref{F} we indeed recover the same initial expression  \eqref{F^1} of $F^1$:
\begin{equs}
F^1(\hat{\mathbf{v}},\xi) &= (e^{-i\xi\Delta}u^a_2)(e^{i\xi\Delta}u^a_1), \quad u^a_1 = -iv^2, \ \  u^a_2 = \bar{v}.
\end{equs}
We are now ready to compute $\mathcal{K}_{o,A}^0(F^1)$ using equation \eqref{K_reson}. Given the form of $F^1$, and the order of the scheme we enter the case where $l=0$ and $r=0$ respectively.  This implies that $q = 0 = n = m = p$. Hence, $\mathcal{K}_{o,A}^0(F^1)$ takes the following simple form,
\begin{equs}\label{approx_lambda0}
\mathcal{K}_{o,A}^0(F^1(\hat{\mathbf{v}},\cdot))(\tau) &= \int_0^\tau (e^{i\tau\Delta} \mathcal{M}_{\{1\}})(e^{-2i\xi\Delta+ i\tau\Delta }\mathcal{M}_{\{2\}})(-iv^2,\bar{v})d\xi \\
&=-i\int_0^\tau (e^{i\tau\Delta}v^2)(e^{-2i\xi\Delta+ i\tau\Delta }\bar{v})d\xi .
\end{equs}
%
By integrating exactly the above expression using the $\varphi_1$ function, and by \eqref{int-1} we have hence shown that,
\begin{equs}\label{lambda_0}
\Pi_A^0 \left( \CI_{o}(\lambda_0) \right)(\hat{\mathbf{v}},\tau) = -i\tau(e^{i\tau\Delta }v^2)(e^{i\tau\Delta}\varphi_1(-2i\tau\Delta)\bar{v}).
\end{equs} \\
2. Computation of $\Pi_A^0 \left( \CI_{o}(\lambda_1) \right)(\hat{\mathbf{v}},\tau)$.
Similarly, we apply definition \ref{def_CK} to compute $\mathcal{K}_{o,A}^0(F^2)$. Given the expression \eqref{F^2_GP} of $F^2$, one has
\begin{equs}
&\mathcal{L}_1 = i\Delta, \ \mathcal{L}_2 = 0, \ J = \{1,2 \}, u_1^b = -iv, \ \  u_2^b = V,\ \text{and hence}, \\
& \mathcal{A}_1 = 0, \ \mathcal{A_2} = -i\Delta.
\end{equs} 
This implies that 
\begin{equs}
\mathcal{L}_{\tiny{\text{dom}}} = \mathcal{A_2} = -i\Delta, \quad I = \{ 2\},
\end{equs}
and that
\begin{equs}
& \mathcal{A}^1_{\tiny{\text{low}}} =\mathcal{A}_1 = 0,  \quad\mathcal{A}^2_{\tiny{\text{low}}} = \mathcal{A}_2 - \mathcal{L}_{\tiny{\text{dom}}} = 0.
\end{equs}
%
%
%
Then, again by equation \eqref{K_reson} we have
\begin{equs}
\mathcal{K}_{o,A}^0(F^2)(\tau) &= \int_0^\tau (e^{i\tau\Delta} \mathcal{M}_{\{1\}})(e^{i(\tau - \xi)\Delta}\mathcal{M}_{\{2\}})(-iv,V)d\xi \\
&=-i\int_0^\tau (e^{i\tau\Delta}v)(e^{i(\tau - \xi)\Delta }V)d\xi \\
&=-i\tau(e^{i\tau\Delta}v)(e^{i\tau\Delta}\varphi_1(-i\tau\Delta)V).
\end{equs}
Hence, we have 
\begin{equs}
\Pi_A^0 \left( \CI_{o}(\lambda_1) \right)(\hat{\mathbf{v}},\tau) = -i\tau(e^{i\tau\Delta}v)(e^{i\tau\Delta}\varphi_1(-i\tau\Delta)V).
\end{equs}
Plugging the above computation into \eqref{GP1} yields the first order scheme \eqref{schemeGP}. 

\underline{The case of $A = (H^2\cap H^1_0)(\Omega)^3$}:\\
Given that $A = D(\CL_{o})^3$, we apply the Taylor expansion \eqref{def_fullTalor} with $r=\ell = 0$, to obtain:
\begin{equs}
\CK_{o,A}^0\left((\Upsilon_{\text{\scriptsize{root}}}^{\Psi_{o}}[\lambda_0] + \Upsilon_{\text{\scriptsize{root}}}^{\Psi_{o}}[\lambda_1])(\mathbf{v},\cdot) \right)(\tau) & = \int_0^\tau \Pi_{i\in \lbrace 1,2 \rbrace} u_i^ad\xi + \int_0^\tau \Pi_{i\in \lbrace 1,2 \rbrace} u_i^bd\xi\\
&= -i\tau(v^2v + vV),
\end{equs}
which thanks to \eqref{GP1} yields the first order scheme \eqref{schemeGPclass}.
\\
{\noindent \textbf{Local error analysis}.} Using the recursive formula in Definition~\ref{def:Llow}, one gets for $ T \in \lbrace  \lambda_0, \lambda_1 \rbrace $
\begin{equs}
\mathcal{L}^{0}_{\text{\tiny{low}}}(\CI_{o}( T ),\hat{\mathbf{v}},A)  =    \mathcal{L}^{-1,o}_{\text{\tiny{low}}}(  T, \hat{\mathbf{v}},A) +     R^0_{o,A}\left(\left(\Pi_A \mathcal{I}^0_{o}( T) \right)(\hat{\mathbf{v}},\cdot)  \right) .
\end{equs}
Then by \eqref{local_error_1}, we get for $ \mathfrak{l}\in \lbrace 0,1 \rbrace $
\begin{equs}
\mathcal{L}^{-1,o}_{\text{\tiny{low}}}(\lambda_{\mathfrak{l}},\hat{\mathbf{v}},A)   =  \frac{\Upsilon_{\text{\scriptsize{root}}}^{\Psi_o}[T]}{S_{\text{\scriptsize{root}}}(\lambda_i)}(\mathbf{v},0) V_{\mathfrak{l}}   
\end{equs}
which gives
\begin{equs}
\mathcal{L}^{-1,o}_{\text{\tiny{low}}}(\lambda_{1},\hat{\mathbf{v}},A) =-i vV, \quad  \mathcal{L}^{-1,o}_{\text{\tiny{low}}}(\lambda_{0},\hat{\mathbf{v}},A) =  -i v^2\bar{v}.
\end{equs}
It remains to compute $ R^0_{o,A}\left(\left(\Pi_A  T \right)  \right)$ whose value will depend on $ A $. We start with the case $ A =   (H^1_0(\Omega))^3$. From our previous computations for the scheme, one has to apply $ (ii) $ from Lemma~\ref{Taylor_bound} that gives using notations of the proof of Corollary~\ref{cor_schemeGP1}
\begin{equs}
R_{o,A}^r((\tilde{\Pi}_{o,A}\lambda_0)(\hat{\mathbf{v}},\cdot))  & =   \sum_{ \sum_{i \in J} k_i + m + n + p = r - \ell +1  }   \mathcal{C}^{m}[ \left(  \mathcal{C}^{n}[\mathcal{M}_{J \setminus I},  \mathcal{L}_o] \right) \\ & \left(  \mathcal{C}^{p}[ \mathcal{M}_{ I}   , \mathcal{L}_{\text{\tiny{dom}}} +  \mathcal{L}_o ] \right) , \mathcal{L}_{o}](((\CA^{i}_{\text{\tiny{low}}})^{k_i}  u^{a}_i )_{i \in J}) .
\end{equs}
Here, $ r = \ell = k_i = 0 $, so it remains $ m+ n+ p =1 $. As $ I $ and $ J \setminus I$ are singleton, there is only one non zero value given by $ m =1 $. In the end, one obtains
\begin{equs}
R_{o,A}^0((\tilde{\Pi}_{o,A}\lambda_0)(\hat{\mathbf{v}},\cdot)) = \mathcal{C}[ \mathcal{M}_{\lbrace 1,2 \rbrace},i \Delta](-iv^2,\bar{v}).
\end{equs}
With a similar computation, one gets
\begin{equs}
R_{o,A}^0((\tilde{\Pi}_{o,A}\lambda_1)(\hat{\mathbf{v}},\cdot)) = \mathcal{C}[ \mathcal{M}_{\lbrace 1,2 \rbrace},i \Delta](-i v,V).
\end{equs}
In both cases, one sees that these terms ask first order derivatives on $ v $ and $ V $. Indeed using Definition \ref{def:comm} of the commutator we have,
\begin{equs}\label{com-1}
\mathcal{C}[ \mathcal{M}_{\lbrace 1,2 \rbrace},i \Delta](w,z) = -2i\nabla w\cdot\nabla z.
\end{equs}

Next, for $ A = (H^2\cap H^1_0)(\Omega)^3 $, one has to use $ (i) $ from Lemma~\ref{Taylor_bound}, namely
\begin{equs}
R_{o,A}^r((\tilde{\Pi}_{o,A}\lambda_0)(\hat{\mathbf{v}},\cdot))  & = \sum_{\sum_i k_i + n = r-\ell + 1 }  \mathcal{L}_o^n \left( \prod_{i \in J} \mathcal{L}_i^{k_i} u_i^a \right)
\end{equs} 
Here, $ r =  \ell   = 0  $, there for one obtains three terms: 
\begin{equs}
R_{o,A}^0((\tilde{\Pi}_{o,A}\lambda_0)(\hat{\mathbf{v}},\cdot))  =  \Delta \left(  v^2\bar{v} \right) +   \Delta( v^2) \bar{v} -  v^2 \Delta \bar{v} 
\end{equs}
A similar computation shows
\begin{equs}
R_{o,A}^0((\tilde{\Pi}_{o,A}\lambda_1)(\hat{\mathbf{v}},\cdot))  =  \Delta \left(  v V\right) +   \Delta( v) V +  v \Delta V 
\end{equs}
which allows us to conclude.
\end{proof}
\begin{remark}[Error improvement]
Classical approximation techniques, such as splitting or exponential integrator methods (see, e.g., \cite{R1,R2,R3}) introduce a local error structure of type $\mathcal{O}(\Delta u(t), \Delta V)$ and hence requires the solution and potential $(u(t), V) \in D(\Delta)^2=(H^2\cap H^1_0)(\Omega)^2$ (see also \eqref{dom:Ddi}). The scheme \eqref{schemeGPclass} together with its local error coincides with a first order scheme obtained via exponential integrator methods. The local error of the low regularity GP integrator \eqref{schemeGP} on the other hand only requires the boundedness of first instead of second order spatial derivatives of the potential $V$ and solution~$u$.
\end{remark}

\subsubsection{Second order Duhamel integrator for Gross--Pitaevskii}
\label{sec::Gross_2}
We first recall from equation \eqref{2_order} in our first example, that by following the Taylor expansion \eqref{main_equation} and linearization \eqref{commutator} steps, we seek to provide a second order low-regularity approximation to the following iterated integrals:
\begin{equs}\label{2_order}
 w^{2} (\bf{v},\tau) &  =  e^{i\tau \Delta} v -i\int_{0}^{\tau} 
e^{i(\tau-\xi)\Delta}  \left((e^{i\xi \Delta} v^2)(e^{-i\xi\Delta}\bar{v}) + (e^{i\xi \Delta} v)V \right)d\xi \\ 
&  -i \int_{0}^{\tau} \xi
e^{i(\tau-\xi)\Delta}  \left( (e^{i\xi \Delta} \mathcal{C}[u^2,i\Delta](v))(e^{-i\xi\Delta}\bar{v}) \right) d\xi \\ 
& -\int_{0}^{\tau}
e^{i(\tau-\xi)\Delta} \left(\left( \int_{0}^{\xi} 
e^{i(\xi-\xi_1)\Delta}  \left((e^{i\xi_1 \Delta} v^2)(e^{-i\xi_1\Delta}\bar{v}) \right.\right.\right. \\
&\quad\left. \left. \left. + (e^{i\xi_1 \Delta} v)V\right)d\xi_1 \right) (2(e^{i\xi \Delta} v)(e^{-i\xi \Delta}\bar{v}) + V) \right) d\xi \\
& + \int_{0}^{\tau}
e^{i(\tau-\xi)\Delta} \left(\left( \int_{0}^{\xi} 
e^{-i(\xi-\xi_1)\Delta}  \left((e^{-i\xi_1 \Delta} \bar{v}^2)(e^{i\xi_1\Delta}v) \right.\right.\right. \\
&\left. \left.\left. \quad + (e^{-i\xi_1\Delta}\bar{v}) \bar{V}\right)d\xi_1 \right) (e^{i\xi \Delta} v^2) \right) d\xi.
\end{equs}

\begin{corollary}\label{cor:schemeGP2}
At second order our general low regularity scheme  \eqref{genscheme}  for the Gross-Pitaevskii equation \eqref{evGP} takes the form,
%
\begin{equs}\label{schemeGP2}
u^{n+1} 
=& \ e^{i \tau    \Delta} u^n -i \tau  \left( 
(e^{i \tau \Delta}(u^n)^2) (e^{i \tau \Delta} \varphi_1( -2 i\tau \Delta) \overline u^n) \right.\\
&\left. \quad +(e^{i \tau \Delta}u^n) (e^{i \tau \Delta}\varphi_1(- i \tau \Delta) V) \right) \\
&-i \tau^2 \CC[(e^{i\tau\Delta}\CM_{\{1\}})(e^{i\tau\Delta}(\varphi_1(-2i\tau\Delta) - \varphi_2(-2i\tau\Delta))\CM_{\{2\}}),i\Delta]((u^n)^2,\bar{u}^n) \\
&-i \tau^2 \CC[(e^{i\tau\Delta}\CM_{\{1\}})(e^{i\tau\Delta}(\varphi_1(-i\tau\Delta) - \varphi_2(-i\tau\Delta))\CM_{\{2\}}),i\Delta](u^n,V) \\
&-i\tau^2\left (e^{i\tau\Delta} \CC[u^2,i\Delta](u^n)\right)\left(e^{i\tau\Delta}\varphi_2(-2i\tau\Delta)\bar{u}^n\right) \\
&- \frac {\tau^2}{2}( u^n |u^n|^4 + 3 u^n |u^n|^2V - |u^n|^2u^n\bar{V} + u^n V^2)\\
=& \ \Phi_{\text{\tiny GP2}}^\tau(u^n),
\end{equs}
where the filter function $\varphi_2$ is defined as $\varphi_2(\sigma) = \frac{e^\sigma - \varphi_1(\sigma)}{\sigma}$.
The scheme \eqref{schemeGP2} is locally of order $\CO(\tau^3\Delta(u+V))$.
%
\end{corollary}

\begin{remark}
In case of more regular solution and potential the second order scheme~\eqref{schemeGP2} can be simplified, recovering for sufficiently smooth solutions and potential $(u(t),V) \in D(\Delta^2)^2$ classical schemes (see also Remark \ref{rem:ErrImp2} below).

\end{remark}
\begin{proof}
We choose $r=1$ in Definition \ref{scheme} in order to obtain a local error of order two. From equation \eqref{genscheme} it then follows that the second-order scheme is of the form,

\begin{equs}\label{GP2}
u^{1}( \hat{\mathbf{v}},\tau) = e^{ i\tau\Delta }v + \sum_{T \in \CT^{2}} \Pi_A^1 \left( \CI_{o}(T) \right)(\hat{\mathbf{v}},\tau),
\end{equs}
where $ \hat{\mathbf{v}} = (\bf{v},\bf{V})$ and where are interested in the decorated trees $ \mathcal{I}_{o}(T) $ where $ T $ belongs to $ \CT^2 $ defined by,
\begin{equs}
\CT^2 & =  \lbrace  T_0, ..., T_8  \, \rbrace, \quad T_0 = \lambda_0 =  \begin{tikzpicture}[scale=0.2,baseline=-5]
\coordinate (root) at (0,-0.5);
\node[xi] (rootnode) at (root) {};
\end{tikzpicture}\, , \quad T_1 = \lambda_1 = \begin{tikzpicture}[scale=0.2,baseline=-5]
\coordinate (root) at (0,-0.5);
\node[xi,blue] (rootnode) at (root) {};
\end{tikzpicture}\, , \quad T_2 = \lambda_0^1 =  \begin{tikzpicture}[scale=0.2,baseline=-5]
\coordinate (root) at (0,-0.5);
\node[xix] (rootnode) at (root) {};
\end{tikzpicture}\, , \quad
\\T_3 & = \lambda_0 \mathcal{I}_{o}(\lambda_0) = \begin{tikzpicture}[scale=0.2,baseline=-5]
\coordinate (root) at (0,-1);
\coordinate (tri) at (0,1);
\draw[kernels2] (root) -- (tri);
\node[xi] (rootnode) at (root) {};
\node[xi] (trinode) at (tri) {};
\end{tikzpicture}\, , \quad T_4 = \lambda_0 \mathcal{I}_{o}(\lambda_1) = \begin{tikzpicture}[scale=0.2,baseline=-5]
\coordinate (root) at (0,-1);
\coordinate (tri) at (0,1);
\draw[kernels2] (root) -- (tri);
\node[xi] (rootnode) at (root) {};
\node[xi,blue] (trinode) at (tri) {};
\end{tikzpicture}\, , \quad 
T_5 = \lambda_1 \mathcal{I}_{o}(\lambda_0) = \begin{tikzpicture}[scale=0.2,baseline=-5]
\coordinate (root) at (0,-1);
\coordinate (tri) at (0,1);
\draw[kernels2] (root) -- (tri);
\node[xi,blue] (rootnode) at (root) {};
\node[xi] (trinode) at (tri) {};
\end{tikzpicture} , \\
T_6 &= \lambda_1 \mathcal{I}_{o}(\lambda_1) = \begin{tikzpicture}[scale=0.2,baseline=-5]
\coordinate (root) at (0,-1);
\coordinate (tri) at (0,1);
\draw[kernels2] (root) -- (tri);
\node[xi,blue] (rootnode) at (root) {};
\node[xi,blue] (trinode) at (tri) {};
\end{tikzpicture}\, , \quad
T_7 = \lambda_0 \mathcal{I}_{\overline{o}}(\lambda_0) = \begin{tikzpicture}[scale=0.2,baseline=-5]
\coordinate (root) at (0,-1);
\coordinate (tri) at (0,1);
\draw[kernels2,tinydots] (root) -- (tri);
\node[xi] (rootnode) at (root) {};
\node[xi] (trinode) at (tri) {};
\end{tikzpicture}\, , \quad
T_8  = \lambda_0 \mathcal{I}_{\overline{o}}(\lambda_1) = \begin{tikzpicture}[scale=0.2,baseline=-5]
\coordinate (root) at (0,-1);
\coordinate (tri) at (0,1);
\draw[kernels2,tinydots] (root) -- (tri);
\node[xi] (rootnode) at (root) {};
\node[xi,blue] (trinode) at (tri) {};
\end{tikzpicture}.
\end{equs}
As stated previously, the approximation of the above integrals, and hence the  structure of the scheme, will depend on the regularity assumptions of the initial data $v$ and the potential $V$. Here, we show the construction of a second order scheme which requires $H^2$-regularity on $v$ and $V$. Hence, by taking into account the boundary conditions, we fix $A = (H^2(\Omega)\cap H^1_0(\Omega))^3$ and construct a second order scheme for $(v, \bar{v}, V) \in A $.
\\
1. Computation of $\Pi_A^1 \left( \CI_{o}(\lambda_0) \right)(\hat{\mathbf{v}},\tau)$.\\
We recall that $\CI_{o}(\lambda_0)$ encodes the first integral \eqref{int_NLS1}.
Next, we make the following remark; given our regularity assumptions $(v,\bar v, V) \in A$ , we have that $v, V \not\in D(\CL_{o}^2) \subset H^4(\Omega)$. Hence, our approximation cannot consist of merely applying Taylor-expansions \eqref{def_fullTalor} of all the operators, since it would require $H^4$ regularity on the initial data and on the potential, together with the according boundary conditions. Indeed, given that $r=1$ and $\ell = 0$ we have that $$\cup_{n + \sum_{i\in J}k_i \le 2} D(G_{n,(k_i)_{i\in J}}) \subseteq (H^4(\Omega))^3$$ which implies that we do not have that $A \subseteq D(G_{n,(k_i)_{i\in J}})$, for all $n + \sum_{i\in J}k_i \le 2$. In order to make an approximation of order two of the integrals \eqref{int-1} while only requiring $H^2(\Omega)$ regularity on the inital data we will apply the first point $ (i) $ of Definition~\ref{def_CK}, as done in Example \ref{ex_5}.
Indeed, following the results established in Example \ref{ex_5} we have:
\begin{equs}
\left( \Pi_A \CI^1_{o}(\lambda_{0}) \right) (\hat{\mathbf{v}},\xi) &= -i\int_{0}^{\tau} \left( e^{i\tau \Delta} v^2 \right) \left( e^{ i(\tau- 2 \xi) \Delta }  \overline{v} \right) \,  d \xi \\ & +  \int_{0}^{\tau} (\tau-\xi) \mathcal{C}[ \left( e^{i\tau \Delta} \mathcal{M}_{\lbrace 1 \rbrace} \right) \left( e^{i(\tau-2\xi) \Delta}  \mathcal{M}_{ \lbrace 2 \rbrace} \right), i\Delta](-iv^2,\overline{v}) \,  d \xi \\
& = -i[\tau(e^{i\tau\Delta }v^2)(e^{i\tau\Delta}\varphi_1(-2i\tau\Delta)\bar{v}) \\
&+ \tau^2 \CC[(e^{i\tau\Delta}\CM_{\{1\}})(e^{i\tau\Delta}(\varphi_1(-2i\tau\Delta) \\
&\quad- \varphi_2(-2i\tau\Delta))\CM_{\{2\}}),i\Delta](v^2,\bar{v})]
\end{equs}

\noindent 2. Computation of $\Pi_A^1 \left( \CI_{o}(\lambda_1) \right)(\hat{\mathbf{v}},\tau)$.\\
We recall that $\CI_{o}(\lambda_1)$ encodes the second integral \eqref{int_GP1}. Using the same arguments as in the above together with the computations of $\Pi_A^0 \left( \CI_{o}(\lambda_1) \right)(\hat{\mathbf{v}},\tau)$ made in the proof of Corollary \ref{cor_schemeGP1}, it follows that,
\begin{equs}
\left( \Pi_A \CI^1_{o}(\lambda_1) \right) (\hat{\mathbf{v}},\xi) 
 &= -i\left(\tau(e^{i\tau\Delta }v)(e^{i\tau\Delta}\varphi_1(-i\tau\Delta)V) \right. \\
&\left. + \tau^2 \CC[(e^{i\tau\Delta}\CM_{\{1\}})(e^{i\tau\Delta}(\varphi_1(-i\tau\Delta) \right. \\ & \left. - \varphi_2(-i\tau\Delta))\CM_{\{2\}}),i\Delta](v,V)\right).
\end{equs} \\
3. Computation of $\Pi_A^1 \left( \CI_{o}(\lambda_0^1) \right)(\hat{\mathbf{v}},\tau)$.\\
We have that 
\begin{equs} 
\CI_{o}(\lambda_{0}^1)  & = \begin{tikzpicture}[scale=0.2,baseline=-5]
\coordinate (root) at (0,-1);
\coordinate (tri) at (0,1);
\draw[kernels2] (root) -- (tri);
\node[not] (rootnode) at (root) {};
\node[xix] (trinode) at (tri) {};
\end{tikzpicture}
\end{equs}
 encodes the iterated integral
\begin{equs}
-i \int_{0}^{\tau} \xi
e^{i(\tau-\xi)\Delta}  \left( (e^{i\xi \Delta} \mathcal{C}[u^2,i\Delta](v))(e^{-i\xi\Delta}\bar{v}) \right) d\xi
\end{equs}
Using the computation made in example \eqref{ex_Gross_cont} we have,
\begin{equs}
(\Pi^1_{A} \CI_{o}(\lambda_0^1))(\hat{\mathbf{v}},\tau) &= \mathcal{K}_{o,A}^1 ( (\tilde{\Pi}_{o,A} \CD_{0}(\lambda_0^1))(\hat{\mathbf{v}},\cdot))(\tau) \\
&= \mathcal{K}_{o,A}^1\left(\frac{\Upsilon_{\text{\scriptsize{root}}}^{\Psi_{o}}[\lambda_0^1]}{S_{\text{\scriptsize{root}}}(\lambda_0^1)}(\mathbf{v},\xi) V_{0}\xi \right)(\tau)\\
&= \mathcal{K}_{o,A}^1\left(\xi \left( e^{ \xi \mathcal{L}}\mathcal{C}[ u^2, \Delta](v) \right)\left( e^{- \xi \mathcal{L}} \bar{v}\right)\right)(\tau).
\end{equs}
To apply Definition \ref{def_CK} we let $u_1 = \CC[u^2,\Delta](v)$, and $u_2 = \bar{v}$. The operator interactions are the same as those computed previously.
We are in the case where $r=1=\ell$, hence $q=0=m=n=p.$ From Definition \ref{def_CK} part i., it then follows that,
\begin{equs}
\left( \Pi^0_{A} \CI_{o}(\lambda_0^1)\right)(\hat{\mathbf{v}},\tau) & = \int_0^\tau \xi [(e^{i\tau\Delta}\CM_{\{1\}})(e^{-2i\xi\Delta +i\tau\Delta}\CM_{\{2\}})](u_1,u_2) d \xi\\
& = -i\tau^2\left (e^{i\tau\Delta} \CC[u^2,i\Delta](v)\right)\left(e^{i\tau\Delta}\varphi_2(-2i\tau\Delta)\bar{v}\right)
\end{equs}

\noindent 4. Computation of $\Pi_A^1 \left( \CI_{o}(T_3) \right)(\hat{\mathbf{v}},\tau)$.\\
We have that 
\begin{equs} 
\CI_{o}(T_3) & = \begin{tikzpicture}[scale=0.2,baseline=-5]
\coordinate (root) at (0,-1);
\coordinate (tri) at (0,1);
\coordinate (tri1) at (0,3);
\draw[kernels2] (root) -- (tri);
\draw[kernels2] (tri) -- (tri1);
\node[not] (rootnode) at (root) {};
\node[xi] (trinode1) at (tri) {};
\node[xi] (trinode) at (tri1) {};
\end{tikzpicture}
\end{equs}
encodes the iterated integral
\begin{equs}
-\int_{0}^{\tau}
e^{i(\tau-\xi)\Delta} \left(\left( \int_{0}^{\xi} 
e^{i(\xi-\xi_1)\Delta}  \left((e^{i\xi_1 \Delta} v^2)(e^{-i\xi_1\Delta}\bar{v})\right)d\xi_1 \right) 2(e^{i\xi \Delta} v)(e^{-i\xi \Delta}\bar{v}) \right) d\xi, 
\end{equs}
By definition of the projection operator $\CD_r$ and \eqref{def_CK} we have,
\begin{equs}\label{T_3-1}
\Pi_A^1 \left( \CI_{o}(T_3) \right)(\hat{\mathbf{v}},\tau) 
&= \mathcal{K}_{o,A}^1 ( \tilde{\Pi}_{o,A} \CD_{0}(\lambda_0 \CI_{o}(\lambda_0))(\hat{\mathbf{v}},\cdot))(\tau) \\
&= \mathcal{K}_{o,A}^1 ( \tilde{\Pi}_{o,A}(\lambda_0 \CI_{o}^0(\lambda_0)))(\hat{\bf{v}}, \tau)\\
&= \mathcal{K}_{o,A}^1\left(\frac{\Upsilon_{\text{\scriptsize{root}}}^{\Psi_{o}}[T_3]}{S_{\text{\scriptsize{root}}}(T_3)}(\mathbf{v},\xi) V_{0} (\Pi_{A}\CI_{o}^0(\lambda_0))(\hat{\bf{v}}, \xi) \right)(\tau)
\end{equs}
where $V_0 = 1$.
%
Hence, in order to compute the above approximation we are left to compute $S_{\text{\scriptsize{root}}}(T_3)$, $\Upsilon_{\text{\scriptsize{root}}}^{\Psi_{o}}[T_3]$, and $(\Pi_{A}\CI_{o}^0(\lambda_0))(\hat{\bf{v}}, \xi)$.  First, by definition \eqref{S_root} of $S_{\text{\scriptsize{root}}}(T_3)$ we have that $i=1=j$ and $\beta_{1,1}  = 1 $, where $o_1 = o$.
Hence, we have that $S_{\text{\scriptsize{root}}}(T_3) = 1$. Second, by definition \eqref{Upsilon_root} of $\Upsilon_{\text{\scriptsize{root}}}^{\Psi_{o}}[T]$ we have that,
\begin{equs}
\Upsilon_{\text{\scriptsize{root}}}^{\Psi_{o}}[T_3](\mathbf{v},\xi) & = (D_{o}\hat{\Psi}_{o}^0)(\mathbf{v}, \xi) \\ & = D_{o} \left( (e^{i\xi\Delta}(-iv_o^2))(e^{-i\xi\Delta} v_{\bar{o}}) \right) \\& = -2i(e^{i\xi\Delta}v)(e^{-i\xi}\bar{v})
\end{equs}
since $v_o = v$, and $v_{\bar{o}}=\bar{v}$. Finally, following the computations made at the beginning of example \ref{ex_5} we have,
\begin{equs}
\left( \Pi_A \CI^0_{o}(\lambda_{0}) \right) (\hat{\mathbf{v}},\xi)  & = \mathcal{K}_{o,A}^0 ( F(\hat{\mathbf{v}},\cdot))(\xi)
\end{equs}
where
\begin{equs} 
F(\mathbf{v},\xi) =  \prod_{i \in \lbrace 1,2 \rbrace} e^{\xi \mathcal{L}_i} u_i,\quad &u_1 = -iv^2, \quad u_2 = \bar{v}, \quad \mathcal{L}_1 = i \Delta, \quad \mathcal{L}_2 = -i \Delta 
\end{equs}
In our current setting we have that $r=0=\ell$ and hence $D(\CL_{o}^{r-\ell +1}) = D(\Delta) $. Given the regularity assumptions on $v$, namely that $v \in H^2\cap H^1_0$, it follows that $v \in D(\Delta)$. Following \eqref{def_fullTalor} we can then Taylor expand all the operators appearing in the integral which yields the following approximation,
\begin{equs}
\left( \Pi_A \CI^0_{o}(\lambda_{0}) \right) (\hat{\mathbf{v}},\xi)  
 = \mathcal{K}_{o,A}^0 (F(\hat{\mathbf{v}},\cdot))(\xi) = \int_0^\xi \prod_{i \in \lbrace 1,2 \rbrace} u_i d\xi  = -i\xi v^2\bar{v}.
\end{equs}
Hence, plugging the above computations in \eqref{T_3-1} yields,

\begin{equs}\label{T_3-2}
\Pi_A^1 \left( \CI_{o}(T_3) \right)(\mathbf{v},\tau) 
&= -2\mathcal{K}_{o,A}^1\left( \xi(e^{i\xi\Delta}v)(e^{-i\xi\Delta}\bar{v})v^2{\bar v} \right)(\tau)\qquad\qquad\qquad \\
&= -2\mathcal{K}_{o,A}^1\left( {\tilde F}(\hat{\mathbf{v}},\xi)  \right)(\tau)
\end{equs}

where 
\begin{equs}\label{F^1}
{\tilde F}(\hat{\mathbf{v}},\xi) = \xi^{\ell} \prod_{i \in \lbrace 1,2,3 \rbrace} e^{\xi\CL_i}u_i, \quad\quad
& \CL_1 = i\Delta, \ \CL_2 = -i\Delta, \ \CL_3=  0, \quad \\
& u_1 = v  \ u_2 = \bar{v} , \ u_3 = v^2 \bar{v} \ \text{and} \ \ell = 1.
\end{equs}
In our current setting we have that $r=1 = \ell$ and hence $D(\CL_{o}^{r-\ell +1}) = D(\Delta) $. Given that $v \in D(\Delta)$, by following \eqref{def_fullTalor} we can once again Taylor expand the remaining operators inside the integral which yields the following approximation,

\begin{equs}
\mathcal{K}_{o,A}^1\left( {\tilde F}(\hat{\mathbf{v}},\xi)  \right)(\tau) &= \int_0^\tau \xi \prod_{i \in \lbrace 1,2,3 \rbrace} u_i d\xi  = \frac {\tau^2}{2} v |v|^4,
\end{equs}
by definition of the $(u_i)_{i=1,\dots,4}$. Hence, from \eqref{T_3-2} it follows that,
\begin{equs}
\Pi_A^1 \left( \CI_{o}(T_3) \right)(\hat{\mathbf{v}},\tau) 
&=-\tau^2 v|v|^4.
\end{equs}

\bigskip

\noindent 5. Computation of $\ \Pi_A^1 \left( \CI_{o}(T_4) \right)(\hat{\mathbf{v}},\tau), \ \Pi_A^1 \left( \CI_{o}(T_5) \right)(\hat{\mathbf{v}},\tau),$ and $\ \Pi_A^1 \left( \CI_{o}(T_6) \right)(\hat{\mathbf{v}},\tau)$.\\
We have that 
\begin{equs} 
\CI_{o}(T_4) = \begin{tikzpicture}[scale=0.2,baseline=-5]
\coordinate (root) at (0,-1);
\coordinate (tri) at (0,1);
\coordinate (tri1) at (0,3);
\draw[kernels2] (root) -- (tri);
\draw[kernels2] (tri) -- (tri1);
\node[not] (rootnode) at (root) {};
\node[xi] (trinode1) at (tri) {};
\node[xi,blue] (trinode) at (tri1) {};
\end{tikzpicture},\quad
\CI_{o}(T_5) = \begin{tikzpicture}[scale=0.2,baseline=-5]
\coordinate (root) at (0,-1);
\coordinate (tri) at (0,1);
\coordinate (tri1) at (0,3);
\draw[kernels2] (root) -- (tri);
\draw[kernels2] (tri) -- (tri1);
\node[not] (rootnode) at (root) {};
\node[xi,blue] (trinode1) at (tri) {};
\node[xi] (trinode) at (tri1) {};
\end{tikzpicture},\quad
\CI_{o}(T_5) = \begin{tikzpicture}[scale=0.2,baseline=-5]
\coordinate (root) at (0,-1);
\coordinate (tri) at (0,1);
\coordinate (tri1) at (0,3);
\draw[kernels2] (root) -- (tri);
\draw[kernels2] (tri) -- (tri1);
\node[not] (rootnode) at (root) {};
\node[xi,blue] (trinode1) at (tri) {};
\node[xi,blue] (trinode) at (tri1) {};
\end{tikzpicture}.
\end{equs}
respectively encode 
\begin{equs}
&-\int_{0}^{\tau}
e^{i(\tau-\xi)\Delta} \left(\left( \int_{0}^{\xi} 
e^{i(\xi-\xi_1)\Delta}  \left((e^{i\xi_1 \Delta} v)V\right)d\xi_1 \right) 2(e^{i\xi \Delta} v)(e^{-i\xi \Delta}\bar{v}) \right) d\xi,\\ 
&-\int_{0}^{\tau}
e^{i(\tau-\xi)\Delta} \left(\left( \int_{0}^{\xi} 
e^{i(\xi-\xi_1)\Delta}  \left((e^{i\xi_1 \Delta} v^2)(e^{-i\xi_1\Delta}\bar{v})\right)d\xi_1 \right)V \right) d\xi, \\
&-\int_{0}^{\tau}
e^{i(\tau-\xi)\Delta} \left(\left( \int_{0}^{\xi} 
e^{i(\xi-\xi_1)\Delta}  \left((e^{i\xi_1 \Delta} v)V\right)d\xi_1 \right) V \right) d\xi, 
\end{equs}
namely the remaining integrals in the third line of \eqref{2_order}. These computation follow exactly the same lines as for the computation of $\Pi_A^1 \left( \CI_{o}(T_3)\right)$. Namely, given our regularity assumptions we can apply the Taylor based expansion \eqref{def_fullTalor}, to obtain,
\begin{equs}
\left(\Pi_A^1 \left( \CI_{o}(T_4) \right) + \Pi_A^1 \left( \CI_{o}(T_5) \right)\right)(\hat{\mathbf{v}},\tau) 
&= -2\int_0^\tau \xi \prod_{i \in \lbrace 1,2,3,4 \rbrace} u_i^a d\xi -\int_0^\tau \xi \prod_{i \in \lbrace 1,2,3 \rbrace} u_i^b\\
&= -\tau^2 |v|^2vV + -\frac{\tau^2}{2} Vv^2\bar{v},\\
&=-\frac{3\tau^2}{2}|v|^2vV,
\end{equs}
and
\begin{equs}
\Pi_A^1 \left( \CI_{o}(T_6) \right)(\hat{\mathbf{v}},\tau) 
&= -\int_0^\tau \xi \prod_{i \in \lbrace 1,2,3 \rbrace} u_i^c =-\frac{\tau^2}{2}V^2v,
\end{equs}
where $(u_1^a,u_2^a, u_3^a,u_4^a) = (v, V, v, \bar{v})$, $(u_1^b, u_2^b, u_3^b,)= (V, v^2, \bar{v}),$ and $(u_1^c, u_2^c, u_3^c = (V, v, V)$.

\noindent 6. Computation of $\Pi_A^1 \left( \CI_{o}(T_7) \right)(\hat{\mathbf{v}},\tau)$.\\
We have that, 
\begin{equs} 
\CI_{o}(T_7) = \begin{tikzpicture}[scale=0.2,baseline=-5]
\coordinate (root) at (0,-1);
\coordinate (tri) at (0,1);
\coordinate (tri1) at (0,3);
\draw[kernels2] (root) -- (tri);
\draw[kernels2,tinydots] (tri) -- (tri1);
\node[not] (rootnode) at (root) {};
\node[xi] (trinode1) at (tri) {};
\node[xi] (trinode) at (tri1) {};
\end{tikzpicture}
\end{equs}
encodes the iterated integral
\begin{equs}
\int_{0}^{\tau}
e^{i(\tau-\xi)\Delta} \left(\left( \int_{0}^{\xi} 
e^{-i(\xi-\xi_1)\Delta}  \left((e^{-i\xi_1 \Delta} \bar{v}^2)(e^{i\xi_1\Delta}v)\right)d\xi_1 \right) (e^{i\xi \Delta} v^2) \right) d\xi.
\end{equs}
Similar to the computations done in \eqref{T_3-1} it follows that,
\begin{equs}\label{T7-1}
\left(  \Pi_A^1 \CI_{o}(T_7) \right)(\hat{\mathbf{v}},\tau) 
&= \mathcal{K}_{o,A}^1 ( \tilde{\Pi}_{o,A} \CD_{0}(\lambda_0 \CI_{\bar{o}}(\lambda_0))(\hat{\mathbf{v}},\cdot))(\tau) \\
&= \mathcal{K}_{o,A}^1 ( \tilde{\Pi}_{o,A}(\lambda_0 \CI_{\bar{o}}^0(\lambda_0))(\hat{\mathbf{v}},\cdot))( \tau)\\
&= \mathcal{K}_{o,A}^1\left(\frac{\Upsilon_{\text{\scriptsize{root}}}^{\Psi_{o}}[T_7]}{S_{\text{\scriptsize{root}}}(T_7)}(\mathbf{v},\cdot)(\Pi_{A}\CI_{\bar{o}}^0(\lambda_0))(\hat{\bf{v}}, \cdot) \right)(\tau).
\end{equs}
We are left to calculate $\Upsilon_{\text{\scriptsize{root}}}^{\Psi_{o}}[T_7]$, $S_{\text{\scriptsize{root}}}(T_7)$, and $(\Pi_{A}\CI_{\bar{o}}^0(\lambda_0)))(\hat{\bf{v}}, \xi)$.
 First, by definition \eqref{S_root} of $S_{\text{\scriptsize{root}}}(T_7)$ we have that $i=1=j$ and $\beta_{1,1}  = 1 $, where $o_1 = \bar{o}$.
Hence, we have that $S_{\text{\scriptsize{root}}}(T_7) = 1$. Secondly, by definition \eqref{Upsilon_root} of $\Upsilon_{\text{\scriptsize{root}}}^{\Psi_{o}}[T]$ we have that,
\begin{equs}
\Upsilon_{\text{\scriptsize{root}}}^{\Psi_{o}}[T_7] & = (D_{\bar{o}}\hat{\Psi}_{o}^0)(\bf{v}, \xi) \\
&= -iD_{\bar{o}} \left( (e^{i\xi\Delta}(v_{o}^2))(e^{-i\xi\Delta} {v}_{\bar{o}}) \right)\\
& = -i(e^{i\xi\Delta}v^2)
\end{equs}
since $v_{o} = v$, and $v_{\bar{o}}=\bar{v}$. Thirdly, by following definition \eqref{def_CK} we have,
\begin{equs}
(\Pi_{A}\CI_{\bar{o}}^0 \lambda_0)(\hat{\bf{v}}, \xi) &= 
 \CK_{\bar{o},A}^0((\tilde{\Pi}_{\bar{o},A}\lambda_0)(\hat{\bf{v}},\cdot))(\xi)= \CK_{\bar{o},A}^0\left( \frac{\Upsilon_{\text{\scriptsize{root}}}^{\Psi_{\bar{o}}}[\lambda_0]}{S_{\text{\scriptsize{root}}}(\lambda_0)}(\hat{\mathbf{v}},\cdot) \right)(\xi)
\end{equs}
where by equation \eqref{S_root} we have that $S_{\text{\scriptsize{root}}}(\lambda_0) = 1$ and by equation \eqref{Upsilon_root} it follows that,
\begin{equs}
\Upsilon_{\text{\scriptsize{root}}}^{\Psi_{\bar{o}}}[\lambda_0](\bf{v},\zeta) &= \prod_{\mathcal{o} \in \Lab_{+}^{0,\bar{o}}} e^{\zeta \CL_{\mathcal{o}}}f^{0}_{\bar{o},\mathcal{o}}(v_{\mathcal{o}})=(e^{i\zeta\Delta}f^0_{\bar{o},o}(v))(e^{-i\zeta\Delta}f^0_{\bar{o},\bar{o}}(\bar{v}))\\
&=(e^{i\zeta\Delta} v)(e^{-i\zeta\Delta}(i\bar{v}^2)).
\end{equs}
Collecting the above computations, we have,
\begin{equs}
\left(\Pi_{A}\CI_{\bar{o}}^0(\lambda_0)\right)(\hat{\bf{v}}, \xi)  & = \mathcal{K}_{\bar o,A}^0 ( F (\hat{\mathbf{v}}, \cdot))(\xi)
\end{equs}
where
\begin{equs} 
F(\hat{\mathbf{v}},\xi) =  \prod_{i \in \lbrace 1,2 \rbrace} e^{\xi \mathcal{L}_i} u_i,\quad &u_1 = v, \quad u_2 = i\bar{v}^2, \quad \CL_{1} = i\Delta,\ \CL_{2} = -i\Delta,
\end{equs}
and where we consider the operator $\CL_{\bar o} = -i\Delta$ associated to the decoration $\bar o \in \Lab_{+}$.
Given that $v \in D(\Delta)$, by following \eqref{def_fullTalor} we can once again Taylor expand the operators inside the integral which yields the following approximation,
\begin{equs}
\left(\Pi_{A}\CI_{\bar{o}}^0(\lambda_0)\right)(\hat{\bf{v}}, \xi)  
& = \mathcal{K}_{\bar o,A}^0 ( F ( \hat{\mathbf{v}}, \cdot))(\xi)  = \int_0^\xi \prod_{i \in \lbrace 1,2 \rbrace} u_i d\xi = i\xi v\bar{v}^2.
\end{equs}
Hence, by collecting the above computations it follows from \eqref{T7-1} that,
\begin{equs}
\left( \Pi_A^1 \CI_{o}(T_7) \right)(\hat{\mathbf{v}},\tau) 
&=\mathcal{K}_{o,A}^1\left(\xi(e^{i\xi\Delta}v^2) v\bar{v}^2\right)(\tau) = \mathcal{K}_{o,A}^1\left( {\tilde F}(\hat{\mathbf{v}},\cdot)  \right)(\tau),
\end{equs}
where 
\begin{equs}
{\tilde F}(\hat{\mathbf{v}},\xi) = \xi^{\ell} \prod_{i \in \lbrace 1,2,3 \rbrace} e^{\xi\CL_i}u_i, \quad\quad
& \CL_1 = i\Delta, \ \CL_2= \CL_3 = 0, \quad \\
& u_1 = v^2, \ u_2 = v, \ u_3 = \bar{v}^2 \ \text{and} \ \ell = 1.
\end{equs}
Given that $v \in D(\Delta)$, by following \eqref{def_fullTalor} we can once again Taylor expand the remaining operators inside the integral which yields the following approximation,
\begin{equs}
\left( \Pi_A^1  \CI_{o}(T_7) \right)(\hat{\mathbf{v}},\tau) &=
\mathcal{K}_{o,A}^1\left( {\tilde F}(\hat{\mathbf{v}},\xi)  \right)(\tau) = \int_0^\tau \xi \prod_{i \in \lbrace 1,2,3 \rbrace} u_i d\xi  = \frac {\tau^2}{2} v |v|^4,
\end{equs}

\noindent 7. Computation of $\Pi_A^1 \left( \CI_{o}(T_8) \right)(\hat{\mathbf{v}},\tau)$.\\
We have that, 
\begin{equs} 
\CI_{o}(T_8) = \begin{tikzpicture}[scale=0.2,baseline=-5]
\coordinate (root) at (0,-1);
\coordinate (tri) at (0,1);
\coordinate (tri1) at (0,3);
\draw[kernels2] (root) -- (tri);
\draw[kernels2,tinydots] (tri) -- (tri1);
\node[not] (rootnode) at (root) {};
\node[xi] (trinode1) at (tri) {};
\node[xi,blue] (trinode) at (tri1) {};
\end{tikzpicture}
\end{equs}
encodes the iterated integral
\begin{equs}
\int_{0}^{\tau}
e^{i(\tau-\xi)\Delta} \left(\left( \int_{0}^{\xi} 
e^{-i(\xi-\xi_1)\Delta}  \left((e^{-i\xi_1 \Delta} {\bar v}) \bar{V}\right)d\xi_1 \right) (e^{i\xi \Delta} v^2) \right) d\xi.
\end{equs}
By using the same arguments as made in the above computation of the iterated integral $\CI_{o}(T_7)$, namely by following the Taylor expansion \eqref{def_fullTalor}, we obtain  
\begin{equs}
\Pi_A^1 \left( \CI_{o}(T_8) \right)(\hat{\mathbf{v}},\tau) 
&= \int_0^\tau \xi \prod_{i \in \lbrace 1,2,3 \rbrace} u_i d\xi = \frac{\tau^2}{2} v|v|^2\bar{V},
\end{equs}
where $u_1 = \bar{v}$, $u_2 =\bar{V}$, and $u_3 = v^2$. 
%
Collecting the above computation together with equation \eqref{GP2} yields the second order low-regularity scheme given in \eqref{schemeGP2}.
\\
\\
\noindent \textbf{Local error analysis}. The error for the trees $ \CI^{1}_{o}(\lambda_0) $ and $ \CI^{1}_{o}(\lambda_1) $ follows the same computations as for the first order. One has to go one step further in the Taylor approximation:
\begin{equs}
R_{o,A}^1((\tilde{\Pi}_{o,A}\lambda_0)(\hat{\mathbf{v}},\cdot)) & = \mathcal{C}^{2}[ \mathcal{M}_{\lbrace 1,2 \rbrace},i \Delta](-iv^2,\bar{v})
\\ R_{o,A}^1((\tilde{\Pi}_{o,A}\lambda_1)(\hat{\mathbf{v}},\cdot)) & = \mathcal{C}^2[ \mathcal{M}_{\lbrace 1,2 \rbrace},i \Delta](-i v,V).
\end{equs}
In the end, we obtain the following contribution to the local error. This produce the following second order commutators: 
\begin{equs}
\mathcal{L}^{1}_{\text{\tiny{low}}}(\CI_{o}( \lambda_0 ),\hat{\mathbf{v}}, A) & = - i v^2 \bar{v} + \mathcal{C}^{2}[ \mathcal{M}_{\lbrace 1,2 \rbrace},i \Delta](-iv^2,\bar{v}) \\ 
\mathcal{L}^{1}_{\text{\tiny{low}}}(\CI_{o}( \lambda_1 ),\hat{\mathbf{v}}, A) &  =- iv V + \mathcal{C}^2[ \mathcal{M}_{\lbrace 1,2 \rbrace},i \Delta](-i v,V),
\end{equs}
where using the Leibniz rule it follows from Definition \ref{def:comm} that these second order commutators ask for two spacial derivatives on $ v $ and $ V $. 

Next, the decorated tree $\CI_{o}(\lambda_{0}^1)$ contains a polynomial decoration meaning that one has $ \xi $ inside the integral. Therefore, applying Definition~\ref{def:Llow}, one gets
\begin{equs}
\mathcal{L}^{1}_{\text{\tiny{low}}}(\CI_{o}( \lambda_0^1 ),\hat{\mathbf{v}}, A) = \mathcal{L}^{0,o}_{\text{\tiny{low}}}(  \lambda_0^1, \hat{\mathbf{v}},A) +     R^r_{o,A}\left(\left(\tilde{\Pi}_{o,A}  \lambda^1_0 \right)(\hat{\mathbf{v}},\cdot)  \right),
\end{equs}
where 
\begin{equs}
\mathcal{L}^{r,o}_{\text{\tiny{low}}}(\lambda^{1}_{0},A)  &  =  \frac{\Upsilon_{\text{\scriptsize{root}}}^{\Psi_o}[\lambda_0^1]}{S_{\text{\scriptsize{root}}}(\lambda_0^1)}(\mathbf{v},0) V_{0} = \mathcal{C}[u^2,i\Delta](-iv) \bar{v} \\
 R^1_{o,A}\left(\left(\tilde{\Pi}_{o,A}  \lambda^1_0 \right)(\hat{\mathbf{v}},\cdot)  \right) & = \mathcal{C}[\mathcal{C}[\mathcal{M}_{\lbrace 1,2 \rbrace},\Delta](\mathcal{M}_{\lbrace 1 \rbrace}, \mathcal{M}_{\lbrace 1 \rbrace}) \mathcal{M}_{\lbrace 2 \rbrace}, i \Delta] (-iv,\bar{v}).   
\end{equs}
It remains to compute the local error for trees with two edges such as $ T_3 = \CI_o(\lambda_0 \mathcal{I}_{o}(\lambda_0))$. We proceed with a full derivation for this decorated trees. Two edges mean that we have to go for one extra level of recursivity. One has
\begin{equs} \label{express_1}
\mathcal{L}^{1}_{\text{\tiny{low}}}(T_3,\hat{\mathbf{v}}, A) = \mathcal{L}^{0,o}_{\text{\tiny{low}}}(  \lambda_0 \mathcal{I}_{o}(\lambda_0), \hat{\mathbf{v}},A) +     R^1_{o,A}\left(\left(\tilde{\Pi}_{o,A}^{0}  \lambda_0 \mathcal{I}_{o}(\lambda_0) \right)(\hat{\mathbf{v}},\cdot)  \right).\quad
\end{equs}
Then, for the first term, one gets
\begin{equs}
\mathcal{L}^{0,o}_{\text{\tiny{low}}}(  \lambda_0 \mathcal{I}_{o}(\lambda_0), \hat{\mathbf{v}},A) = \frac{\Upsilon_{\text{\scriptsize{root}}}^{\Psi_o}[\lambda_0\mathcal{I}_{o}(\lambda_0)]}{S_{\text{\scriptsize{root}}}(\lambda_0 \CI_{o}(\lambda_0))}(\mathbf{v},0)\, V_{0} \, \mathcal{L}^{0}_{\text{\tiny{low}}}(\CI_o(\lambda_0),\hat{\mathbf{v}}, A).
\end{equs}
Given that $ A = (H^2\cap H^1_0)(\Omega)^3 $, one has from the local error analysis of the first order scheme:
\begin{equs}
\mathcal{L}^{0}_{\text{\tiny{low}}}(\CI_o(\lambda_0),\hat{\mathbf{v}}, A) = - i v^2 \bar{v} + \Delta \left(  v^2\bar{v} \right) +   \Delta( v^2) \bar{v} -  v^2 \Delta \bar{v}. 
\end{equs} 
Therefore, one gets
\begin{equs}
\mathcal{L}^{0,o}_{\text{\tiny{low}}}(  \lambda_0 \mathcal{I}_{o}(\lambda_0), \hat{\mathbf{v}},A) = -  v^3 \bar{v}^2 - i v \bar{v} \Delta \left(  v^2\bar{v} \right) - i v  \Delta( v^2) \bar{v}^2 + i v^3 \bar{v} \Delta \bar{v}.
\end{equs}
For the second term of \eqref{express_1}, we Taylor-expand the operators that give:
\begin{equs}
R^1_{o,A}\left(\left(\tilde{\Pi}_{o,A}^{0}  \lambda_0 \mathcal{I}_{o}(\lambda_0) \right)(\hat{\mathbf{v}},\cdot)  \right) = - i \Delta \left( v^3   \bar{v}^2 \right) - i (\Delta v)   v^2 \bar{v}^2 + i (\Delta \bar{v}) v^3 \bar v.
\end{equs}
In fact for the approximation of $ \CI_o(\lambda_0\CI_o(\lambda_0)) $, one uses the classical exponential integrators by Taylor-expanding all the operators. We do not ask too much regularity due to the size of the trees which involves two integrations in time.
For the other trees $ \CI_o(T_i) $, $ i \in \lbrace 4,5,6,7,8 \rbrace $, one can proceed with the same computations and get an error of the same order.
\end{proof}

\begin{remark}[Error improvement]\label{rem:ErrImp2}
 
Note that classical approximation techniques, such as splitting or exponential integrator methods (see, e.g., \cite{R1,R2,R3})  introduce a local error of type $\mathcal{O}(\Delta^2 v, \Delta^2 V)$ and hence require solution and potential in
\begin{equs}
D(\Delta^2) =\left \{f \in (H^2\cap H^1_0)(\Omega)\,;\, \Delta f \in (H^2\cap H^1_0)(\Omega)\right\}.
\end{equs}
The local error of the  second order low regularity GP integrator \eqref{schemeGP2} on the other hand only requires the boundedness of second instead of fourth order spatial derivatives of the potential $V$ and solution~$u$, as well as only the classical boundary conditions that the trace vanishes, i.e., $H^1_0(\Omega)$.
\end{remark}

Note that the second order low regularity scheme~\eqref{schemeGP2} involves  commutators of the form $\mathcal{C}[f,i\Delta](u^n,\overline{u^n},V)$ which in general involve spatial derivatives. Indeed, using Definition~\ref{def:comm}, one has
\begin{equs} \label{com}
& \mathcal{C}[f,i\Delta](v,\bar{v},V) =  - i \Delta (f(v,\bar{v},V)) +  D_1 f(v,\bar{v},V) \cdot i \Delta v \\ &  +  D_2 f(v,\bar{v},V) \cdot i \Delta \bar{v} + D_3 f(v,\bar{v},V) \cdot i \Delta V.
\end{equs}
For practical computations we need to address the stability issues caused by the inclusion of these commutator terms in the numerical schemes. Different approaches can be made to treat this issue and guarantee the stability of the scheme. One way to overcome the stability issue lies in imposing a CFL  condition on the fully discrete scheme,  introducing a restriction on the ratio between time step size $\tau$ and spatial step size $\Delta x$.   In this paper we will not impose any CFL type condition, but stabilize the scheme a posteriori by the use of  properly chosen filter operators $\Psi$. More precisely, instead of considering the second order low regularity integrator~\eqref{schemeGP2} for the discretisation of the GP equation, we consider its stabilised counterpart
\begin{equs}\label{schemeGP2:stab}
& u^{n+1} 
 = e^{i \tau    \Delta} u^n -i \tau  \left( 
(e^{i \tau \Delta}(u^n)^2) (e^{i \tau \Delta} \varphi_1( -2 i\tau \Delta) \overline u^n) \right.\\
&\left. \quad +(e^{i \tau \Delta}u^n) (e^{i \tau \Delta}\varphi_1(- i \tau \Delta) V) \right) \\
&-i \tau^2\Psi  \CC[(e^{i\tau\Delta}\CM_{\{1\}})(e^{i\tau\Delta}(\varphi_1(-2i\tau\Delta) - \varphi_2(-2i\tau\Delta))\CM_{\{2\}}),i\Delta]((u^n)^2,\bar{u}^n) \\
&-i \tau^2\Psi \CC[(e^{i\tau\Delta}\CM_{\{1\}})(e^{i\tau\Delta}(\varphi_1(-i\tau\Delta) - \varphi_2(-i\tau\Delta))\CM_{\{2\}}),i\Delta](u^n,V) \\
&-i\tau^2\left (e^{i\tau\Delta} \Psi\CC[u^2,i\Delta](u^n)\right)\left(e^{i\tau\Delta}\varphi_2(-2i\tau\Delta)\bar{u}^n\right) \\
&- \frac {\tau^2}{2}( u^n |u^n|^4 + 3 u^n |u^n|^2V - |u^n|^2u^n\bar{V} + u^n V^2)\\
&= \ \Phi_{\text{\tiny GP2STAB}}^\tau(u^n),
\end{equs}
where we have inserted the filter operator $\Psi$ in front of all commutator terms. This filter function allows us the stabilization of the scheme. In the following section we briefly discuss the appropriate choice of filter function (which allows us to stabilize the scheme without worsening the local error structure) and  outline an example. For a general introduction to filter functions in case of ODEs we refer to \cite{HLW}.\medskip

\subsection{Filter functions, commutators and stabilisation}
\label{stabilisater}

To illustrate the appropriate choice for the filter operator $\Psi$ we place ourselves in the following general framework:
to construct a stable scheme at order $p = r+1$, we wish to stabilize a term of the form 
$$\tau^p \mathcal{C}^{p-1}[f,i\Delta](v,\bar{v}, V).$$
In case of the second order scheme for \eqref{evGP}, the commutator terms are of first order and hence involve first order derivatives (with according boundary conditions), see \eqref{com}, \eqref{com-1} and the Remark \ref{rem:domi}. This motivates the following choice we make on the filter operator:  we denote the filter operator by
$$\Psi = \psi(i\tau |\nabla|),$$
where
$$
 |\nabla| := \left( - \Delta \right)^{1/2}
$$
is self-adjoint on the domain $D(|\nabla|) = H^1_0(\Omega) $ and $\psi$ denotes a suitably chosen filter function 
which allows us to 
\begin{itemize}
\item[(a)]
stabilize the scheme such that the numerical flow is locally Lipschitz continuous (see Assumption \ref{assump:filter} below), while still
\item[(b)] allowing approximations at low regularity (see Assumption \ref{assump:filter2} below). 
\end{itemize}
Here it is important to note that the main purpose of our new framework lies in reducing the regularity assumptions. Hence, we do not want to introduce a ``bad'' (classical) error term through the stabilization, otherwise we could use directly classical Taylor series expansion techniques in the construction of our schemes. Thus, we require the filter function $\psi$ to introduce the same optimal local error $O(\tau^{p+1} |\nabla|^p(v+V))$ as is introduced by our low-regularity schemes (see Section~\ref{subsec:num-erranal}).

%

The two necessary assumptions on the filter operator $\Psi = \psi(i\tau |\nabla|)$ at arbitrary order $p$ reads as follows:

\begin{assumption}\label{assump:filter} 
Let us denote by $X$ the space endowed with norm $\Vert \cdot \Vert$ in which we carry out the error analysis. Then the filter function $\Psi$ should satisfy the following estimate
$$
\Vert \tau^{p-1} \Psi [\mathcal{C}^{p-1}[f,i\Delta](v,\bar{v}, V)]\Vert \le C ||v||^\sigma
$$
for some $\sigma= \sigma(f,p) > 0$.
\end{assumption}
 Assumption \ref{assump:filter} guarantees the stability of the scheme, such that (for the appropriate choice of $X$) we have local Lipschitz continuity of the numerical flow $ \Phi^\tau$.


\begin{assumption}\label{assump:filter2}   The filter operator $\Psi = \psi(i\tau |\nabla|)$ shall satisfy the following low regularity expansion
\begin{equation}\label{A2}
\Psi[\mathcal{C}^{p-1}[f,i\Delta](v,\bar{v}, V)] = \mathcal{C}^{p-1}[f,i\Delta](v,\bar{v}, V) + \mathcal{O}(\tau |\nabla|^p (v+V)).
\end{equation}
\end{assumption}
Assumption \ref{assump:filter2} preserves the optimal local error structure  $O(\tau^{p+1} |\nabla|^p(v+V))$ of our schemes. This is essential as the inclusion of the filter function should not require more regularity than we assumed in the construction of our low regularity framework.

An example of a possible choice of a filter function in order to stabilize the second order scheme \eqref{schemeGP2} is given in the following example.
\begin{example}[Filter functions for the GP equation]
A possible choice of filter function for the second order low regularity integrator \eqref{schemeGP2:stab} is given by
\begin{equation}\label{ffp2}
\Psi =\psi(i\tau |\nabla|) = \varphi_1(i\tau |\nabla|) = (i\tau \vert \nabla\vert)^{-1}(e^{i \tau |\nabla|}-1),
\end{equation}

where we recall that $|\nabla| = \left( - \Delta \right)^{1/2}$ is a self-adjoint operator on the $D(|\nabla|) = H^1_0(\Omega)$.
\end{example}

For a higher $p$-th order method we will encounter higher order commutators of the form $$\tau^p \mathcal{C}^{p-1}[f,i\Delta](u^n, \overline{u^n}, V)$$
in our schemes. In order to stabilize the latter we need to introduce a higher order  filter operator $\Psi_p$. For this purpose, we first note that using the Leibnitz formula, the commutator term $\mathcal{C}^{p-1}[f,i\Delta](v, \bar{v},V)$ can be written as a combination of terms involving at most $p-1$ derivatives of $v$ and $V$. 

As an example, in the case $p=2$, $\mathcal{C}[f,i\Delta](v,\bar{v}, V)$ has the explicit form \eqref{com}. Hence the role of $\Psi_p$ lies in stabilizing terms which involve the differential operator $|\nabla|^{p-1}$. 
This motivates the higher order filter operator
\begin{equation} \label{filterfcnP}
\Psi_p = \underbrace {\varphi_1(i\tau |\nabla|)... \varphi_1(i\tau |\nabla|)}_\text {$p-1$ times}.
\end{equation}

Note that at first glance another  (cheaper) choice of high order filter function would be $\tilde{\Psi}(i\tau |\nabla|) :=  \varphi_1(i\tau |\nabla|^{p-1})$. The latter, however, introduces a local error of order $\mathcal{O}(\tau|\nabla|^{2(p-1)})$ which at order $p \ge 3$ is not optimal in the sense of our low regularity approximations. Namely, this (cheaper) choice of filter function does not satisfy Assumption \ref{assump:filter2}.

\subsection{Klein-- and Sine--Gordon equations}\label{sec:sine}

As a second example we consider a problem with real non-polynomial nonlinearity: the nonlinear Klein--Gordon (KG) equation
\begin{equation}\label{kgrO}
\begin{aligned}
& \partial_{tt} z - \Delta z  + m^2 z =  g(z), \quad \quad (t,x) \in \R \times  \Omega, 
\\
& z(0) = u_{0}, \quad \partial_{t} z(0) =  u_{1}, 
\end{aligned}
\end{equation}
where for simplicity we assume non zero mass $m \neq 0$ and real-valued solutions $z(t,x)\in \R$. Nevertheless, our framework can also allow for the complex setting $z(t,x)\in \C$ and for the case $m= 0$ for wave equations. 

First, in order to apply our abstract framework \eqref{ev} to the nonlinear Klein-Gordon model \eqref{kgrO} we rewrite \eqref{kgrO} as a first-order complex system. For this purpose we define 
\[
 \cnab = \sqrt{-\Delta +m^2}
 \]
 and assume that our domain $\Omega$ (equipped with appropriate boundary conditions) is chosen such that $ \cnab $ is well defined and invertible. For instance in the case of periodic boundary conditions the operator $ \cnab $ can be expressed as the Fourier multiplier
 $$
  \left( \cnab\right)_{k\in \mathbb{Z}^d}  = \sqrt{(k_1+\ldots+k_d)^2 +m^2},
 $$
 and is a well defined invertible operator on the space $H^1(\T^d)$, given that $m \neq 0$.
While on a sufficiently smooth bounded domain $\Omega \subset \R^d$ equipped with homogeneous Dirichlet boundary conditions, using pseudo-spectral methods we have that the operator
$
\cnab: (H^2\cap H^1_0)(\Omega) \rightarrow L^2(\Omega)
$
is essentially self-adjoint; its closure is self-adjoint on the domain $$D(\cnab) = H^1_0(\Omega).$$
See \cite{ReedSimon} for a full discussion on the study of unbounded operators, spectral theorems for self-adjoint operators (pg. 263) and of the notion of closable operators (pg. 250).
Next we introduce the transformation 
  \begin{equation}
 \label{uz}
u= z - i \cnab^{-1} \partial_t z,
\end{equation}
which allows one to rewrite the Klein-Gordon equation \eqref{kgrO} as the following first order complex system
\begin{equation}\label{kgroo}
i \partial_t u = - \cnab u +  \cnab^{-1}g\big((\textstyle \frac12 (u+ \overline u)\big),
\end{equation}
and where by \eqref{uz} we have that $$z = \frac12 (u+\overline u) = \mathrm{Re}(u).$$ 

A natural space to study wave type equations of the form \eqref{kgrO} is for $(z,\partial_t z) \in H^1 \times L^2.$ In terms of $u$, from the transformation \eqref{uz}, we see that a natural space to measure the error in when studying equation \eqref{kgroo} is $H^1$.

In order to illustrate our general theory to the case of a non-polynomial nonlinearity,  we choose to study the Sine--Gordon equation. The Sine--Gordon equation corresponds to the following choice of nonlinearity 
$$g(z)= -  \sin z,$$
which by equation \eqref{kgroo} yields the first-order complex Sine--Gordon equation:
\begin{equation}\label{S-Geqn}
\partial_t u - i\cnab u =i  \cnab^{-1}\big( \sin(\textstyle \frac12 u) \cos(\textstyle \frac12 \bar{u}) + \cos(\textstyle \frac12 u)\sin(\textstyle \frac12 \bar{u}) \big).
\end{equation}

We now apply our general framework to the above Sine--Gordon equation. The above equation \eqref{S-Geqn} is of the form \eqref{ev} with
\begin{align*}
  \Lab_- & = \lbrace  0,1 \rbrace , \quad  \Lab_+ = \Lab_+^{o,0} = \Lab_+^{o,1} = \Lab_+^{\bar o,0} = \Lab_+^{\bar o,1}= \lbrace o, \bar o \rbrace,  \quad 
 V_0 =V_1 =1 \\ 
 \mathcal{L}_o &  =  i \cnab, \quad \mathcal{L}_{\bar o}  = - \mathcal{L}_{o}, \quad u_o = v, \quad u_{\bar o} = \bar v \\
 f^{0}_{o, o}(u) &  =  \sin(\textstyle \frac12 u) , \quad  f^{0}_{o, \bar o}( \bar u) =  \cos(\frac{1}{2}\bar u),  \quad f^{1}_{o, o}(u)   = \cos(\textstyle\frac{1}{2}u) , \quad  f^{1}_{o, \bar o}( \bar u) =  \sin(\frac{1}{2} \bar u), \\
 f^{0}_{\bar o, o}(u) &  =  \cos(\textstyle\frac{1}{2} u) , \quad  f^{0}_{\bar o, \bar o}( u) =  \sin(\textstyle\frac{1}{2}\bar u),  \quad f^{1}_{\bar o, o}(u)   = \sin(\textstyle\frac{1}{2} u) , \quad  f^{1}_{\bar o, \bar o}( \bar u) =  \cos(\frac{1}{2}\bar u).
\end{align*}
Using the above we have that for $ (\ell,\mathfrak{l}) \in \lbrace o, \bar o \rbrace \times \lbrace 0,1\rbrace $ the nonlinearity is given by
\begin{equs} \label{diff_Psi}
\Psi_{\ell}^{\mathfrak{l}}(\mathbf{u}_{\ell}^{\mathfrak{l}}) = i \cnab^{-1} \left(  f_{\ell,o}^{\mathfrak{l}}(u) f_{\ell, \bar o}^{\mathfrak{l}}(\bar u) \right).
\end{equs}
Before introducing the first order analysis for \eqref{S-Geqn}, we make the following comparative remark. The decorated trees which will be introduced in the following section for the first order analysis of \eqref{S-Geqn} will resemble the ones introduced for the analysis of the Gross--Pitaevskii equation (see Section \ref{sec::Gross_1}). The main difference is that the nodes will encode different nonlinearities, given that the nonlinearities associated to $ V_0 $ and $ V_1 $ for the Sine-Gordon equation are different from the ones for the Gross-Pitaevskii equation. In addition, the Sine-Gordon nonlinearity \eqref{diff_Psi} includes the linear operator $\CB^{0}_{o} =  \CB^{1}_{o}= i \cnab^{-1}  $. See Section \ref{sec::Gross_1} where for the study of the Gross-Pitaevskii equation we had the more simple case where $\CB^{0}_{o} =  \CB^{1}_{o}= \id$.
%
%
\subsubsection*{First order Duhamel integrator for Sine--Gordon}
\label{sec::Sine} 
We seek to provide a first order low-regularity approximation to the following oscillatory integral:
\begin{equs}
& w^{0}_o(\hat{{\bf v}},\tau)  =  e^{i\tau \cnab} v +i \cnab^{-1}\int_{0}^{\tau} 
e^{i(\tau-\xi)\cnab}\\
&  \left(  \right. ( e^{i\xi \cnab}\sin(\textstyle \frac12 v))(e^{-i\xi\cnab}\cos(\textstyle \frac12 \bar v)) \left.+ (e^{i\xi \cnab} \cos(\textstyle \frac12 v))(e^{-i\xi\cnab}\sin(\textstyle \frac12 \bar v))\right)d\xi.
\end{equs}
The first order low regularity integrator for the Sine--Gordon equation is given in the following corollary.
\begin{corollary}\label{cor:kg1}
At first order our general low regularity scheme  \eqref{genscheme}  for the Sine-Gordon equation  \eqref{S-Geqn}   takes the form
\begin{equs}\label{schemeSG1}
u^{n+1}  & =  e^{i \tau    \cnab} u^n +  i \tau     \cnab^{-1}   \left[ \right.
\\ &  \left( e^{i \tau \cnab} \sin(\textstyle \frac12 u^n) \right) \left(  e^{i \tau \cnab}
 \left( \varphi_1\big( -2 i\tau \cnab \big)   \cos(\textstyle  \frac12 \bar u^n) \right) \right) \\ & \left. +  \left( e^{i \tau \cnab} \cos(\textstyle  \frac12 u^n) \right) \left( e^{i \tau \cnab} \left( \varphi_1\big( -2 i\tau \cnab \big)  \sin(\textstyle  \frac12 \bar u^n)  \right) \right) \right]
\end{equs}
with local error 
$$\CO\left(\tau^2 \mathcal{C}[\cnab^{-1} \mathcal{M}_{\lbrace 1,2 \rbrace},i \cnab](u, \bar{u})\right)$$
In  case of more regular solutions, 
we have the following simplified first-order scheme,
\begin{equs}\label{schemeSG2}
u^{n+1}  & =  e^{i \tau    \cnab} u^n +  i \tau \cnab^{-1}   \Big(  \sin(\textstyle \frac12 u^n)  
\cos(\textstyle  \frac12 \bar u^n) \Big. \\ 
&\Big.+  \cos(\textstyle  \frac12 u^n)  \sin(\textstyle  \frac12 \bar u^n) \Big)
\end{equs}
with a local error of order $ \mathcal{O}\Big(\tau^2 \cnab^{-1}\left( (\cnab u)\bar{u}\right)\Big)$.
\end{corollary}

\begin{remark}[Remark on the imposed regularity]
In general the local error of type
$$
\CO\left(\tau^2 \mathcal{C}[\cnab^{-1} \mathcal{M}_{\lbrace 1,2 \rbrace},i \cnab]( u , \overline{u})\right)
$$
asks for less regularity assumptions than the local error $ \mathcal{O}\Big(\tau^2 \cnab^{-1}\left( (\cnab u)\bar{u}\right)\Big)$, due to the structure of the commutator (see Definition \ref{def:comm}). Indeed, we portray this in the following example, where we measure the error in the space $H^1(\T^d)$. 

First, using the commutator estimate established in \cite[Lemma 27]{FS}, we have
\begin{equs}
||\mathcal{C}[\cnab^{-1} \mathcal{M}_{\lbrace 1,2 \rbrace},i \cnab]( u , \overline{u})||_{H^1} 
& \lesssim ||\mathcal{C}[ \mathcal{M}_{\lbrace 1,2 \rbrace},i \cnab]( u , \overline{u})||_{L^2} \\
& \lesssim ||u ||^2_{H^{\frac{1}{2} + \frac{d}{4}}}.
\end{equs}
Namely we have that the commutator term asks for solutions in $H^{\frac{1}{2} + \frac{d}{4}}(\T^d)$.

Secondly, using the classical bilinear estimate: 
$$ ||uv||_{L^2(\Omega)} \lesssim ||u||_{L^2(\Omega)} ||v||_{H^{\frac{d}{2} + \epsilon}(\Omega)}, \ \epsilon > 0,$$
which follows by the Sobolev injection $H^{\frac{d}{2} + \epsilon} \hookrightarrow L^\infty$, we have,
\begin{equs}
||\cnab^{-1} ((\cnab u)\bar{u})||_{H^1(\Omega)} &\lesssim ||(\cnab u)\bar{u} ||_{L^2(\Omega)} \\
&\lesssim ||\cnab u ||_{L^2(\Omega)}||u||_{H^{\frac{d}{2} + \epsilon}(\Omega)}.
\end{equs}
It follows from the above that the {\it classical} local error term asks for solutions in $(H^1\cap H^{\frac{d}{2} + \epsilon})(\T^d)$ when working on the torus, while when working on a bounded domain equipped with Dirichlet boundary conditions it asks for solutions in $(H^1_0\cap H^{\frac{d}{2} + \epsilon})(\Omega)$. Hence, the commutator-term asks for less regularity than the classical local error term. For example, when $d=1$, the low-regularity scheme asks for $H^{\frac{3}{4}}(\T)$-solutions, whereas the classical scheme asks for $u(t) \in D(\cnab ) = H^1(\T)$. 


\end{remark}

\begin{proof}

The trees we are considering are $\CT^1 =  \lbrace  T_0, T_1 \rbrace$, where following the framework given in \eqref{nodeDef} we have that,


\begin{equs}
\CB_{o} \Bigg(\prod_{\mathcal{o} \in \Lab_+^{0,o}} e^{\xi \mathcal{L}_{\mathcal{o}}} f^{0}_{o,\mathcal{o}}(v_{\mathcal{o}}) \Bigg) V_0 
&= i\cnab^{-1}\Bigg(\left( e^{i\xi \cnab}u_1 \right) \, \left( e^{-i\xi \cnab} u_2 \right)\Bigg) \\
& \equiv 
 \begin{tikzpicture}[scale=0.2,baseline=-5]
\coordinate (root) at (0,-0.5);
\node[xi] (rootnode) at (root) {};
\end{tikzpicture} = \lambda_0 = T_0,
\\
\CB_{o} \Bigg(\prod_{\mathcal{o} \in \Lab_+^{1,o}} e^{\xi \mathcal{L}_{\mathcal{o}}} f^{1}_{o,\mathcal{o}}(v_{\mathcal{o}}) \Bigg) V_1 
&= i\cnab^{-1}\Bigg(\left( e^{i\xi \cnab}u_3 \right) \, \left( e^{-i\xi \cnab} u_4 \right)\Bigg)  \\&\equiv 
 \begin{tikzpicture}[scale=0.2,baseline=-5]
\coordinate (root) at (0,-0.5);
\node[xi,blue] (rootnode) at (root) {};
\end{tikzpicture} = \lambda_1 = T_1,
\end{equs}
where $\mathbf{v} = (v,\bar{v})$ and $u_1 = \sin(\textstyle \frac12 v)$, $u_2 = \cos(\textstyle \frac 12 \overline{v})$, $u_3 = \cos(\textstyle \frac12 v)$, $u_4 = \sin(\textstyle \frac 12 \overline{v})$.
From equation \eqref{genscheme} it then follows that the first-order scheme is of the form,
\begin{equs}\label{SG1}
u^{0}( \hat{\mathbf{v}},\tau) = e^{ i\tau\Delta }v + \sum_{T \in \CT^{1}} \Pi_A^0 \left( \CI_{o}(T) \right)(\hat{\mathbf{v}},\tau),
\end{equs}
where $ \mathbf{v} = (v,\bar{v})$. We have, for $i=0,1$, that
\begin{equs}
\Pi^0_{A} \CI_{o}(\lambda_{i})(\hat{\mathbf{v}},\tau) = \mathcal{K}_{o,A}^0 ( \tilde{\Pi}_{o,A} \CD_{-1}(\lambda_i)(\hat{\mathbf{v}},\cdot))(\tau) 
= \mathcal{K}_{o,A}^0\left(\frac{\Upsilon_{\text{\scriptsize{root}}}^{\Psi_{o}}[\lambda_i]}{S_{\text{\scriptsize{root}}}(\lambda_i)}(\mathbf{v},\xi) \right)(\tau)  
\end{equs}
where one has,
\begin{equs}
&S_{\text{\scriptsize{root}}}(\lambda_0) =  S_{\text{\scriptsize{root}}}(\lambda_1) = 1, \\ 
&\Upsilon_{\text{\scriptsize{root}}}^{\Psi_{o}}[\lambda_0](\mathbf{v},\xi) = i\cnab^{-1}\Bigg(\left( e^{i\xi \cnab}\sin(\textstyle \frac12 v) \right) \, \left( e^{-i\xi \cnab} \cos(\textstyle \frac 12 \overline{v}) \right)\Bigg), \quad \\
& \Upsilon_{\text{\scriptsize{root}}}^{\Psi_{o}}[\lambda_1](\mathbf{v},\xi) =  i\cnab^{-1}\Bigg(\left( e^{i\xi \cnab}\cos(\textstyle \frac12 v) \right) \, \left( e^{-i\xi \cnab} \sin(\textstyle \frac 12 \overline{v}) \right)\Bigg).
\end{equs}
 1. When $A \supset (H^1\cap H^{\frac{d}{2} + \epsilon})(\T^d)^2$ or $A \supset (H^1_0\cap H^{\frac{d}{2} + \epsilon})(\Omega)^2$. \\
%
Given that $r, \ell = 0$, and $ v \not\in D(G)$, we do not have enough regularity to simply apply the Taylor-expansion \eqref{def_fullTalor}. Instead we perform a low-regularity based approximation of the integrals and apply the first point $ (i) $ of Definition~\ref{def_CK}.
From definition \eqref{def_CK} it then follows that 
$$
\CL_1 = \CL_{o} = i \cnab,  \ \CL_2 = -i\cnab, \ \CL_{\tiny{\text{dom}}} =-2i\cnab
$$
with $\CB_{o} = \cnab^{-1}$, and hence we have,
\begin{equs}
&\mathcal{K}_{o,A}^0\left( \Upsilon_{\text{\scriptsize{root}}}^{\Psi_{o}}[\lambda_0](\mathbf{v},\cdot) \right)(\tau)  = i \tau \cnab^{-1} \left( e^{i \tau \cnab} u_1 \right) \left(  e^{i \tau \cnab} \left( \varphi_1\big( -2 i\tau \cnab \big)   u_2 \right) \right)\\
&\mathcal{K}_{o,A}^0\left( \Upsilon_{\text{\scriptsize{root}}}^{\Psi_{o}}[\lambda_1](\mathbf{v},\cdot) \right)(\tau)  = i \tau \cnab^{-1} \left( e^{i \tau \cnab} u_3 \right) \left(  e^{i \tau \cnab} \left( \varphi_1\big( -2 i\tau \cnab \big) u_4 \right) \right).
\end{equs}
Plugging the above in \eqref{SG1} yields the first order low-regularity scheme \eqref{schemeSG1} for the Sine-Gordon equation.

\noindent 2. When $A = (H^1\cap H^{\frac{d}{2} + \epsilon})(\T^d)^2$ or $A = (H^1_0\cap H^{\frac{d}{2} + \epsilon})(\Omega)^2$\\
In this case we apply the Taylor expansion \eqref{def_fullTalor} with $r=\ell = 0$, to obtain:
\begin{equs}
\CK_{o,A}^0 & \left((\Upsilon_{\text{\scriptsize{root}}}^{\Psi_{o}}[\lambda_0] + \Upsilon_{\text{\scriptsize{root}}}^{\Psi_{o}}[\lambda_1])(\mathbf{v},\cdot) \right)(\tau) \\ & = i\cnab^{-1} \left(  \int_0^\tau  u_1 u_2 d\xi + \int_0^\tau  u_3 u_4 d\xi \right)\\
&= i\tau\cnab^{-1} (\sin(\textstyle \frac12 v) \cos(\textstyle \frac 12 \overline{v})+\cos(\textstyle \frac12 v)\sin(\textstyle \frac 12 \overline{v}),
\end{equs}
which thanks to \eqref{SG1} yields the first order scheme \eqref{schemeSG2}.
\\
\noindent \textbf{Local error analysis}: We follow the same steps as for the proof of Corollary~\ref{cor_schemeGP1}.
 Using the recursive formula in Definition~\ref{def:Llow}, one gets for $ T \in \lbrace  \lambda_0, \lambda_1 \rbrace $
\begin{equs}
\mathcal{L}^{0}_{\text{\tiny{low}}}(\CI_{o}( T ),\hat{\mathbf{v}},A)  =    \mathcal{L}^{-1,o}_{\text{\tiny{low}}}(  T, \hat{\mathbf{v}},A) +     R^0_{o,A}\left(\left(\tilde{\Pi}_{o,A}  T \right)(\hat{\mathbf{v}},\cdot)  \right) .
\end{equs}
Then by \eqref{local_error_1}, we get for $ \mathfrak{l}\in \lbrace 0,1 \rbrace $
\begin{equs}
\mathcal{L}^{-1,o}_{\text{\tiny{low}}}(\lambda_{\mathfrak{l}},\hat{\mathbf{v}},A)   =  \frac{\Upsilon_{\text{\scriptsize{root}}}^{\Psi_o}[T]}{{S_{\text{\scriptsize{root}}}(T)}}(\mathbf{v},0) V_{\mathfrak{l}}   
\end{equs}
which gives
\begin{equs}
\mathcal{L}^{-1,o}_{\text{\tiny{low}}}(\lambda_{0},\hat{\mathbf{v}},A) & = i\cnab^{-1}\Bigg(\left( e^{i\xi \cnab}\sin(\textstyle \frac12 v) \right) \, \left( e^{-i\xi \cnab} \cos(\textstyle \frac 12 \overline{v}) \right)\Bigg), \\  \mathcal{L}^{-1,o}_{\text{\tiny{low}}}(\lambda_{1},\hat{\mathbf{v}},A) & = i\cnab^{-1}\Bigg(\left( e^{i\xi \cnab}\cos(\textstyle \frac12 v) \right) \, \left( e^{-i\xi \cnab} \sin(\textstyle \frac 12 \overline{v}) \right)\Bigg).
\end{equs}
It remains to compute $ R^0_{o,A}\left(\left(\Pi_A  T \right)  \right)$ whose values will depend on $ A $. For $A \supset (H^1\cap H^{\frac{d}{2} + \epsilon})(\T^d)^2$ or $A \supset (H^1_0\cap H^{\frac{d}{2} + \epsilon})(\Omega)^2$
one gets
\begin{equs}
 R_{o,A}^0((\tilde{\Pi}_{o,A}\lambda_0)(\hat{\mathbf{v}},\cdot)) & = \mathcal{C}[ i\cnab^{-1} \mathcal{M}_{\lbrace 1,2 \rbrace},i \cnab](\sin(\textstyle \frac12 v) ,\cos(\textstyle \frac 12 \overline{v}))
 \\ 
 R_{o,A}^0((\tilde{\Pi}_{o,A}\lambda_1)(\hat{\mathbf{v}},\cdot)) & = \mathcal{C}[ i\cnab^{-1} \mathcal{M}_{\lbrace 1,2 \rbrace},i \cnab](\cos(\textstyle \frac12 v) ,\sin(\textstyle \frac 12 \overline{v})).
\end{equs} 
 For $A = (H^1\cap H^{\frac{d}{2} + \epsilon})(\T^d)^2$ or $A = (H^1_0\cap H^{\frac{d}{2} + \epsilon})(\Omega)^2$, one Taylor-expands  the operators which give:
\begin{equs}
& R_{o,A}^0 ((\tilde{\Pi}_{o,A}\lambda_0)(\hat{\mathbf{v}},\cdot))  = -\sin(\textstyle \frac12 v) \cos(\textstyle \frac 12 \overline{v})
- \cnab^{-1}\Bigg(\left(  \cnab \sin(\textstyle \frac12 v) \right) \, \left( \cos(\textstyle \frac 12 \overline{v}) \right)\Bigg) \\ &   
\quad + \cnab^{-1}\Bigg(\sin(\textstyle \frac12 v) \, \left( \cnab \cos(\textstyle \frac 12 \overline{v}) \right)\Bigg)
 \\ 
 & R_{o,A}^0((\tilde{\Pi}_{o,A}\lambda_1)(\hat{\mathbf{v}},\cdot))  = -\cos(\textstyle \frac12 v) \sin(\textstyle \frac 12 \overline{v})
- \cnab^{-1}\Bigg(\left(  \cnab \cos(\textstyle \frac12 v) \right) \, \left( \sin(\textstyle \frac 12 \overline{v}) \right)\Bigg) 
\\ & \quad + \cnab^{-1}\Bigg(\cos(\textstyle \frac12 v) \, \left( \cnab \sin(\textstyle \frac 12 \overline{v}) \right)\Bigg)
\end{equs} 
which concludes the proof.
\end{proof}
\begin{remark}
Similarly to Corollary \ref{cor:kg1} our general framework \eqref{scheme} also allows for low regularity approximations at higher order for Klein--Gordon equations.
\end{remark}

\appendix
\endappendix

\end{document}